\documentclass[aap]{imsart}
\usepackage{amsmath,amsfonts,amssymb,amsthm,graphicx}
\usepackage{dsfont}  
\usepackage{enumerate}
\usepackage{color}
\usepackage[numbers,sort&compress]{natbib}
\usepackage{multirow}\usepackage{graphicx}
\usepackage{amsfonts}
\usepackage{amsmath}
\usepackage{amssymb,mathtools}

\DeclarePairedDelimiter\floor{\lfloor}{\rfloor}
\DeclareMathOperator*{\argmin}{arg\,min}
\usepackage{color}
\usepackage{dsfont}

\usepackage{tikz-cd}
\usepackage{pgfplots}
\usepackage{subcaption}
\usepgfplotslibrary{fillbetween}
\usetikzlibrary{patterns}
\usepackage{amssymb}
\usepackage{graphicx}
\usepackage{amsmath}
\usepackage{bm}
\usepackage{tikz}
\usepackage{booktabs}

\usepackage{algorithm}
\usepackage{algpseudocode}

\usepackage{caption}

\usetikzlibrary{decorations.pathreplacing}

\usepackage{tkz-graph}
\usepackage{pgfplots}
\usetikzlibrary{calc,patterns}

\usepackage{subcaption}
\usepackage{url}
\usepackage{fancyhdr}
\usepackage{indentfirst}

\usepackage{enumerate}
\usepackage{dsfont}
\usepackage[british]{babel}
\usepackage{siunitx}

\usepackage[colorlinks,citecolor=blue,urlcolor=blue]{hyperref}
\usepackage{amsthm}
\usepackage{color}
\usepackage{natbib}

\usepackage{comment}
\usepackage{tikz-cd}
\usepackage{graphicx}
\usepackage{amsfonts}
\usepackage{amsmath}
\usepackage{amssymb}
\usepackage{url}
\usepackage{bbm}
\usepackage{fancyhdr}
\usepackage{indentfirst}
\usepackage{enumerate}
\usepackage{dsfont}
\usepackage[british]{babel}
\usepackage{cleveref}
\usepackage{amsthm}
\usepackage{color}
\renewcommand{\cite}{\citet}
\usepackage{comment}

\usepackage{algorithm}

\renewcommand{\d}{\,\mathrm{d}}

\usepackage{xcolor}

\newcommand{\E}{\mathbb{E}}    
\newcommand{\R}{\mathbb{R}}    
\newcommand{\N}{\mathbb{N}}    
\newcommand{\Z}{\mathbb{Z}}

\newcommand{\F}{\mathcal{F}}   
\newcommand{\sF}{\mathscr{F}}

\startlocaldefs

\theoremstyle{plain}
\newtheorem{theorem}{Theorem}
\newtheorem{corollary}{Corollary}
\newtheorem{lemma}{Lemma}
\newtheorem{proposition}{Proposition}

\theoremstyle{definition}
\newtheorem{definition}{Definition}
\newtheorem{example}{Example}

\newtheorem{assumption}{Assumption}

\theoremstyle{remark}
\newtheorem{remark}{Remark}

\endlocaldefs

\usepackage{tikz}
\usetikzlibrary{arrows.meta,positioning}

\usepackage{tcolorbox}

\newcommand{\ee}{\varepsilon}

\newcommand{\n}[1]{\left\lVert#1\right\rVert}
\newcommand{\nn}[1]{\left|#1\right|}
\newcommand{\A}{\mathcal{A}}
\newcommand{\B}{\mathcal{B}}

\newcommand{\Q}{\mathbb{Q}}

\newcommand{\mpd}{\mathrm{MPD}}

\def\d{\mathrm{d}}

\def\laweq{\buildrel \mathrm{law} \over =}

\def\lawis{\buildrel \mathrm{law} \over \sim}

\newcommand{\bone}{ {\mathbbm{1}} }
\renewcommand{\S}{\mathcal{S}}
\renewcommand{\P}{\mathbb{P}}

\usepackage{mathrsfs}

\renewcommand{\A}{\mathcal{A}}

\renewcommand{\ge}{\geqslant}
\renewcommand{\le}{\leqslant}
\renewcommand{\geq}{\geqslant}
\renewcommand{\leq}{\leqslant}
\renewcommand{\epsilon}{\varepsilon}

\renewcommand{\cdots}{\dots}

\newcommand{\ddd}{\overset{\d}{\Rightarrow}}
\pgfplotsset{compat=1.18}

\begin{document}
\begin{frontmatter}
\title{Empirical martingale projections via the adapted Wasserstein distance}

\runtitle{Empirical martingale projections via the adapted Wasserstein distance}

\begin{aug}

\author[A]{\fnms{Jose}~\snm{Blanchet}\ead[label=e1]{jose.blanchet@stanford.edu}\orcid{0000-0000-0000-0000}},
\author[B]{\fnms{Johannes}~\snm{Wiesel}\ead[label=e2]{wiesel@cmu.edu}\orcid{0000-0000-0000-0000}},
\author[A]{\fnms{Erica}~\snm{Zhang}\ead[label=e3]{yz4232@stanford.edu}},
\and
\author[C]{\fnms{Zhenyuan}~\snm{Zhang}\ead[label=e4]{zzy@stanford.edu}}
\address[A]{Department of Management Science and Engineering, Stanford University\printead[presep={,\ }]{e1,e3}}

\address[B]{Department of Mathematics, Carnegie Mellon University\printead[presep={,\ }]{e2}}
\address[C]{Department of Mathematics, Stanford University\printead[presep={,\ }]{e4}}
\end{aug}

\begin{abstract}
Given a collection of multidimensional pairs $\{(X_i,Y_i)\}_{1 \leq i\leq n}$, we study the problem of projecting the associated suitably smoothed empirical measure onto the space of martingale couplings (i.e., distributions satisfying $\E[Y|X]=X$) using the adapted Wasserstein distance. We call the resulting distance the \textit{smoothed empirical martingale projection distance} (SE-MPD), for which we obtain an explicit characterization. We also show that the space of martingale couplings remains invariant under the smoothing operation. We study the asymptotic limit of the SE-MPD, which converges at a parametric rate as the sample size increases, if the pairs are either i.i.d.~or satisfy appropriate mixing assumptions. Additional finite-sample results are also investigated. Using these results, we introduce a novel consistent martingale coupling hypothesis test, which we apply to test the existence of arbitrage opportunities in recently introduced neural network-based generative models for asset pricing calibration.
\end{abstract}

\begin{keyword}[class=MSC]
\kwd[Primary ]{49Q22}
\kwd[; secondary ]{62G10}
\kwd{60G42}
\end{keyword}

\begin{keyword}
\kwd{Martingale projection} 
\kwd{adapted Wasserstein distance}
\kwd{martingale coupling hypothesis test}
\end{keyword}

\end{frontmatter}

\tableofcontents


\section{Introduction}
Consider a collection $\{(X_i,Y_i)\}_{1\leq i\leq n}$ of random pairs with values in $\R^d\times\R^d,\,d\geq 1$. We denote the associated empirical measure of this sample by $\P_n$. The fundamental goal of this paper is to study the following question:
\begin{tcolorbox}
\begin{center}
How far is $\P_n$ from the set of laws $\Q\lawis (X,Y)$ that satisfy the martingale condition $\E[Y|X]=X$?
\end{center}
\end{tcolorbox}
 It is natural to formulate this question as a projection problem. In order to do this we first need to overcome a series of modeling challenges, which we describe in the paragraphs below. Once these are addressed, we give a precise formulation of the projection problem and study it rigorously, including an asymptotic analysis as $n$ increases. Both the modeling and the technical methodology constitute the first portion of the main contributions of this paper. The second portion is driven by applications to statistical learning. In fact, the above question is of fundamental importance from an applied standpoint, because it forms the basis of a new consistent statistical test for martingale pairs. We apply this test in order to verify the no-arbitrage condition in asset pricing models and to test a (specific) Markov kernel.

The first modeling issue that arises is the choice of a projection distance. A natural choice to measure the distance between two distributions is the Wasserstein distance, which is given by
\begin{align}\label{eq:wass}
\mathcal{W}_\gamma(\P, \Q)^\gamma =\inf\, \E[|X-X'|_2^\gamma +|Y-Y'|_2^\gamma]
\end{align}
in our context.
Here $\gamma\ge 1$, $\P,\Q$ are two probability measures on $\R^d\times \R^d$, $|\cdot|_2$ is the Euclidean norm on $\R^d$ and the infimum is taken over joint distributions of $(X,Y,X',Y')$ with $(X,Y)\lawis\P$ and $(X',Y')\lawis\Q.$
The Wasserstein distance has gained popularity in recent years because of its versatility in a wide range of machine learning tasks; we refer to the monograph \citep{peyre2019computational} and the references therein for an overview. Such versatility may be aided by the fact that the Wasserstein distance embeds $\R^d\times \R^d$ and metrizes the weak topology, see e.g., \citep{villani2009optimal}. In order to focus the discussion of our contributions in the introduction on the main conceptual challenges, we implicitly assume that $\gamma=1$ and do not delve into the integrability assumptions imposed. We will be precise about these important considerations in the statement of our results and the discussion section at the end of this paper.

A key problem that arises when considering the Wasserstein distance is that the conditional expectation is not a continuous function in the weak topology. That is the case even in the setting of simple examples going back to \citep{veraguas2020fundamental}, see Appendix \ref{subsec:deferred_mpd}. For instance, if $X=0$ and $\P(Y=1)=\P(Y=-1)=1/2$, then $(X,Y)$ forms a martingale pair under $\P$ since $\E[Y|X]=X=0$ under $\P$. However, if $\Q(X=\epsilon,Y=1)=\Q(X=-\epsilon,Y=-1)=1/2$, then $|\E[Y|X]|=1$ under $\Q$ even though the Wasserstein distance between $\Q$ and $\P$ is less than $\epsilon$. In particular, it is possible to approximate models that describe martingale pairs by a sequence of probability measures for which the martingale property fails to hold by a strictly positive gap uniformly along the sequence. For this reason, we consider an enhancement to the Wasserstein distance which addresses these types of issues.

This enhancement is called the \emph{adapted Wasserstein distance} $\mathcal{AW}_\gamma$ and has been studied precisely to deal with situations of this type. Instead of considering all joint distributions of $(X,Y,X',Y')$ preserving the marginal laws $\P,\Q$ in \eqref{eq:wass}, $\mathcal{AW}_\gamma$ only considers those, for which additionally the probability distribution $\text{Law}(Y, Y'|X,X')$ on $\R^d\times \R^d$ has marginals $\text{Law}(Y|X)$ and $\text{Law}(Y'|X')$; joint distributions of this type are called \emph{adapted} or \emph{bi-causal} (see Definition \ref{def:bi-causal} below). We refer to \citep{backhoff2020adapted} for a well-written introduction and summary of contributions to the theory of adapted distances, with historical references and comparisons; an incomplete list is given by \citep{aldous1981weak,hoover1984adapted,  ruschendorf1985wasserstein, hellwig1996sequential,bion2019wasserstein, lassalle2013causal} and \citep{pflug2010version, pflug2012distance, backhoff2020adapted, acciaio2020causal, backhoff2020all, veraguas2020fundamental}.
In particular, under the adapted Wasserstein distance, martingale pairs can only be approximated by probability measures that are close to martingale couplings themselves. As a consequence, the discontinuity issue brought up earlier cannot occur. In fact, the adapted Wasserstein distance induces the coarsest topology that addresses this discontinuity, see \citep{backhoff2020all}. This motivates considering the projection distance to the space of martingale pairs using the adapted distance. However, computing this distance is a highly non-trivial task because the bi-causal constraints mentioned above are defined in terms of infinitely many conditional distribution constraints, and the space of martingale pairs is also defined in terms of infinitely many constraints. One of the main contributions of this work is to provide a closed-form expression for this projection distance in great generality.

The next modeling issue that we consider is the fact that the empirical measure $\P_n$ will typically appear far from being a martingale pair, simply because it encodes an empirical sample -- note that $Y$ given $X$ under $\P_n$ is almost surely deterministic if the samples are drawn i.i.d.~from a continuous distribution $\P_0$. This issue, as we shall show, is resolved by introducing a smoothing technique. In particular, we consider the law of $(X+\xi,Y+\xi)$, where $\xi$ has a suitable density and is independent of $(X,Y)$. We identify a family of densities for $\xi$ that does not change the nature of the problem. Precisely, we show a result of independent interest: if the characteristic function (i.e.,~the Fourier transform) of $\xi$ has no zeros in $\R^d$ (e.g.,~if $\xi$ is standard Gaussian) and all random variables involved have finite variance, then $(X,Y)$ forms a martingale pair if and only if $(X+\xi,Y+\xi)$ forms a martingale pair. Therefore, by projecting a smoothed version of $\P_n$ using $\mathcal{AW}_\gamma$,  we do not fundamentally change the nature of the estimation task.

Once these modeling issues have been addressed, we turn to the study of a precise mathematical formulation of our problem based on the \emph{smoothed empirical martingale projection distance} (SE-MPD): we minimize the adapted Wasserstein distance between the smoothed empirical measure and any martingale pair. We then answer the following technical questions:

\begin{itemize}
\item What is the rate of convergence of the SE-MPD if the pairs $\{(X_i,Y_i)\}_{ 1\le i \le n}$ are i.i.d.~samples from a distribution $\P_0$ or satisfy some mixing conditions? Does it converge at a parametric rate?
\item Can we compute the asymptotic statistics?
\item Can we use this projection approach to develop a hypothesis test for martingale pairs? How can we study the power of this projection test?
\end{itemize}

In this paper, we provide affirmative answers to all of these questions under suitable integrability conditions. In particular, in Theorem \ref{thm:limitd simple} (and Theorem \ref{thm:mixingconvergence} for the mixing case), we show that the SE-MPD converges at the parametric rate $O(n^{-1/2})$. Moreover, we characterize the asymptotic limiting distribution in terms of the integral of a powered norm of an $\R^d$-valued Gaussian random field. 

As expected in results that involve kernel smoothing, the asymptotic distribution of the SE-MPD depends on the choice of the kernel's bandwidth $\sigma$. As we noticed earlier, the empirical distribution $\P_n$ is far from being a martingale pair if $\sigma=0$, so it is interesting to ponder the role of the bandwidth parameter. In this paper, we give at least one insight into this issue and consider the asymptotic distribution of the SE-MPD for the case $\sigma \to \infty$. Intuitively, one expects that the distribution degenerates to zero, as the smoothing procedure adds a ``big (constant) martingale'' $(\xi, \xi)$ to the pair $(X,Y)$. Somewhat surprisingly, we show that this is not the case if $\P_0(X=Y)<1$. In other words, the smoothing effect does not seem to artificially hide that $\P_n$ appears to be within $O(n^{-1/2})$-distance from the space of martingale pairs. 
On the contrary, it turns out that the SE-MPD offers a natural way to characterize the non-martingality of a law $\P_0$. If $\xi$ is Gaussian for example, we can characterize the martingale property in terms of polynomial test functions; that is, $\E[(\E[Y|X] - X)X^k]=0$ for all $k\in \N$. If this expectation does not vanish for $k$ but it vanishes for $j<k$, this informs the specific choice $\sigma =o(n^{1/(2k)})$ of the bandwidth parameter, for which the SE-MPD blows up as $n$ increases. Conversely, if only $\P_n$ is observed, this phenomenon suggests a natural way for choosing $\sigma$ in order to maximize the power of the test: we select the bandwidth that maximizes the SE-MPD, see Section \ref{sec:implementation}. A more nuanced question, of course, involves the role of $\sigma$ and even the choice of the smoothing kernel for a fixed sample size $n$. While these are interesting questions and we discuss initial results in this direction in Section \ref{subsec:general_case}, we leave a complete investigation on these issues for future research.

Our asymptotic statistics of the SE-MPD can be used for non-parametric hypothesis testing of the martingale pair property. This property is related to (although different from) martingale testing, where one often considers a sequence of martingales. We refer to Section \ref{subsec:choice_of_sig} for a more detailed discussion of this issue. For now let us simply note, that there are various methodologies and approaches to test the martingale property in different settings. For instance, \citep{phillips2014testing} develops a consistent martingale test for a one-dimensional martingale difference sequence. The test in \citep{phillips2014testing} can be applied to test martingale pairs, but it is not consistent in multiple dimensions (namely, it may be possible to not reject a false null hypothesis of martingale pairs as sample size increases). The work \citep{chang2022testing} also develops a martingale difference test for high dimensional martingales which can be applied to martingale pairs as well, but it is also not a consistent test. 
On the contrary, the test that we propose is consistent for martingale pairs under assumptions complementary to \citep{phillips2014testing} and \citep{chang2022testing}. These assumptions (e.g.,~i.i.d.~or stationarity) are motivated by applications described in our empirical Section \ref{sec:experiment} in connection with, for example, policy evaluation in reinforcement learning, testing a Markov kernel, and testing the no-arbitrage hypothesis in generative models for financial markets. In Section \ref{sec:experiment} we also perform a power analysis of our proposed test and carry out extensive numerical experiments to confirm our findings empirically. 

We conclude this section with a short literature review.
Convergence rates for $\mathcal{W}_\gamma(\P_n, \P_0)$ under various assumptions on the sample have been extensively studied in the last years, see e.g., \citep{fournier2015rate, weed2019sharp} and the references therein. The bottom line is that the Wasserstein distance exhibits the \emph{curse of dimensionality}, i.e., typically $\mathcal{W}_\gamma(\P_n, \P_0)^\gamma \approx n^{-1/d}$ for high dimensions. Similar results were established in \citep{backhoff2022estimating,acciaio2022convergence, glanzer2018} for the adapted Wasserstein distance. In consequence, a direct bound for the empirical MPD using results of \citep{backhoff2022estimating,acciaio2022convergence, glanzer2018} could be obtained. This approach uses the triangle inequality and the fact that the bounds are relatively insensitive to $\P_0$ under suitable regularity. Further, this only yield rates of order $O(n^{-1/(2d)})$---showing that our $O(n^{-1/2})$-rates are a big improvement of currently known, directly applicable, techniques.

The idea of projecting the empirical measure $\P_n$ onto a linear manifold using Wasserstein geometry has been explored in various settings in the literature. In our case, the manifold is defined by the martingale constraint, which in particular consists of infinitely many linear constraints. The work of \citep{tameling2019empirical} considers the case in which $\P_0$ is countably supported and the linear manifold has finitely many constraints. Independently, motivated by problems in distributionally robust optimization and optimal regularization in a class of machine learning estimators (such as square root Lasso among others), \citep{blanchet2019robust} investigates generally supported $\P_0$ (under suitable moment constraints) and finitely many linear constraints. The paper \citep{si2021testing} considers the use of optimal transport projections in the Wasserstein geometry for testing algorithmic fairness; this is an interesting application setting that may benefit from the analysis that we provide in this paper. 

We emphasize that in all of these settings, the linear manifold onto which one projects is defined by finitely many constraints and involves the Wasserstein distance directly. In contrast, our projection problem involves a continuum of constraints both due to the martingale property and the bi-causal restrictions implied in the definition of the adapted Wasserstein distance. The only exception to the finitely many constraints setting is the work of \citep{si2020quantifying}, which studies a class of infinitely many linear constraints (again in the standard Wasserstein setting without causal constraints). However, this reference assumes that the support of the underlying distributions is compact, and it does not obtain the exact asymptotic distribution of the projection statistics.

Lastly, let us mention that $\mathcal{AW}_\gamma$-projections onto the set of martingale measures with fixed marginals are by now classical tools for the so-called martingale optimal transport (MOT) problems, i.e., optimal transport problems with a martingale constraint \emph{and} marginal constraints, see \citep{beiglbock2013model,galichon2014stochastic, beiglbock2016problem}. In particular, the series of works \citep{guo2019computational, backhoff2022stability,wiesel2023continuity, beiglbock2022approximation, beiglbock2023stability, jourdain2023extension} uses $\mathcal{AW}_\gamma$-projection arguments to show stability of the MOT problem for $d=1$. Probably most related to our closed-form expression for the MPD is \citep[Proposition 2.4]{wiesel2023continuity}, which gives a similar result for $\mathcal{AW}_1$-projections onto the space of martingales with fixed marginals. However, next to the additional marginal constraints in the MOT problem (which we do not impose in our work), the scope of these papers differs from ours: they solely offer probabilistic arguments; no statistics are investigated.

\subsection{Outline} The rest of the paper is organized as follows. After defining various notations that we will use throughout the paper, we proceed to give an overview of our main contributions in Section \ref{sec:main results}. They consist of three parts:
\begin{itemize}
\item Section \ref{sec:mpd}, in which we introduce the projection distance to the space of martingale pairs and compute this projection distance in closed form;

\item Section \ref{sec:asymp}, in which we present our results on asymptotic statistics of the martingale projection under i.i.d.~assumptions (and fixed dimensions) as well as suitable mixing conditions. 

\item Section \ref{subsec:choice_of_sig}, in which we discuss the application to the hypothesis testing problem for martingale pairs (we also refer to these as martingale couplings) and present a brief study on the impact of $\sigma$ for the power of our martingale hypothesis test. 

\end{itemize}

In Section \ref{sec:conclusion}, we discuss a few interesting questions arising from our main results and present preliminary results to stimulate appetite for future research. In Section \ref{sec:experiment}, we provide various experimental studies with respect to the power analysis of the martingale pair test as well as its applications. In Appendix \ref{sec:mthd_dev}, we walk through the detailed technical developments of the theoretical results presented in Sections \ref{sec:main results} and \ref{sec:conclusion}. 
In the rest of the appendices, we discuss further applications, as well as deferred plots and algorithms.

\subsection{Notation} Let $\P_0$ be the distribution from which the data is drawn, and let $\P$ denote a generic probability measure. 
We denote the probability density and the probability measure of a smoothing random variable $\xi$ (introduced below) by $f_\xi$ and $\P_\xi$ respectively. Later in this paper, we will often make the specific choices \eqref{eq:fdensity} and \eqref{eq:f general}.

We write $\delta_x$ for the Dirac measure at $x$; $X\lawis\P$ if the random variable $X$ has distribution $\P$; $X\laweq Y$ if $X,Y$ have the same distribution; $\ddd$ for weak convergence. We also introduce the notation $\P\otimes \Q$ for the independent/product coupling of two probability measures $\P$ and $\Q$. We use $|\cdot|_2$ to denote the Euclidean norm, and more generally $|\cdot|_p$ to denote the $\ell^p$ norm of a vector. For $p>0$, let $L^p$ denote the collection of random variables $X$ with $\E[|X|_2^p]<\infty$. Similarly, we denote by $L^p(\R)$ the functions with finite $p$-moment wrt. the Lebesgue measure. We write $\n{X}_p=\E[|X|_2^p]^{1/p}$ for a random variable (or vector) $X\in L^p$ and $p>0$, and we use $\n{\cdot}$ for a generic norm. We write $\langle \cdot, \cdot \rangle$ for the scalar product on $\R^d\times \R^d$.  For two positive functions $f,g$, we write $f\asymp g$ if $f/C\leq  g\leq C f$ for some constant $C>0$.

 Throughout we fix a (standard) probability space $(\Omega, \F, \mu)$, on which all random variables are defined. If not specified otherwise, we take the expectation $\E[\cdot]$ with respect to $\mu$. For an Euclidean space $\mathcal{X}$ we denote by $\mathcal{P}(\mathcal{X})$ the set of probability measures on $\mathcal{X}$.

 \section{Main contributions} \label{sec:main results}
As described in the Introduction, our main contributions are the derivation of an appropriate projection distance between any distribution of the pair $(X,Y)$ and the space of martingale pairs, its asymptotic statistics, and the investigation of a consistent hypothesis test for the martingale pair property. We discuss the main results and provide a detailed technical development in Appendix \ref{sec:mthd_dev}.

\subsection{The empirical martingale projection distance}\label{sec:mpd}

In this subsection, we study the martingale projection distance (MPD) (see Definition \ref{def:MPD} below) and derive a closed-form expression for it. We then define the smoothed empirical MPD (SE-MPD), which will be used in our martingale pair test.

\subsubsection{Introducing the martingale projection distance}\label{subsec:intro_MPD}
Given $\P,\Q\in\mathcal P(\R^d\times\R^d)$ and $\gamma \geq 1$, we define $\mathcal{W}_\gamma$ via  
\begin{align*}
\mathcal{W}_\gamma(\P, \Q)^\gamma :=\inf\big\{\E\big[\nn{Y-Y'}_2^\gamma+\nn{X-X'}_2^\gamma\big]:~&\pi\in \mathcal{P}((\R^d)^4), (X,Y,X',Y')\lawis\pi,\\
&(X,Y) \lawis \P, (X',Y') \lawis \Q \big\}.
\end{align*}

As briefly discussed in the Introduction, this distance is not well suited to distinguish martingale laws from non-martingale laws (see Example \ref{ex:2} in Appendix \ref{subsec:deferred_mpd} below). While the marginal processes are adapted to their natural filtration, the couplings $\pi$ in the definition of $\mathcal{W}_\gamma$ need not be. This motivates the following definition:

\begin{definition}[see e.g.,~Lemma 2.2 of \citep{bartl2021wasserstein}]\label{def:bi-causal}
The probability measure $\pi\sim (X,Y,X',Y')$ is a \emph{bi-causal} coupling of the probability measures $\P$ and $\Q$ if
\begin{itemize}
    \item $\pi$ is a coupling of $\P$ and $\Q$, i.e., $(X,Y)\lawis \P$, $(X',Y')\lawis \Q,$
    \item $\text{Law}(Y|X,X')=\text{Law}(Y|X)$ \quad \emph{(causality from $\P$ to $\Q$)},
    \item $\text{Law}(Y'|X,X')=\text{Law}(Y'|X')$\quad \emph{(causality from $\Q$ to $\P$)}.
\end{itemize}
\end{definition}

\begin{definition} \label{def:bi-causal-wass}
  For two probability measures $\P, \Q$ on $\R^d\times \R^d$, we define the \emph{adapted, nested} or \emph{bi-causal Wasserstein distance}\footnote{These terms are used interchangeably in the literature.}  $\mathcal{AW}_\gamma$ as
\begin{align*}
\begin{split}
\mathcal{AW}_\gamma(\P, \Q)^\gamma :=\inf\big\{\E\big[\nn{Y-Y'}_2^\gamma+&\nn{X-X'}_2^\gamma\big]:~\pi\in \mathcal{P}((\R^d)^4), (X,Y,X',Y')\lawis\pi,\\
&\pi \text{ is a bi-causal coupling of $\P$ and $\Q$}\big\}.
\end{split}
\end{align*} 
\end{definition}

We will see in Example \ref{ex:3} in Appendix \ref{subsec:deferred_mpd} below that applying the adapted Wasserstein distance $\mathcal{A}\mathcal{W}_\gamma$ solves the problem mentioned in the Introduction, that the conditional
expectation is not a continuous function in the weak topology. We can now define the central object of this paper.

\begin{definition}\label{def:MPD}
 Given a probability measure $\P$ and $\gamma\geq 1$, we define the \emph{martingale projection distance} of $\P$ with exponent $\gamma$ as
\begin{align}
\mathrm{MPD}(\P,\gamma): =\inf\{ \mathcal{AW}_\gamma(\P, \Q)^\gamma: \E[Y'|X']=X' \text{ for } (X',Y')\lawis \Q\}.
\label{eq:causalcond2}
\end{align}
\end{definition}

In particular, we find the following explicit characterization of the MPD:
\begin{theorem}[Computing the martingale projection distance]\label{prop:martingale}
Let $(X,Y)\lawis \P$ and suppose that $(X,Y)\in L^\gamma$, i.e., $\E[|(X,Y)|^\gamma_2]<\infty$. Then
\begin{align}\label{eq:n4}
\mathrm{MPD}(\P,\gamma)=2^{1-\gamma}\E\left[\nn{X-\E[Y|X]}_2^\gamma\right].    
\end{align}
\end{theorem}
We refer to Example \ref{ex:causality} in Appendix \ref{subsec:deferred_mpd} for an explicit calculation of the MPD using Theorem \ref{prop:martingale}, which also emphasizes that the causality constraint is essential.

\begin{remark}
     For the case $d=1,$ the expression $\E_{\P}[|\E_{\P}[Y|X]-X|]$ appears in \citep[Definition 2.1]{guo2019computational} as a relaxation of the martingale constraint in discrete-time martingale optimal transport problems. However, no connection to the adapted Wasserstein distance is made. The work \citep{acciaio2022quantitative} considers an $\mathcal{AW}_\infty$-relaxation of the martingale constraint, leading to the inequality $|\E_{\P}[Y|X]-X|\le \epsilon$ for $\epsilon>0$, see \citep[Definition 2.8]{acciaio2022quantitative}. In comparison, we only consider the case of finite $\gamma$ in our work, and do not make any connections to no-arbitrage conditions.
\end{remark}

\subsubsection{The smoothed empirical martingale projection distance (SE-MPD)}\label{subsection:smoothed_MPD}

Consider a sequence of samples $\{(X_i,Y_i)\}_{i\in\N}$.
Having found a general closed-form expression for $\mathrm{MPD}(\P,\gamma)$ in Theorem \ref{prop:martingale} for a general $\P\in\mathcal P(\R^d\times\R^d)$, it is natural to look for a martingale pair test for the plugin estimator $\mathrm{MPD}({\P}_n,\gamma)$, where $${\P}_n:=\frac{1}{n}\sum_{i=1}^n \delta_{(X_i,Y_i)}$$ is the empirical measure associated to $\{(X_i,Y_i)\}_{1\le i\le n}$. However, if the samples are drawn i.i.d.~from distribution $\P_0$ under which $X$ is atomless and $\E_{\P_0}[\nn{(X,Y)}_2^\gamma]<\infty$, then we obtain from Theorem \ref{prop:martingale} and the strong law of large numbers, that $\mu$-a.s.
\begin{align*}
\lim_{n\to \infty} \mathrm{MPD}({\P}_n, \gamma)=\lim_{n\to \infty} \frac{2^{1-\gamma}}{n}\sum_{i=1}^n\nn{X_i-Y_i}_2^\gamma =  2^{1-\gamma}\E_{\P_0}[\nn{X-Y}_2^\gamma], 
\end{align*}
which is strictly greater than zero in general, even if $\P_0$ is a martingale law. In particular, $\mathrm{MPD}({\P}_n,\gamma)$ is \emph{not} a consistent estimator of $\mathrm{MPD}(\P_0,\gamma)$. To overcome this difficulty, we apply the following smoothing technique.

\begin{definition}\label{def:smooth}
Fix a law $\P_\xi$ of the random variable $\xi$. For any $\P\in\mathcal P(\R^d\times\R^d)$, we define the \emph{smoothed law} $\P^{*\xi}$ as 
\begin{align*}
\P^{*\xi}:=\text{Law}((X+\xi, Y+\xi)), \qquad (X,Y,\xi)\lawis \P\otimes\P_\xi.
\end{align*}
We define the \emph{smoothed MPD} as 
\begin{align*}
\mathrm{MPD}^{*\xi}(\P, \gamma) : =\inf\{ \mathcal{AW}_\gamma(\P^{*\xi}, \Q)^\gamma: \E[Y'|X']=X' \text{ for } (X',Y')\lawis \Q\}.
\end{align*}
\end{definition}

In other words, the smoothed MPD is the MPD \eqref{eq:causalcond2} applied to the smoothed law $\P^{*\xi}.$
Theorem \ref{prop:martingale} then leads to the formula
$$\mathrm{MPD}^{*\xi}(\P, \gamma)=2^{1-\gamma}\E\left[\nn{X-\E[Y|X+\xi]}_2^\gamma\right],$$
where $(X,Y)\lawis \P$ is independent of $\xi$.

At this point it might not be obvious to the reader, how the martingale property of $(X,Y)\lawis \P$ is affected by the smoothing via $\P_\xi$ as stated in Definition \ref{def:smooth}. In other words: for which $\P_\xi$ do we have
\begin{align*}
\mathrm{MPD}^{*\xi}(\P, \gamma)=0\quad \Longleftrightarrow\quad \mathrm{MPD}(\P, \gamma)=0?
\end{align*}
This motivates the following definition:

\begin{definition}\label{def:martingality-preserving}
 We say that the law $\P_\xi$ is \emph{martingality-preserving} if the following holds: for any law $\P$ on $\R^d\times\R^d$  and $(X,Y,\xi)\lawis\P\otimes\P_\xi$, 
\begin{align*}
(X,Y)\mathrm{\ is\ a\ martingale}\quad \Longleftrightarrow\quad (X+\xi,Y+\xi)\mathrm{\ is\ a\ martingale.}
\end{align*}
\end{definition}

It turns out that not every law $\P_\xi$ is martingality-preserving (see Example \ref{ex:smoothing}). Nevertheless, under mild assumptions on $\P_\xi$, the martingale property is actually invariant under smoothing.

\begin{proposition}\label{prop:martingale2}
Assume $$(X,Y, \xi)\lawis \P\otimes\P_\xi,$$ $\E[|(X,Y)|_2]<\infty$, and $\E[\nn{\xi}_2]<\infty$. Assume furthermore that $\P_\xi$ has a density $f_\xi$ and that the characteristic function $t\mapsto \E[e^{i\langle t, \xi\rangle}]$ has no real zero. Then $\P_\xi$ is martingality-preserving.
\end{proposition}

There are many smoothing measures $\P_\xi$ that satisfy the assumptions of Proposition \ref{prop:martingale2}.

\begin{example}\label{ex:main}
The assumptions of Proposition \ref{prop:martingale2} are satisfied for the density
\begin{align}
   f_{\xi,\rho}(x)=C_\rho(|x|_2+1)^{-\rho},\label{eq:fdensity}
\end{align}
where $\rho> d+1$ and $C_\rho^{-1}=\int (|x|_2+1)^{-\rho}\,\d x$.
\end{example}

We refer to Appendix \ref{subsec:deferred_mpd} for more examples of martingality-preserving laws, such as infinitely divisible distributions (Example \ref{ex:1}) and the Student's $t$ distribution (Example \ref{ex:smoothex}). For the rest of the paper, we will mostly work with \eqref{eq:fdensity} for simplicity.  We expect that a similar analysis works for the Student's $t$-distribution as well since it exhibits a similar tail behavior.

\sloppy Recall that our aim is to find an estimator of $\mathrm{MPD}(\P_0, \gamma)$ given i.i.d.~samples $\{(X_ n,Y_n): n\in\N\}$ drawn from a probability measure $\P_0$. While the plugin estimator $\mathrm{MPD}({\P}_n, \gamma)$ was unsuitable, we instead consider the following:

\begin{definition}
We call $\mathrm{MPD}^{*\xi}({\P}_n, \gamma)$
the \emph{smoothed empirical martingale projection distance} ($\mathrm{SE}$-$\mathrm{MPD}$) of $\P_n$ with exponent $\gamma$ and smoothing kernel $\xi$.
\end{definition}

In other words, instead of considering the raw empirical measure ${\P}_n$, we take its smoothed counterpart ${\P}_n^{*\xi}$, which is obtained by a convolution of the density $\P_\xi$ with the empirical measure. This is a classical procedure in statistics. In fact, a similar idea was used in \citep{glanzer2018} to construct an empirical measure that converges to $\P_0$ in $\mathcal{AW}_\gamma$. Furthermore, Proposition \ref{prop:martingale} states that there is no information about the martingale property lost when using the SE-MPD for estimation of $\mathrm{MPD}(\P, \gamma)$.
We emphasize that $\P_\xi$ is chosen by the statistician. 

We will show in the following that $\mathrm{MPD}^{*\xi}({\P}_n, \gamma)$ is a consistent estimator of $\mathrm{MPD}(\P_0, \gamma)$. In fact, under mild assumptions, $\mathrm{MPD}^{*\xi}({\P}_n, \gamma)$ has a parametric rate. To the best of our knowledge, this is the first martingale pair test statistic that breaks the curse of dimensionality. 

\subsection{Asymptotic distribution of the SE-MPD}\label{sec:asymp}

 We suppose throughout this section that $\P_0$ is a martingale measure. Our main findings can be summarized as follows:
\begin{tcolorbox}
\begin{center}
Under suitable conditions on $\P_0$ with regard to integrability and weak dependence, $\sqrt{n}\mathrm{MPD}^{*\xi}({\P}_n, \gamma)$ converges weakly to some explicit nontrivial random variable.
\end{center}
\end{tcolorbox}
We now summarize the main results that provide rigorous support for this message.
\begin{itemize}
    \item The simplest result of such form is Proposition \ref{prop:general_xi2} below, where a wide family of smoothing kernels are allowed. However, it is restricted to the i.i.d.~case with  $\gamma=1$, and the moment conditions are not optimal. The proof employs classical results on empirical processes.
    \item The more interesting Theorem \ref{thm:limitd simple} allows for a general choice of $\gamma\ge 1$ and less stringent moment assumptions in the i.i.d.~case. On the other hand, we will restrict to a special family of smoothing kernels that are heavy-tailed. Our proof builds on finite-sample estimates of empirical processes.
    \item When the data are not i.i.d.~but sufficiently mixing, Theorem \ref{thm:mixingconvergence} provides the desired convergence given sufficiently many moments in the case $\gamma=1$, with the same class of smoothing kernels as in Theorem \ref{thm:limitd simple}.
\end{itemize}

\subsubsection{The i.i.d.~case}\label{subsection:iid}

Our main result of this subsection is the following: 

\begin{theorem}\label{thm:limitd simple}
Let $\gamma\geq 1$ and for $\rho>\gamma+d$, consider the density   $f_{\xi,\rho}$ from \eqref{eq:fdensity}. There exists $C(\rho,d,\gamma)>0$ such that if the $\R^d\times\R^d$-valued martingale coupling $(X,Y)\in L^{C(\rho,d,\gamma)}$, then $f_{\xi,\rho}$ is a martingality-preserving law and we have the convergence in distribution 
\begin{align*}
n^{\gamma/2}\mpd^{*\xi}({\P}_n,\gamma)\ddd 2^{1-\gamma}\int_{\R^d}\frac{\nn{G_x}_2^\gamma}{\E[f_{\xi,\rho}(x-X)]^{\gamma-1}}\,\d x, \qquad n\to \infty,
\end{align*}
where  $\{G_x\}$ is a centered $\R^d$-valued Gaussian random field with covariance 
\begin{align*}
    \E[G_xG_y^\top]=\E[(Y-X) f_{\xi,\rho}(x-X)f_{\xi,\rho}(y-X)(Y-X)^\top],~x,y\in\R^d.
\end{align*}
In particular, the sequence $\{n^{\gamma/2}\mpd^{*\xi}({\P}_n,\gamma): n\in \N\}$ is tight.
\end{theorem}

Theorem \ref{thm:limitd simple} is a simplified version of its general version, Theorem \ref{thm:limitd}, which can be found in Appendix \ref{sec:mthd_dev}. Theorem \ref{thm:limitd simple} focuses on the densities from  \eqref{eq:fdensity}, while a standard scaling argument shows that densities of the form \begin{align}\label{eq:f general}
f_{\xi,\rho,\sigma}(x)=\sigma^{-d}C_\rho\left(\frac{\nn{x}_2}{\sigma}+1\right)^{-\rho}
\end{align}
also work. That is, $f_{\xi ,\rho,\sigma}$ is the density of $\sigma\xi$, where $\xi$ has density $f_{\xi,\rho}$. In particular, $f_{\xi,\rho,1}=f_{\xi,\rho}$.

A natural question is how the limit distribution of the rescaled MPD shown above depends on $f_{\xi,\rho}$ and $\sigma$. This question is in general quite hard to answer, and we focus on the most fundamental yet important case $\gamma=1$, where $(X,Y)\in L^{2(d+1)+\delta}$ for some $\delta>0$ (we refer to Section \ref{subsec:choice_of_sig} for discussions on the choice of $\sigma$ and Section \ref{subsec:general_case} for the case of general $f_\xi$). In the following, we fix $\rho>d+1$. Recall  \eqref{eq:f general}, which implies that 
\begin{align}
    f_{\xi,\rho,\sigma}(x)\asymp\begin{cases}\sigma^{-d}&\text{ for }\nn{x}_2\leq \sigma,\\
\sigma^{\rho-d}f_{\xi,\rho}(x)&\text{ for }\nn{x}_2> \sigma,
\end{cases}\label{eq:sigma}
\end{align}
where the constants may depend on the dimension $d$. 
Taking $\gamma=1$, Theorem \ref{thm:limitd simple} yields that
\begin{align}
\sqrt{n}\,\mathrm{MPD}^{*\xi}({\P}_n, 1)\ddd \int\nn{G_x}_2\,\d x.\label{eq:limit1}
\end{align}
where $\{G_x\}$ is a centered Gaussian process (depending on $\sigma$) with covariance
\begin{align}\label{eq:covariance}
    \E[G_xG_y^\top]=\E[(Y-X)f_{\xi,\rho,\sigma}(x-X)f_{\xi,\rho,\sigma}(y-X)(Y-X)^\top],~x,y\in\R^d.
\end{align}

\begin{theorem}[Expectation of the limit distribution]\label{thm:expectation}Suppose that the martingale coupling $(X,Y)$ is non-degenerate, i.e., $\P_0(X=Y)<1$. Then
    $$\E\left[\int\nn{G_x}_2\,\d x\right]\asymp 1\text{ as }\sigma\to\infty.$$
\end{theorem}

\begin{remark}\label{remark:generalsmooth}
So far in our analysis, $f_\xi$ lies in a fixed parametric class \eqref{eq:f general} with quite heavy tails. However, for $\gamma = 1$, restrictions in choosing $f_\xi$ can be relaxed. For example, we will see in Section \ref{subsec:general_case} that taking $f_\xi$ the Gaussian density still works, provided that the pair $(X,Y)$ is sufficiently integrable.
\end{remark}

\subsubsection{Stationary \texorpdfstring{$\alpha$}{}-mixing sequences: asymptotic distribution for \texorpdfstring{$\gamma=1$}{}}\label{subsection:alpha_mix}

A natural question is whether our theory can also handle dependent data while relaxing the i.i.d.~assumption. This is e.g., motivated by the Markov chain transition kernel testing problem detailed in Appendix \ref{sec:markov}. A prevalent setting of analyzing dependent data is where the data form a stationary mixing sequence. Among the classic mixing conditions for stationary processes, $\alpha$-mixing appears the least restrictive while still being technically tractable \citep{bradley2005basic}.

For two $\sigma$-algebras $\A,\B$, we define their $\alpha$-mixing coefficient
$$\alpha(\A,\B):=\sup_{S\in\A,T\in\B}|\P(S\cap T)-\P(S)\P(T)|.$$
\sloppy Given a stationary sequence $\{X_i\}_{i\in\N}$, we say it is  $\alpha$-mixing with coefficients $\{\alpha_n\}_{n\in\N}$ if $ \sup_{i\in\N}\alpha(\sigma(X_1,\dots,X_i),\sigma(X_{i+n},\dots))\le \alpha_n$ for any $n\in\N$; see \citep{rio2017asymptotic}. For $\lambda >2$ we define the quantity
$$A_{\alpha,\lambda }:=\sqrt{\int_0^1\alpha^{-1}(u)u^{-2/\lambda }\,\d u}\in\R\cup\{\infty\},$$where $\alpha^{-1}(u)=\sup\{n\in\N:\alpha_n\geq u\}$.
For example, we have $A_{\alpha,\lambda }<\infty$ if there exists $\kappa>1$ such that $\kappa^{-1}+2\lambda ^{-1}<1$ and $\alpha_n=O(n^{-\kappa})$.
This holds, for instance, for a stationary Gaussian sequence whose spectral density $f(\lambda)$ can be represented by $|P(e^{i\lambda})|^2g(\lambda)$, where $P$ is a polynomial, $g$ is $\lceil \kappa\rceil$-times differentiable and $g\geq m>0$ for some $m$; see \citep[Theorem 5]{ibragimov1970spectrum}.

We consider the most fundamental case $\gamma=1$ and recall the martingality-preserving density $f_{\xi,\rho,\sigma}(x)$ from \eqref{eq:f general}, where we now assume $\rho>2d^2$. We expect that a similar parametric analysis can be carried out for the case $\gamma>1 1$. Nevertheless, as in the i.i.d.~setting, $\gamma=1$ yields the largest class of feasible martingale couplings $(X,Y)$ (i.e., the moment condition being weakest; see Remark \ref{remark:momentchoice}). For simplicity of our presentation, we do not pursue this direction in detail.

\begin{theorem}[The limit distribution for the $\alpha$-mixing case]\label{thm:mixingconvergence}
    Let  $\lambda >2$ be such that $\{(X_i,Y_i)\}_{i\in\N}$ forms a stationary $\alpha$-mixing sequence with coefficients $\{\alpha_n\}_{n\in\N}$ satisfying $A_{\alpha,\lambda }<\infty$ and $\rho>2d^2$. Suppose also that all moments of $(X,Y)$ exist. Consider a   smoothing kernel $\xi$ with density given by \eqref{eq:fdensity}.  Then we have the convergence in distribution
    \begin{align*}
\sqrt{n}\,\mpd^{*\xi}({\P}_n,1)\ddd \int{\nn{G_x}_2}\,\d x, \qquad n\to \infty,
\end{align*} where $\{G_x\}$ is a centered Gaussian process with covariance given by \eqref{eq:gx}. In particular, the sequence $\{\sqrt{n}\,\mpd^{*\xi}({\P}_n,1)\}_{n\in\N}$ is tight.
\end{theorem}

\begin{remark}
The current form of Theorem \ref{thm:mixingconvergence} is stated for clarity but not generality: an optimal moment condition can be derived from Lemma \ref{lemma:numerator}
    in Appendix \ref{subsec:deferred_mpd} by solving the system of inequalities \eqref{eq:system of ineqs}. %
\end{remark}

\subsection{Martingale pair hypothesis test}\label{subsec:choice_of_sig}

As a direct application of our results, we introduce a novel martingale pair hypothesis test. Testing the martingale pair hypothesis is related to (but different from) testing if a sequence forms a martingale (which we call a martingale sequence hypothesis test in the following). For instance, if a sequence forms a martingale, then one can easily check by the tower property that consecutive pairs of random variables along the sequence form a martingale pair. 
In addition, a relatively simple extension from the martingale pair test to the martingale test is to project the couplings to the space $\E[Y|X]=AX$, where $A$ is a known matrix. In particular, by taking $A=[0,\dots,0,1]$, we can test  the hypothesis that $\E[X_{k}|X_{0},\dots,X_{k-1}]=X_{k-1}$ consistently for every fixed $k\in\N$. So, in principle, our test can be used to test the martingale sequence hypothesis.\footnote{Such a test involves i.i.d.~samples of the entire sequence. As will become clear below, existing martingale sequence hypothesis tests are often built on a single sample of the sequence.} 
However, the type of assumptions that we impose (e.g.,~i.i.d.~or stationarity) is better suited for applications such as testing a Markov kernel, certifying the no-arbitrage condition in generative finance models, or testing the efficiency of reinforcement learning policies. We will study these types of applications in detail in Section \ref{sec:apps}.

\subsubsection{A brief survey of existing approaches}\label{sec:survey}

One may consider the performance of well-known martingale sequence hypothesis tests in the context of testing the martingale pair hypothesis. A well-known martingale sequence hypothesis test has been developed in \citep{park2005test,phillips2014testing} for one-dimensional sequences, which can also be regarded as a generalized Kolmogorov--Smirnov test.\footnote{\citep{phillips2014testing} also consider Cram\'{e}r--von Mises tests, which we omit here for brevity.} We briefly describe (a slightly simplified version of) the approach of \citep{phillips2014testing} as follows: define the martingale null hypothesis
$$X_t=\theta X_{t-1}+u_t\quad\text{ with }\quad\theta=1, $$
where $\{u_t\}$ satisfies $\E[u_t|X_0,\dots,X_{t-1}]=0$ and certain uniform moment conditions. 
Next, define the least squares residual $\widehat{u}_{t}:=X_{t}-\widehat{\theta} X_{t-1}$, where 
$$\widehat{\theta}:=\frac{\sum_{t=1}^nX_{t-1}X_t}{\sum_{t=1}^nX_{t-1}^2}.$$
Let $\Delta X_{t}:=X_t-X_{t-1}$ and $\overline{\Delta X}:=(\sum_{t=1}^{n} \Delta X_{t})/n$.
Define
$$J_n(a):=\frac{\sum_{t=1}^{n} \Delta X_{t} \bone_{\{X_{t-1} \leq a\}}}{\left(\sum_{t=1}^{n} \widehat{u}_{t}^{2}\right)^{1 / 2}}$$
and its centered counterpart
$$J_n^*(a):=\frac{\sum_{t=1}^{n}\left(\Delta X_{t}-\overline{\Delta X}\right)  \bone_{\{X_{t-1} \leq a\}}}{\left(\sum_{t=1}^{n} \widehat{u}_{t}^{2}\right)^{1 / 2}}.$$
The quantities $J_{n}(a)$ and $J_{n}^{*}(a)$ are stochastic processes indexed by $a \in \mathbb{R}$, taking values in the space of RCLL functions. Define
$$\operatorname{GKS}_{n}  :=\sup _{a \in \mathbb{R}}\left|J_{n}(a)\right|\quad\text{ and }\quad \operatorname{GKS}_{n}^{*}  :=\sup _{a \in \mathbb{R}}\left|J_{n}^{*}(a)\right|.$$
Let $\{W(t)\}$ be a Brownian motion and $\{B(t)\}$ denote a Brownian bridge.
Under the null hypothesis, \citep[Theorem 3]{phillips2014testing} states that 
\begin{align}
    \operatorname{GKS}_{n}  \ddd\sup_{a\in\R}J(a)\quad\text{ and }\quad \operatorname{GKS}_{n}^{*}  \ddd\sup_{a\in\R}J^*(a),\label{eq:test}
\end{align}
where
$$J(a):=\int_{0}^{1} \bone_{\{W(s) \leq a\}} \d W(s)\quad\text{ and }\quad J^{*}(a)  :=\int_{0}^{1} \bone_{\{W(s) \leq a\}} \d B(s).$$
Based on \eqref{eq:test}, \citep{phillips2014testing} develop a consistent test against a wide class of nonlinear,
non-martingale processes including explosive AR(1) processes,
exponential autoregressive processes, etc. 

This approach naturally extends to a Kolmogorov--Smirnov-type statistic as follows: given i.i.d.~samples $\{(X_i,Y_i)\}_{1\leq i\leq n}$ sampled from $\P_0$, define
$$I_n(a):=\frac{\sum_{t=1}^n(Y_t-X_t)\bone_{\{X_t\leq a\}}}{(\sum_{t=1}^n(Y_t-X_t)^2)^{1/2}}.$$
Assuming suitable uniform moment conditions similarly to the above, under the null hypothesis (i.e., $\P_0$ is a martingale law), 
$$\sup_{a\in\R}|I_n(a)|\ddd \sup_{a\in\R}J(a).$$
Under the alternative hypothesis (i.e., $\P_0$ is not a martingale law), $\sup_{a\in\R}|I_n(a)|\to\infty$ in probability. This yields a consistent martingale pair test in one dimension.

However, this martingale pair test is inconsistent in $\mathbb{R}^d$. For example, take i.i.d.~standard Gaussian random variables $\xi_1,\xi_2$ and consider $X=(\xi_1,\xi_2)$ and $Y=(\xi_1+\xi_2,\xi_1+\xi_2)$. We use $X^j$ for $j=1,2$ to denote the $j$-th entry of the vector $X$ (similarly for $Y^j$). It follows that for a fixed $j$, the pair $(X^j,Y^j)$ forms a martingale pair. However, $(X,Y)$ is not a martingale in dimension $d=2$. In other words, the martingale property in separate dimensions does not guarantee joint martingality. Our martingale pair test solves this inconsistency issue. 
It is, however, important to note that in contrast to \citep{phillips2014testing}, we impose a mixing assumption on the sequence itself, and not on the martingale differences. While the approach we present can be adapted to martingale differences, our assumptions are motivated by the applications mentioned earlier and discussed in Section \ref{sec:apps}.

The work of \citep{chang2022testing} tests the martingale difference hypothesis for high-dimensional sequences. A $d$-dimensional time series $\{X_t\}_{t\in\N_0}$ is called a martingale difference sequence if $\E[X_t|X_0,\dots,X_{t-1}]=0\in\R^d$ for every $t\in\N_0$. In this case, $\E[\mathrm{vec}\{\phi(X_t)X_{t+j}^\top\}]=0$ for every measurable $\phi:\R^d\to\R^p$ and $j\in\N$ such that the expectation exists, where $\mathrm{vec}\{\cdot\}$ denotes the vectorization of a matrix. Consequently, given the observations from the sequence $\{X_t\}_{1\leq t\leq n}$, \citep{chang2022testing} proposes the following martingale difference sequence test: select a family $\phi$ of test functions, let
\begin{align}
    \gamma_j:=\frac{1}{n-j}\sum_{t=1}^{n-j}\mathrm{vec}\{\phi(X_t)X_{t+j}^\top\},\label{eq:gammaj}
\end{align}
and consider the null hypothesis that $\gamma_j=0\in\R^{dp}$ for all $j\geq 1$. This test is, however, not a consistent test since it only involves finitely many moments $\phi$. The univariate case $d=1$ is also considered in the earlier works \citep{dominguez2003testing,escanciano2009testing}, with the same idea of selecting a finite number of test functions $\phi$. We refer to the references therein and the introduction of \citep{phillips2014testing} for a more comprehensive survey. In our setting of the martingale pair test, a natural and similar approach is replacing \eqref{eq:gammaj} with $\E_n[\mathrm{vec}\{\phi(X)Y^\top\}]$, where $\E_n$ denotes the empirical distribution, but the test is not in general consistent as $d$ is finite. 

In the next subsection, we provide a detailed martingale pair hypothesis test that is consistent in an arbitrary dimension. 

\subsubsection{Implementation and test properties}\label{sec:implementation}

In this section, we provide a concrete guide for the implementation of the test and study a range of test properties including Type I error coverage, consistency, and some power-related results.

To implement our results on martingale pair testing under the assumptions of Theorem~\ref{thm:mixingconvergence}, we propose the following three-step procedure for constructing a test with an asymptotic Type I error of 95\%:

\begin{itemize}

\item \textbf{Step 1:} Compute $n^{1/2}\mpd^{*\xi}({\P}_n,1)$ as a function of $\sigma$ and select $\sigma \geq 1$ in order to maximize $\mpd^{*\xi}({\P}_n,1)$. 

\item \textbf{Step 2:} Compute the 95\% quantile of the generalized chi-squared distribution $\int{\nn{G_x}_2}\,\d x$; this can be computed via Monte Carlo simulation. 

\item \textbf{Step 3:} Reject the hypothesis if $n^{1/2}\mpd^{*\xi}({\P}_n,1)$ is larger than the quantile computed in Step 2.

\end{itemize}


Steps 2 and 3 follow naturally once $\sigma$ is selected in Step 1. The rationale behind Step 1 is to enhance the test's power; we will discuss this motivation precisely following Proposition~\ref{prop:2k}. Intuitively, the chosen $\sigma_n$ from Step 1 remains bounded within a compact set under both the null hypothesis and a broad class of alternative hypotheses. Under the null, the asymptotic Type I error rate—i.e., the probability of incorrectly rejecting the martingale pair hypothesis—is accurately controlled through the quantile of the limiting distribution $\int \nn{G_x}_2\,\d x$. This control follows directly from the uniform continuity (with respect to $\sigma$ on compact sets) of the distribution of $\int \nn{G_x}_2\,\d x$. Conversely, under alternative hypotheses, the test statistic in Step 3 increases, facilitating rejection. Indeed, as the sample size grows, the statistic diverges to infinity, thus ensuring consistency of the test.

\begin{corollary}
    Under the assumptions of Theorem \ref{thm:limitd} (which can be found in Appendix \ref{sec:thm5} below), if $\P(\E[Y | X] \neq X)>0$, we have $\mpd^{*\xi}(\P_n,1) \rightarrow \mpd^{*\xi}(\P_0,1) > 0$. In particular, the probability of rejecting the hypothesis converges to 1.
\end{corollary}

In Step 1 of our description above, we propose selecting $\sigma$ by maximizing the SE-MPD as a function of $\sigma$. We will study the behavior of such $\sigma = \sigma_n$ depending on how similar a non-martingale pair generating process is from a martingale, in a simple setting where $(X,Y)$ is bounded and takes values in $\R^2$. 
In what follows, for strictly positive random variables $\xi_n$ and constants $c_n$, we write $\xi_n\asymp c_n$ if both sequences $(\xi_n/c_n)_{n\in\N}$ and $(c_n/\xi_n)_{n\in\N}$ are tight.

\begin{proposition}\label{prop:2k}
 Fix $k\in\N$ and assume that $(X,Y)$ is bounded and takes values in $\R^2$. Consider the following hypothesis testing problem.
 \begin{align*}
     &\text{$H_0\mathrm{:}$ $(X,Y)$ is a martingale;}\\
     &\text{$H_1\mathrm{:}$ $\E[(Y-X)X^j] = 0$ for $j=0,1,\dots,k-1$ and $\E[(Y-X)X^k] \neq 0$.}
 \end{align*}
  Let $(\sigma_n)_{n\in\N}$ be a positive sequence in $\R$ satisfying $\sigma_n\to\infty$ and $\sigma_n=o(n^{1/(2k)})$, and let $f_{\xi, \rho, \sigma} $ defined in \eqref{eq:f general} be the smoothing kernel. Then under $H_0$, the sequence $\{\sqrt{n}\,\mpd^{*\xi}_{\sigma_n}(\P_n,1)\}_{n\in\N}$ is tight, and under $H_1$,
  \begin{align}
      {\sqrt{n}\,\mpd^{*\xi}_{\sigma_n}(\P_n,1)}\asymp \frac{\sqrt{n}}{(\sigma_n)^k}.\label{eq:alt}
  \end{align}
  On the other hand, if $n^{1/(2k)}=o(\sigma_n)$, the sequence $\{\sqrt{n}\,\mpd^{*\xi}_{\sigma_n}(\P_n,1)\}_{n\in\N}$ is tight under both hypotheses.
\end{proposition}



Under $H_1$, it is a consequence of Proposition \ref{prop:2k} that the maximizer of $\sigma$ in Step 1 corresponds to choices of $\sigma_n$ satisfying $\sigma_n=O(1)$. Indeed, otherwise taking $\sigma_n'=o(\sigma_n)$ with $\sigma_n'\neq O(1)$ results in an asymptotically larger MPD.\footnote{In fact, as long as $\P_0$ is not a martingale law, we can select $k$ arbitrarily large by a density argument.} Along with \eqref{eq:alt}, this implies that the Type II error is controlled as $n\to\infty$. 
On the other hand, under $H_0$, we expect that the maximizer of $\sigma_n\mapsto n^{1/2}\mpd^{*\xi}_{\sigma_n}({\P}_n,1)$ stabilizes as $n\to\infty$, in view of Theorem \ref{thm:limitd simple}. Moreover, Theorem \ref{thm:expectation} and Markov's inequality imply that the quantile in Step 2 will remain bounded even if $\sigma_n$ is large. 
Although not pursued in this paper, we expect that using a similar chaining argument as the proof of Lemma \ref{lemma:Esup when gamma=1}, one can prove that $\{\sup_{\sigma\geq 1}\sqrt{n}\,\mpd^{*\xi}_{\sigma}(\P_n,1)\}_{n\geq 1}$ is tight. In this way, we gain uniform control on the Type I error, a consequence of the above tightness and the boundedness of the quantile in Step 2.

In other words, we argue that $\sigma = O(1)$ leads to an asymptotically exact Type I error specified by the test (this is the point of choosing the quantile as indicated in Step 2). On the other hand, the hard instances of alternatives (i.e., data consistent with processes that are very similar to martingales, such as the setting of Proposition \ref{prop:2k}) lead to a selection according to Step 1 that is also close to $\sigma_n = O(1)$, as discussed in the previous paragraph. 
Based on this intuition, we believe that our selection criterion for $\sigma$ is sensible. The statistical properties of this test (e.g., asymptotic efficiency) are interesting and will be studied in future work.

\section{Setting the stage for future research} \label{sec:conclusion}

The goal of this section is to stimulate the appetite and set the stage for future research questions of importance strongly connected with our main contributions. We divide this section into two subsections. First, we study finite-sample rates for the MPD, which are obviously interesting in their own right, but in particular may be helpful in further studying the martingale pair test that we introduce. We will conclude that an investigation of finite-sample rates involves the choice of the smoothing kernel. Thus, the second subsection revisits our statistical analysis in the context of general smoothing kernels that may not be of the form \eqref{eq:f general}.

\subsection{Finite-sample rates for \texorpdfstring{$\gamma=1$}{}}\label{subsec:finite_sample_rates}
In addition to large-sample asymptotic statistics, finite-sample asymptotics can also be developed. In this section, we provide preliminary results on the finite-sample asymptotics for the SE-MPD when $\gamma=1$. 

Assume that $\P_0$ is the law of a martingale pair. We apply classical tools from empirical process theory to derive an upper bound of $\E[ \mpd^{*\xi}(\P_n,1) ]$ in Proposition \ref{prop:finitesample} below. More concretely we show that $\E[\mpd^{*\xi}(\P_n, 1)]=O(n^{-1/2})$, i.e.,~the SE-MPD exhibits a parametric convergence rate.

Recall the density given by \eqref{eq:f general}. We make this choice mainly for technical reasons:  as will become clear from the proof (see Section \ref{sec:tech_dev_finite_and_general}), it is only important that $\P_\xi$ exhibits heavier tails than $\P_0$. The main result in this subsection is the following:

\begin{proposition}\label{prop:finitesample}
Suppose that $\rho>d+1$ and $(X,Y)\in L^{2p}$  for some $p>d+1$. Then there exists a universal constant $L>0$ such that
\begin{align}
&\hspace{0.5cm} \mathrm{MPD}^{*\xi}(\P_n,1) \nonumber\\
&\leq Ln^{-1/2}\sum_{j=1}^d\n{X^j-Y^j}_{2p}C_\rho\Bigg[ {d^{3/2}\rho  }\left(\E[\nn{X}_{2}^{2p}]^{1/(2q)}2^{p-1}\frac{\sigma^{-p}}{p-1-d}+2^{\rho+1}\frac{\sigma^{-1}}{\rho+1-d}\right)\nonumber\\
&\hspace{1cm}+ d\left(\E[\nn{X}_{2}^{2p}]^{1/(2q)}2^{p-1}\frac{\sigma^{-(p-1)}}{p-1-d}+2^{\rho}\frac{1}{\rho-d}\right)+ (\sqrt{d}\rho\sigma^{-(d+1)}+\sigma^{-d})(\sigma^d+1)\Bigg].\label{eq:finiterate}
\end{align}
\end{proposition}
We remark that a finite-sample guarantee similar to Proposition \ref{prop:finitesample} is also achievable for the mixing case using the explicit bound \eqref{eq:mixingfinitesample} below in Lemma \ref{lemma:numerator} in Appendix \ref{sec:thm5} . 

The upper bound \eqref{eq:finiterate} is far from being tight in general, and there is certainly room for improvement. This can be achieved by changing the smoothing measure $\P_\xi$. In the next section, we provide an exemplary study relaxing for a broader class of $\P_\xi$ under simplified assumptions.  We leave a full analysis of the martingale pair test for different $\P_\xi$ for future research, and focus on $f_{\xi,\rho}(x)=C_\rho(|x|_2+1)^{-\rho}$ in the main body of this paper. 

\subsection{Towards a general selection of smoothing kernel}\label{subsec:general_case}
Let us define $$f_a^r(x,y):=(1\vee |a|_2^r)(y-x)f_\xi(a-x)$$ for some $r>d$ and set 
$$\mathcal{F}^r_j:= \{(f_a^r)_j: a\in \mathbb{R}^d\},\qquad j=1, \dots, d.$$
We also set $$\|\nabla^k_{x,y}f_\xi(x)\|_\infty:=\max_{i_1, \dots, i_k\in\{1,\dots, d\}} |\nabla^k_{x_{i_1}\dots x_{i_{k}}} f_\xi(x)|$$
for $k\in \N_0$ and $\beta:=\lceil 2d/(2-\delta) \rceil$ for $\delta\in  (0,1)$.

Throughout this subsection, we make the following standing assumption:

\begin{assumption}\label{ass:1}
The following are satisfied for some $\delta\in (0,1)$:
\begin{itemize}
\item There exists $D=D(\delta, d)$ such that 
\begin{align}\label{eq:derivative}
(1\vee |x|_2^r) \max_{0\le k\le \beta} \|\nabla_{x,y}^k f_\xi(x)\|_\infty \le D
\end{align} 
holds for all $x\in \mathbb{R}^d$.
\item $(X,Y)\in L^s$ for some $s>4(2d+1+r)$.
\end{itemize}
\end{assumption}

Our main result for this subsection is the following:
\begin{proposition}\label{prop:general_xi}
Under Assumption \ref{ass:1} there exists a constant $C=C(\beta, \delta, d,r, D)$ such that
\begin{align*}
    \sqrt{n}\,\E[ \mathrm{MPD}^{*\xi}(\P_n,1) ]\le C \int_0^{C\big(1+\E\big[|(X,Y)|_2^{2(r+1)}\big]\big)} \Big(\frac{1}{\epsilon}\Big)^{1-\delta/2}\,\d\epsilon.
\end{align*}
\end{proposition}

\begin{remark}
In order to give a unified presentation of results in this subsection, we have refrained from optimizing the moment condition on $(X,Y)$ in Assumption \ref{ass:1}. In fact, as $$4(2d+1+r)>2(1+3d)>2(d+1),$$
this moment condition is more stringent than the moment condition we impose in Theorem \ref{thm:limitd} below for the case $\gamma=1$; see also Remark \ref{remark:momentchoice}. However, the choice of $f_{\xi,\rho}$ in Theorem \ref{thm:limitd} below is in a fixed parametric class with quite heavy tails, while Proposition \ref{prop:general_xi2} offers much greater flexibility in choosing $f_\xi$, e.g.,~one can choose $f_\xi$ to be the normal density or other kernels frequently used in density estimation.
\end{remark}

Once again, the finite-sample bound we developed in Proposition \ref{prop:general_xi} is likely not optimal. Our discussion of the case of the general $f_\xi$ offers a starting point for future research of this topic. Deriving finite-sample bounds that are also of practical use is left for further investigation. The approach leading to Proposition \ref{prop:general_xi} can be used directly to establish a particular case of Theorem \ref{thm:limitd} directly. We record this result as the next proposition. 

\begin{proposition}\label{prop:general_xi2}
Under Assumption \ref{ass:1}, it holds that
\begin{align}
\sqrt{n}\,\mathrm{MPD}^{*\xi}(\P_n,1)\ddd \int \nn{G_x}_2\,\d x,\label{eq:limitd2}
\end{align}
where $\{G_x\}$ is a centered $\R^d$-valued Gaussian random field with covariance 
\begin{align}\label{eq:gx}
    \E[G_xG_y^\top]=\E[(Y-X) f_\xi(x-X)f_\xi(y-X)(Y-X)^\top],~x,y\in\R^d.
\end{align}
\end{proposition}

\section{Experiments and Applications}\label{sec:experiment} 
In this section, we study the power of our martingality test based on simulated data and discuss a few applications, with an emphasis on testing no-arbitrage in
existing pricing models for financial derivatives. Implementations for this section are made publicly available at our \href{https://github.com/Ericavanee/Bicausal_Wasserstein_MtglProj}{GitHub Repository}. 

\subsection{Power analysis} \label{subsec:power}
Recall the martingale pair hypothesis test from Section \ref{sec:implementation}. To analyze the power of our test, we begin by evaluating the asymptotic critical values $c_\alpha$  corresponding to a given significance level $\alpha$. This is done for dimensions $d = 1 $ and $d = 2 $, under the assumption $\gamma = 1$.

To start, consider a coupling $(X,Y) \in \mathbb{R}^d \times \mathbb{R}^d$ defined as follows:
$$X \sim \mathcal{N}(0,I_d), \ \ Z \sim \mathcal{N}(0,I_d), \ \ Y = X+Z,$$ with $I_d$ denoting the $d\times d$ identity matrix and $X,Z$ independent. 
We simulate the distribution of $\int_{\R^d}\nn{G_x}_2\d x$ in \eqref{eq:limit1} by first computing the covariance of $\{G_x\}$ in \eqref{eq:covariance} via numerical integration using $100$ samples from the coupling $(X,Y)$, then simulating the Gaussian process $\{G_x\}$ (on a discrete grid) 
using the covariance, and finally performing numerical integration to compute $\int_{\R^d} |G_x|_2 \d x$ approximating the integral with a sum over the integer grid of $[-50,50]^d$. The same experiment is repeated $1000$ times for $d=1,2$. 
Details of the simulation procedure are provided in Appendix \ref{appn}. 
Figures \ref{fig:hist_d1} and \ref{fig:hist_d2} 
present histograms of $\int_{\R^d} |G_x|_2 \d x$. 
We consider $f_{\xi, \rho, \sigma}$ with $\rho = 5$ and $\sigma = 1$ for all simulations.

\begin{figure}[H]
\begin{subfigure}{.5\textwidth}
  \centering
  \includegraphics[width=1\linewidth]{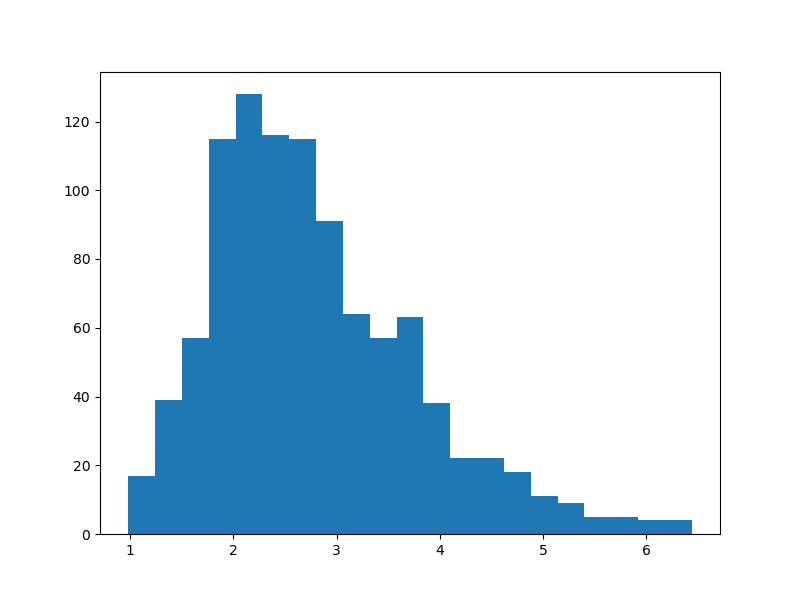}
\end{subfigure}%
\begin{subfigure}{.5\textwidth}
  \centering
  \includegraphics[width=1\linewidth]{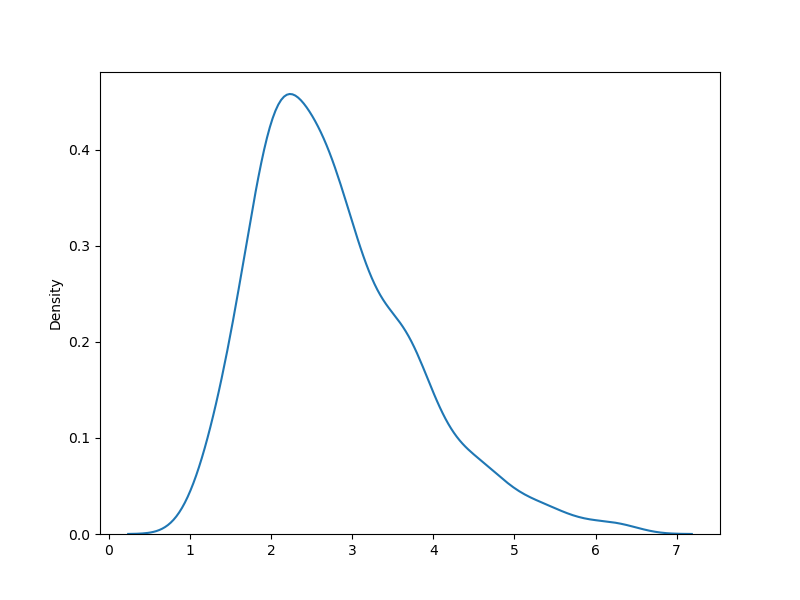}
\end{subfigure}
\caption{ Histogram (with 1000 samples) and approximating density of the asymptotic distribution $\int_{\R^d}\nn{G_x}_2\d x$ of the rescaled MPD for $d=1$, where the martingale coupling $(X,Y)$ is given by $X \sim \mathcal{N}(0,I_d), \ Z \sim \mathcal{N}(0,I_d),\ Y = X+Z$, and $X,Z$ independent.}
\label{fig:hist_d1}
\end{figure}

\begin{figure}[H]
\begin{subfigure}{.5\textwidth}
  \centering
  \includegraphics[width=1\linewidth]{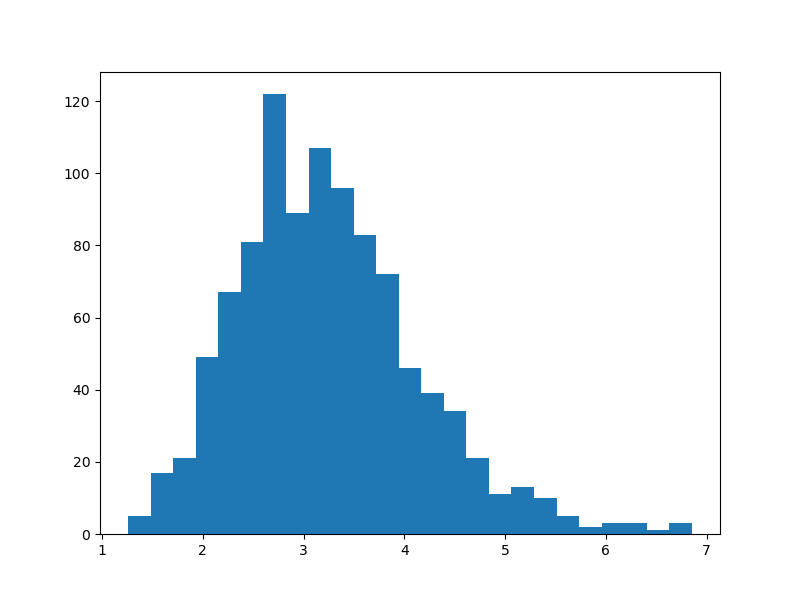}
\end{subfigure}%
\begin{subfigure}{.5\textwidth}
  \centering
  \includegraphics[width=1\linewidth]{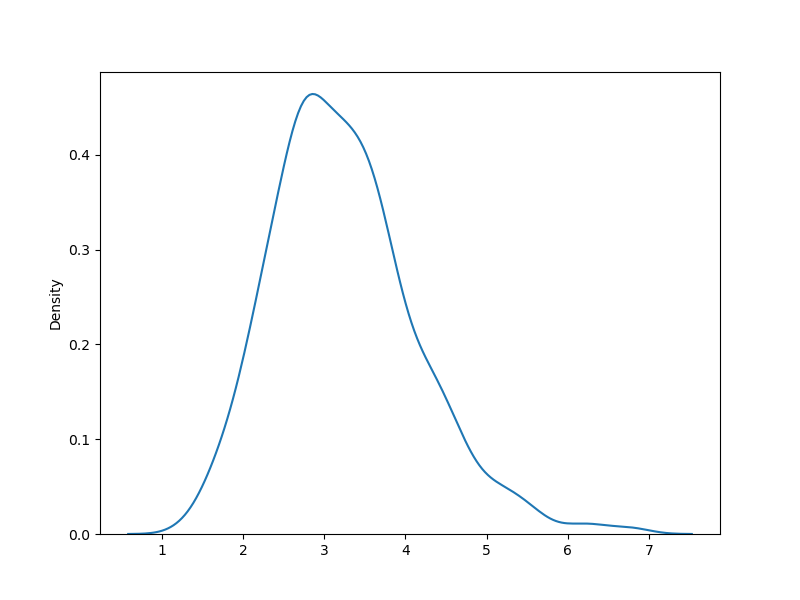}
\end{subfigure}
\caption{Histogram (with 1000 samples) and approximating density of the asymptotic distribution $\int_{\R^d}\nn{G_x}_2\d x$ of the rescaled MPD for $d=2$, where the martingale coupling $(X,Y)$ is given by $X \sim \mathcal{N}(0,I_d), \ Z \sim \mathcal{N}(0,I_d),\ Y = X+Z$, and $X,Z$ independent.}
\label{fig:hist_d2}
\end{figure}

Table \ref{table:critical_values} summarizes the asymptotic critical values of our martingality test statistics for various significance levels obtained from the simulation. These values will be used to perform hypothesis testing in the examples presented in Sections \ref{subsec:simulation} and \ref{sec:sde}.

\begin{table}[ht]
    \centering
    \begin{tabular}{ccc}
        \toprule
        Sig. Level & Dimension 1 & Dimension 2 \\
        \midrule
        0.99 & 1.84 & 1.604 \\
        0.95 & 1.485 & 1.964 \\
        0.90 & 1.717 & 2.199 \\
        0.10 & 4.102 & 4.419 \\
        0.05 & 4.705 & 4.840 \\
        0.01 & 5.740 & 5.835 \\
        \bottomrule
    \end{tabular}
    \caption{Critical Values for Different Significance Levels}
    \label{table:critical_values}
\end{table}

\subsubsection{Synthetic Examples}\label{subsec:simulation} To verify the consistency of our test, we conduct hypothesis testing on three synthetic examples, where the presence or absence of martingality is evident. Below, we describe the setup for our martingality hypothesis test.

Given two observed sequences of random variables $X$ and $Y$, let $H_0$ represent the hypothesis that $(X, Y)$ forms a martingale coupling, i.e., $\E[Y|X] = X$. For each test, we draw $n = 100$ samples from the specified distributions and perform $N = 100$ replications to estimate the asymptotic size of the test.

We consider the following test cases, using a significance level of $\alpha = 0.05$. For better clarity, test cases where the null hypothesis is likely to be accepted (i.e., the tested coupling is close to a martingale) are labeled as `NULL', while cases, where the null hypothesis is likely to be rejected (i.e., the tested coupling is \textit{not} close to a martingale), are labeled as `ALT':

\begin{enumerate}
\item Random Walk (NULL1): $X \sim \mathcal{N}(0,1)$, $Z \sim \mathcal{N}(0,1),~Y= X+Z$, where $X,Z$ are independent. 
\item Hermite Polynomials (ALT1): $X \sim \mathcal{N}(0,1)$, $Y = X+H_k(X)/\sqrt{k!}$, where $k = 1$.
\item Hermite Polynomials (NULL2): $X \sim \mathcal{N}(0,1)$, $Y = X+H_k(X)/\sqrt{k!}$, where $k = 20$.
\end{enumerate}

Here, $H_k(x)$ denotes the Hermite polynomial of order $k$, defined as $$H_k(x) = (-1)^ke^{x^2/2}\frac{\d^k}{\d x^{k}}(e^{-x^2/2}).$$ We also note that the design of $Y$ in the Hermite polynomial examples (ALT1 and NULL2) is motivated by the following properties:
\begin{enumerate}[(i)]
\item If $Z$ is standard Gaussian, then $\E[H_k(Z) H_\ell(Z)] = 0$ for $k \neq \ell$.
\item For $(X, Y) = (Z, Z + H_k(Z))$, we have $\E[(Y - X) X^j] = 0$ for all $j = 0, \dots, k-1$, and $\E[(Y - X) X^k] \neq 0$.
\end{enumerate}
These properties ensure that the coupling $(Z, Z + H_k(Z)/\sqrt{k!})$ is close to a martingale if the order $k$ of the Hermite polynomial $H_k$ is high (NULL2) and vice versa (ALT1). The $1/\sqrt{k!}$ normalization factor arises from the relation $\E[(Z+H_k(Z))^2]=1+\E[H_k(Z)^2]\asymp k!$.

The results of the hypothesis tests to synthetic examples (1), (2), and (3) are summarized in Table \ref{table:synthetic}, where $p$ represents the empirical rejection rate (i.e., the proportion of tests that reject the null hypothesis) and $\bar{T}$ denotes the mean of the test statistic, and we fix $\rho = 5, \sigma = 1$. With a significance level at $\alpha = 0.05$ and $N=100$ tests conducted, we see in Table \ref{table:synthetic} that the martingale test always accepts NULL1 and rejects ALT1, and accepts NULL2, which is very close to a martingale, 99\% of the time, as we expected.

\begin{table}[H]
\centering
\begin{tabular}{l S[table-format=1.2] S[table-format=3.0] S[table-format=3.0] S[table-format=1.2] S[table-format=1.3]}
\toprule
      & {$\alpha$} & {$N$}    & {$n$}   & {$p$} & {$\bar{T}$}                    \\
\midrule
NULL1 & 0.05     & 100 & 100 &   0.00  &    1.643                 \\
ALT1 & 0.05     & 100 & 100 &   1.00   &   7.318                   \\
NULL2 & 0.05     & 100 & 100 &   0.01   &  1.804 \\
\bottomrule
\end{tabular}
\caption{Testing the martingale property of the synthetic examples
}
\label{table:synthetic}
\end{table}

\subsubsection{The impact of \texorpdfstring{$\sigma$}{}} We recall that our martingality test relies on a centered Gaussian process $\{G_x\}$ which depends on $\sigma$ as described in \eqref{eq:covariance}.
This parameter $\sigma$ plays an important role in our test. Varying values of $\sigma$ can have the following effects: 
\begin{enumerate}[(i)]
    \item martingales are smoothed to different degrees;
    \item martingale projection error is reduced at different rates;
    \item the associated critical values change.
\end{enumerate}
We will proceed to present examples that demonstrate these three impacts of $\sigma$. 

To see the smoothing effect of $\sigma$, consider the martingale coupling $(B_1,B_2)$ where $\{B_t\}_{t\geq 0}$ is a Brownian motion. Figure \ref{fig:sig_smooth} illustrates the choice of $\sigma = 0.01$, $\sigma = 0.1$, and $\sigma = 1$ with $100$ samples. One can immediately observe the trend that the bigger the value of $\sigma$, the more prominent the effect of smoothing on the martingale coupling.

\begin{figure}[H]
\centering
\begin{subfigure}[b]{.3\textwidth}
  \centering
  \includegraphics[width=\linewidth]{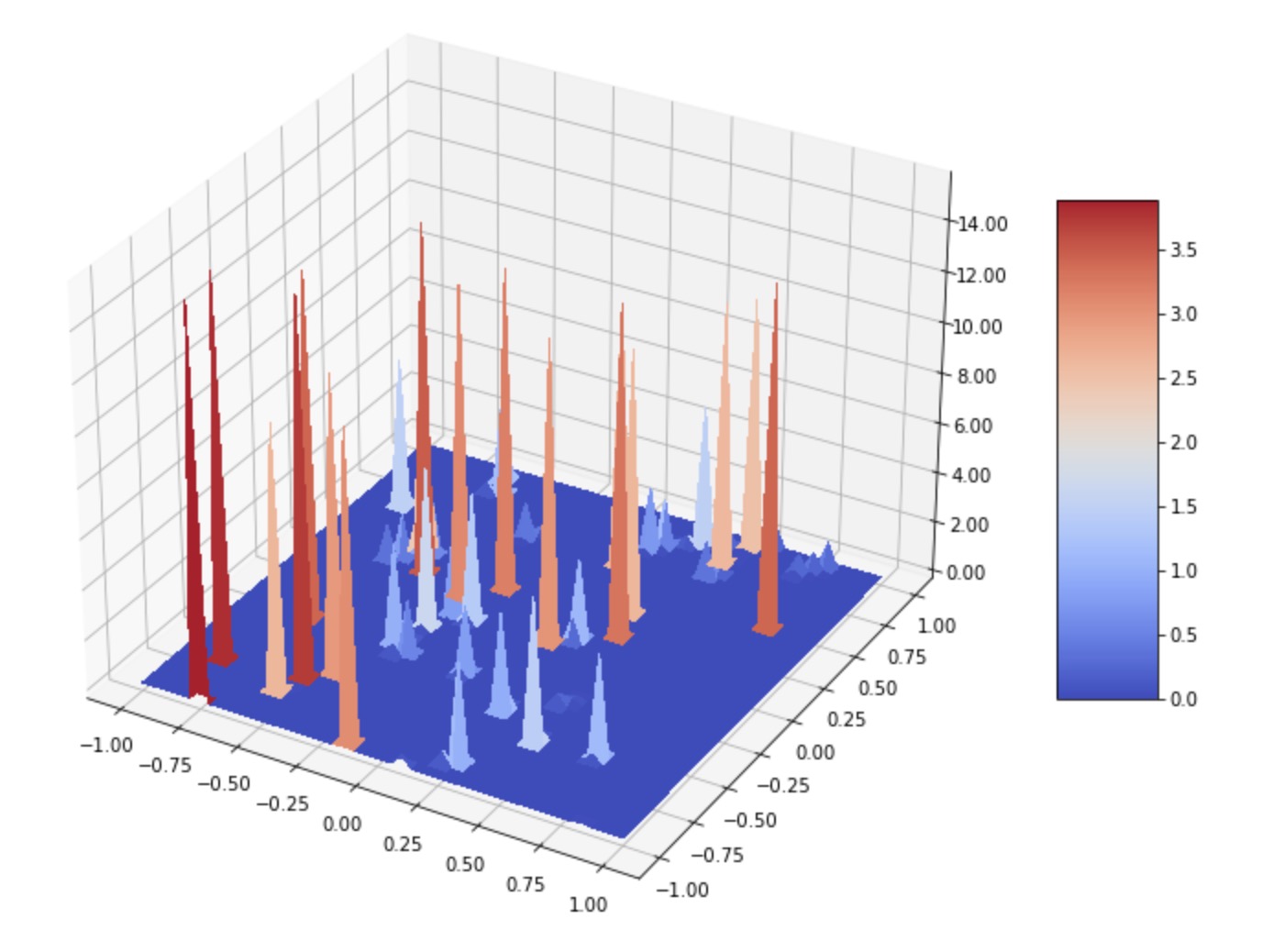}
  \caption{$\sigma = 0.01$}
\end{subfigure}%
\hfill
\begin{subfigure}[b]{.3\textwidth}
  \centering
  \includegraphics[width=\linewidth]{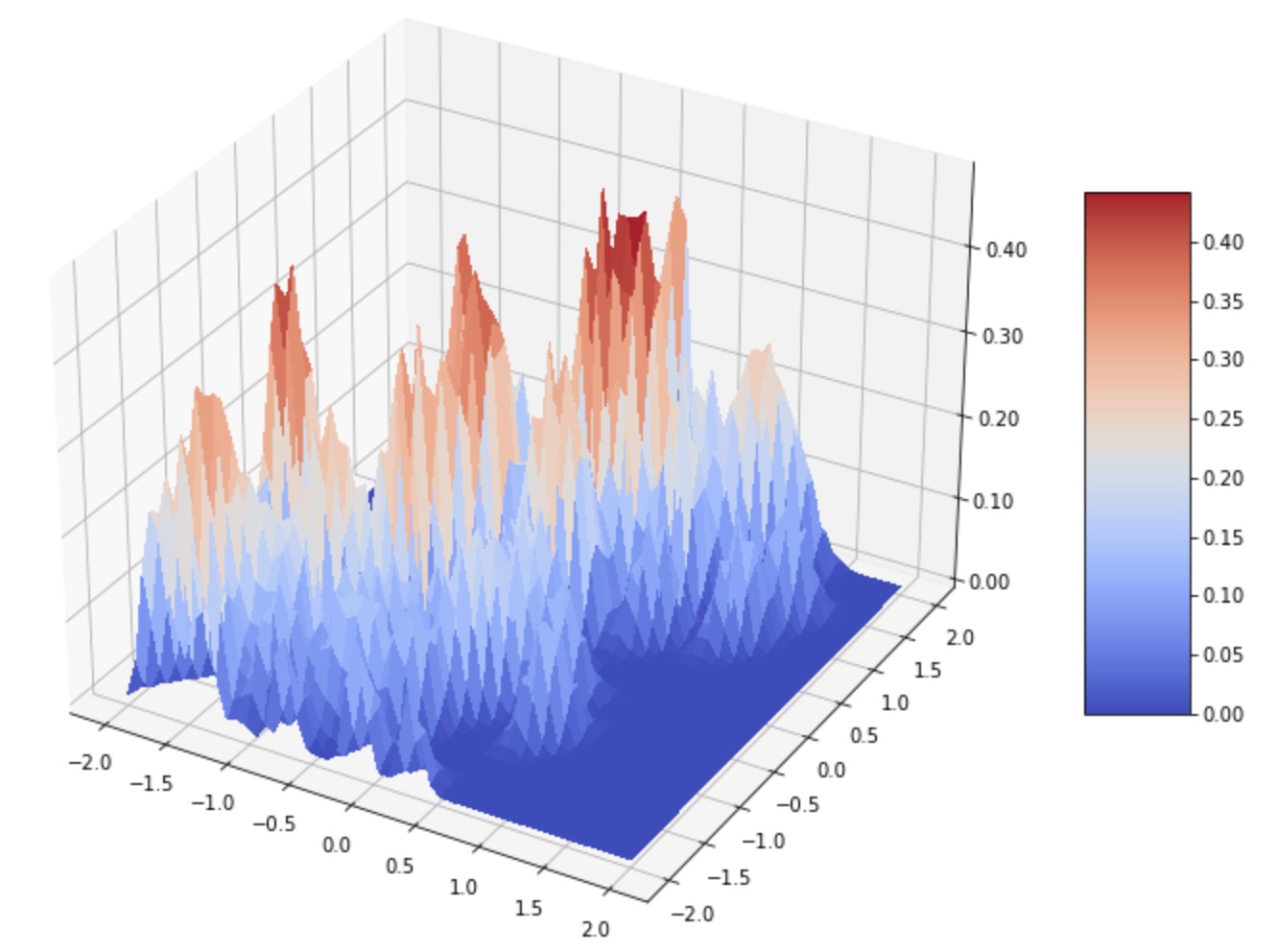}
  \caption{$\sigma = 0.1$}
\end{subfigure}
\hfill
\begin{subfigure}[b]{.32\textwidth}
  \centering
  \includegraphics[width=\linewidth]{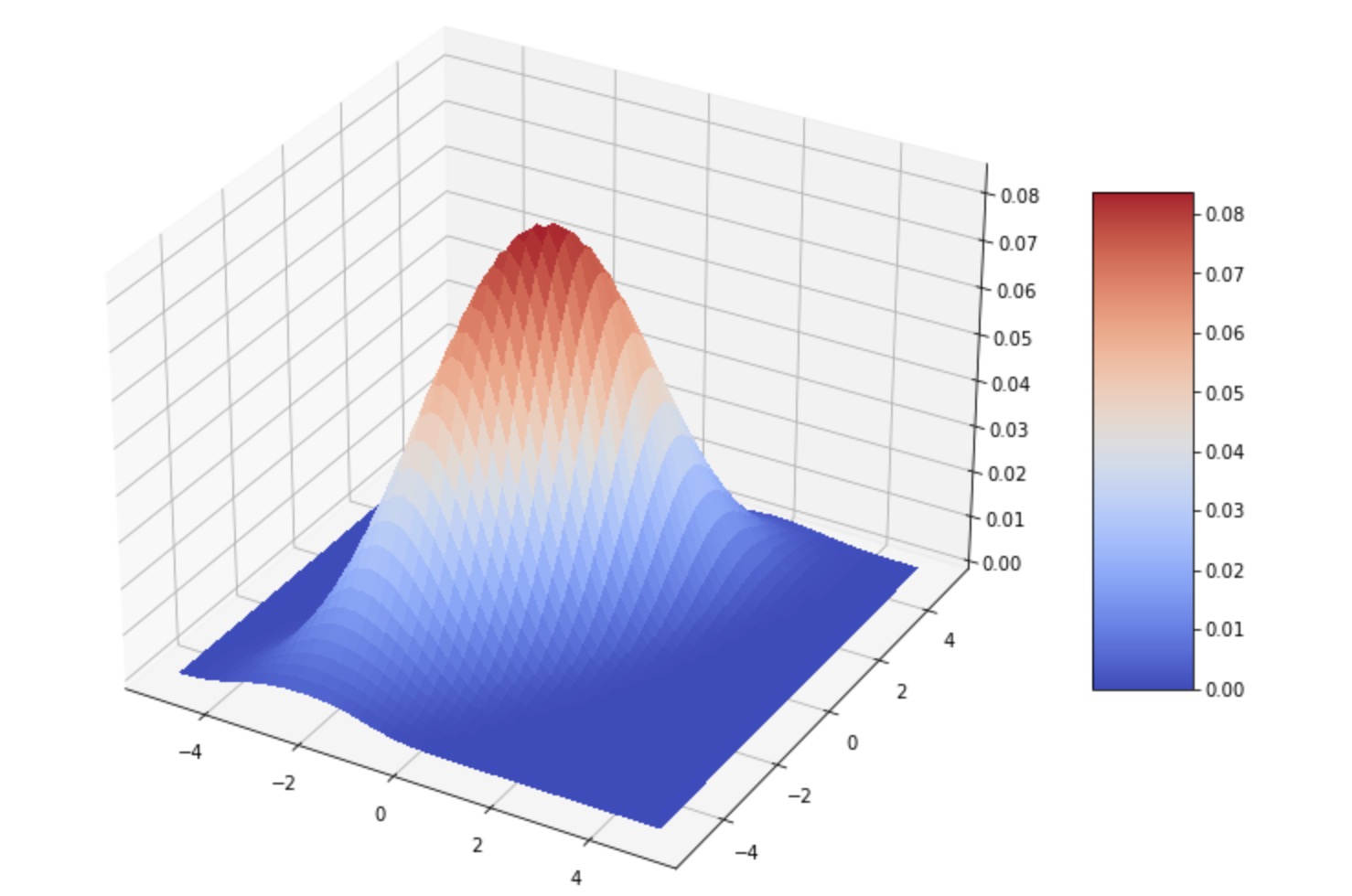}
  \caption{$\sigma = 1$}
\end{subfigure}
\caption{Probability densities of smoothed empirical measures with various $\sigma$. A large $\sigma$ leads to a smooth density, whereas spikes emerge for small $\sigma$. }
\label{fig:sig_smooth}
\end{figure}

We demonstrate effect (ii) using the same simple example of the martingale coupling $(B_1, B_2)$ described earlier. Each plot in Figure \ref{fig:error_reduction} illustrates the relationship between the number of Monte Carlo simulations (scaled by a factor of 100, shown on the $x$-axis) and the resulting martingale projection error (on the $y$-axis) for different values of $\sigma$. As shown, larger values of $\sigma$ lead to a faster reduction in the martingale projection error.

\begin{figure}[H]
\centering
  \includegraphics[width=1.1\linewidth]{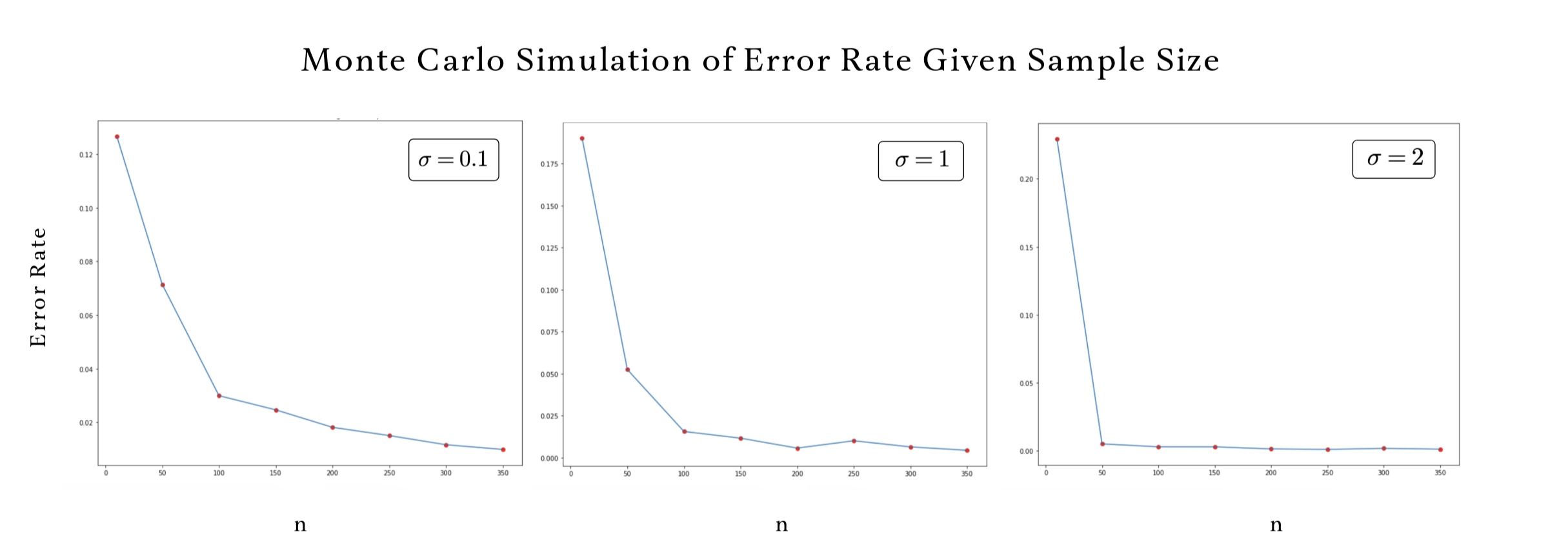}
\caption{Effects of $\sigma$ in reducing projection error}
\label{fig:error_reduction}
\end{figure}

Effect (iii) is supported by Theorem \ref{thm:expectation}, which states that, for a fixed significance level $\alpha$, the mean of the Gaussian random field integral decreases to order 1 as $\sigma$ increases. Consequently, the associated critical value $c_\alpha$ is also reduced. The augmented table below presents the simulated asymptotic critical values for $d = 1$ over 1,000 replications, with $\rho = 5$ and $\sigma = 0.01, 1, 100$, using Hermite polynomial couplings (as formulated in synthetic examples (ALT1) and (NULL2) in Section \ref{subsec:simulation}) of varying degree $k$.

\begin{table}[H] \centering \begin{tabular}{S[table-format=2.0] S[table-format=1.2] S[table-format=3.2] S[table-format=2.3] S[table-format=3.0] S[table-format=3.0] S[table-format=1.3] S[table-format=1.2]} \toprule {$k$} & {$\alpha$} & {$\sigma$} & {$c_\alpha$} & {$N$} & {$n$} & {$\bar{T}$} & {$p$} \\ \midrule 1 & 0.05 & 0.01 & 15.006 & 100 & 100 & 10.000 & 0.00 \\ 1 & 0.05 & 1.00 & 4.705 & 100 & 100 & 7.318 & 1.00 \\ 1 & 0.05 & 100.00 & 1.980 & 100 & 100 & 10.000 & 1.00 \\ 5 & 0.05 & 0.01 & 15.006 & 100 & 100 & 3.781 & 0.00 \\ 5 & 0.05 & 1.00 & 4.705 & 100 & 100 & 3.910 & 0.10 \\ 5 & 0.05 & 100.00 & 1.980 & 100 & 100 & 3.982 & 1.00 \\ 10 & 0.05 & 0.01 & 15.006 & 100 & 100 & 2.778 & 0.00 \\ 10 & 0.05 & 1.00 & 4.705 & 100 & 100 & 2.733 & 0.00 \\ 10 & 0.05 & 100.00 & 1.980 & 100 & 100 & 2.639 & 0.98 \\ 15 & 0.05 & 0.01 & 15.006 & 100 & 100 & 2.106 & 0.00 \\ 15 & 0.05 & 1.00 & 4.705 & 100 & 100 & 2.183 & 0.00 \\ 15 & 0.05 & 100.00 & 1.980 & 100 & 100 & 2.078 & 0.45 \\ 20 & 0.05 & 0.01 & 15.006 & 100 & 100 & 1.813 & 0.00 \\ 20 & 0.05 & 1.00 & 4.705 & 100 & 100 & 1.804 & 0.01 \\ 20 & 0.05 & 100.00 & 1.980 & 100 & 100 & 1.747 & 0.08 \\ \bottomrule \end{tabular}
\caption{Simulation results for Hermite couplings with $\rho = 5$ and various $\sigma$}
\label{tab:my-table}
\end{table}

Observe that the choice of $\sigma$ significantly affects the power of the test. Generally, smaller values of $\sigma$ result in a more ``lenient'' test, increasing the chance of committing a Type I error. In contrast, larger values of $\sigma$ make the test ``stricter'', thus increasing the chance of committing a Type II error. This trade-off is illustrated in Figure \ref{fig:sigma_power}, which shows the empirical rejection rate, computed over 10,000 trials, plotted against different values of $\sigma$ for the non-martingale example using Hermite polynomial couplings with $k = 5$.

\begin{figure}[H]
\centering
\begin{subfigure}[b]{.5\textwidth}
  \centering
  \includegraphics[width=\linewidth]{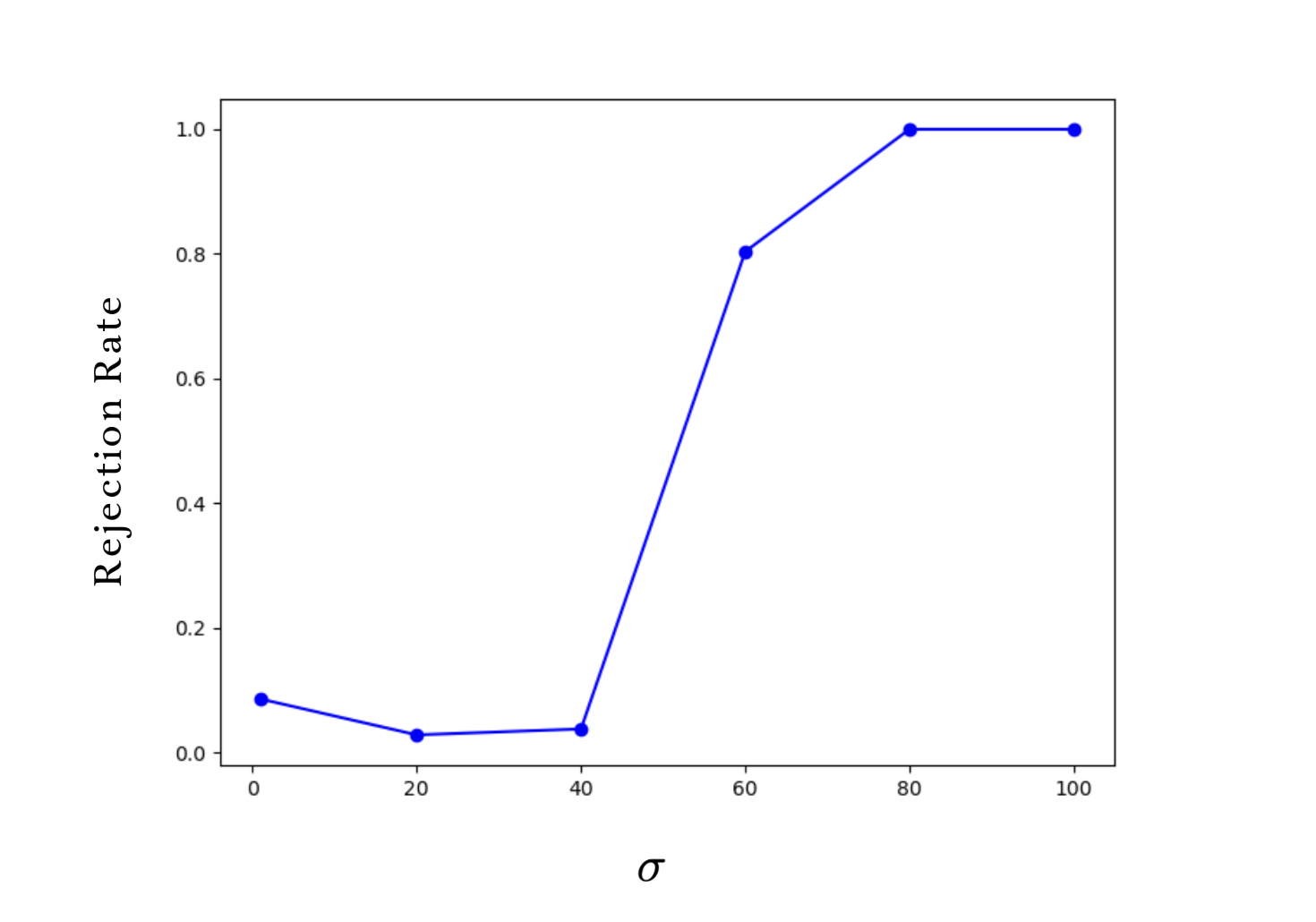}
  \caption{rejection rate vs $\sigma$}
\end{subfigure}%
\hfill
\begin{subfigure}[b]{.5\textwidth}
  \centering
  \includegraphics[width=\linewidth]{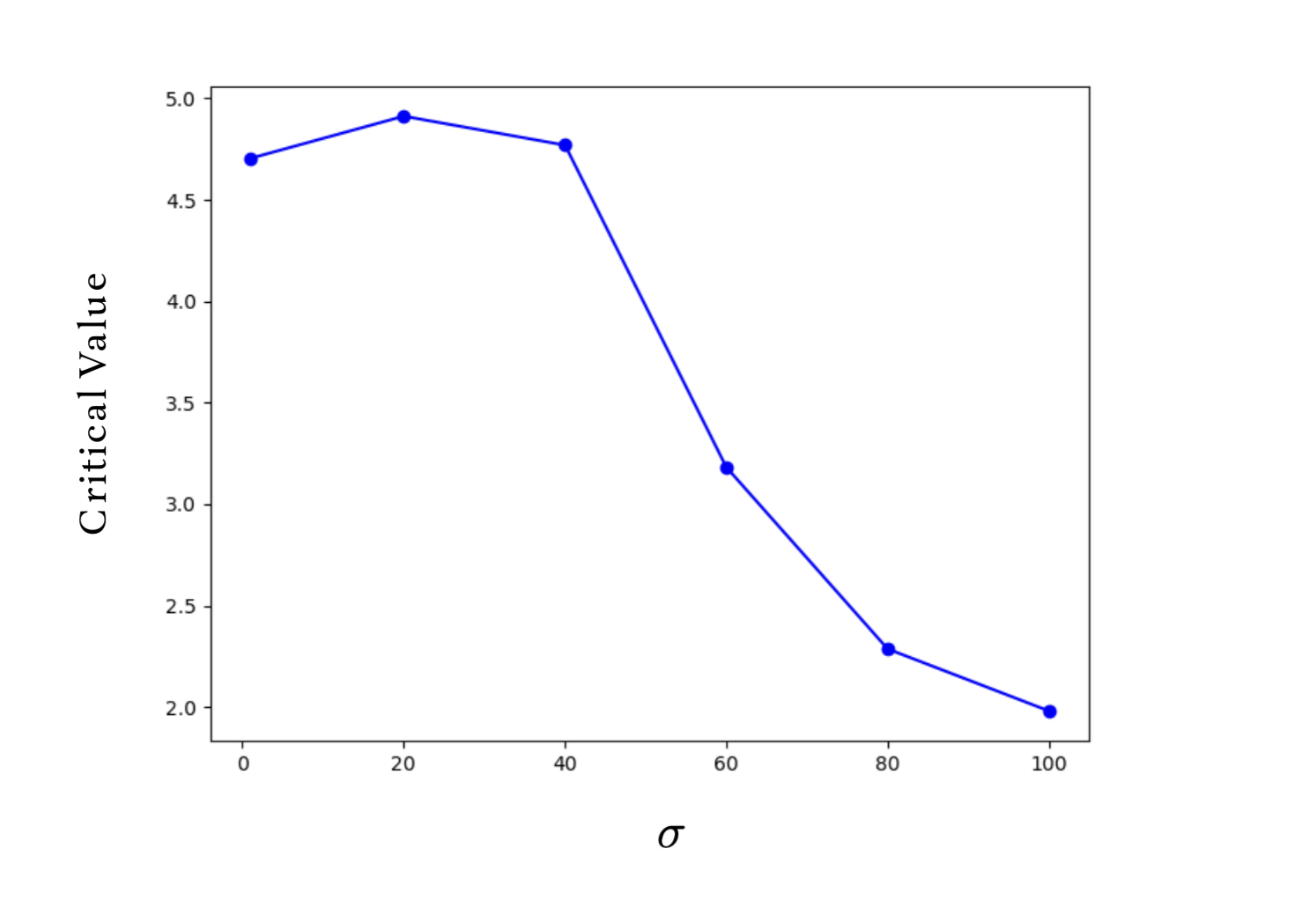}
  \caption{critical value vs $\sigma$}
\end{subfigure}
\caption{Effects of $\sigma$ on the power of test. Plot (a) shows the rejection rate (the proportion of tests that reject the null) generally increases in $\sigma$; plot (b) shows that the critical value with significance level $\alpha=0.05$ generally decreases in $\sigma$. The plots show that a smaller $\sigma$ induces more Type I error and a larger $\sigma$ induces more Type II error.}
\label{fig:sigma_power}
\end{figure}

As shown in Figure \ref{fig:sigma_power}, while the power of the test decreases slightly from $\sigma = 1$ to $\sigma = 20$, it increases consistently as $\sigma$ increases from $20$, with the biggest improvement occurring at $\sigma = 60$. This is consistent with our discussion above.

\subsubsection{Power curves for \texorpdfstring{$d=1$}{}}
In this section, we present the power curve for a series of perturbed martingale couplings to empirically study the power of the martingality test. For simplicity, we fix $\rho = 5$ and $\sigma = 80$ for the tests. We consider two examples.
\begin{enumerate}[(i)]
\item Let $X \sim \mathcal{N}(0,1)$, $Z \sim \mathcal{N}(0,1)$, $Y= X+Z+\epsilon$, where $\epsilon \in \{-1,-0.75,-0.5,-0.25,0,0.25, 0.5,0.75,1\}$.
\item $X \sim \mathcal{N}(0,1)$, $Y = X+H_k(X)/\sqrt{k!}$, where $k = 1, 4, 7, 10, 13, 16, 19, 22, 25$ represent different perturbation levels.\footnote{Note that for Hermite polynomials, a smaller $k$ indicates a larger perturbation as the couplings become further from the space of martingales.}
\end{enumerate}

Using 1000 observations and a replication size of 1,000, we generate the power curves shown in Figure \ref{fig:model1} for Example (i) and Figure \ref{fig:model2} for Example (ii).

\begin{figure}[H]
\centering
\begin{subfigure}[b]{.5\textwidth}
  \centering
  \includegraphics[width=\linewidth]{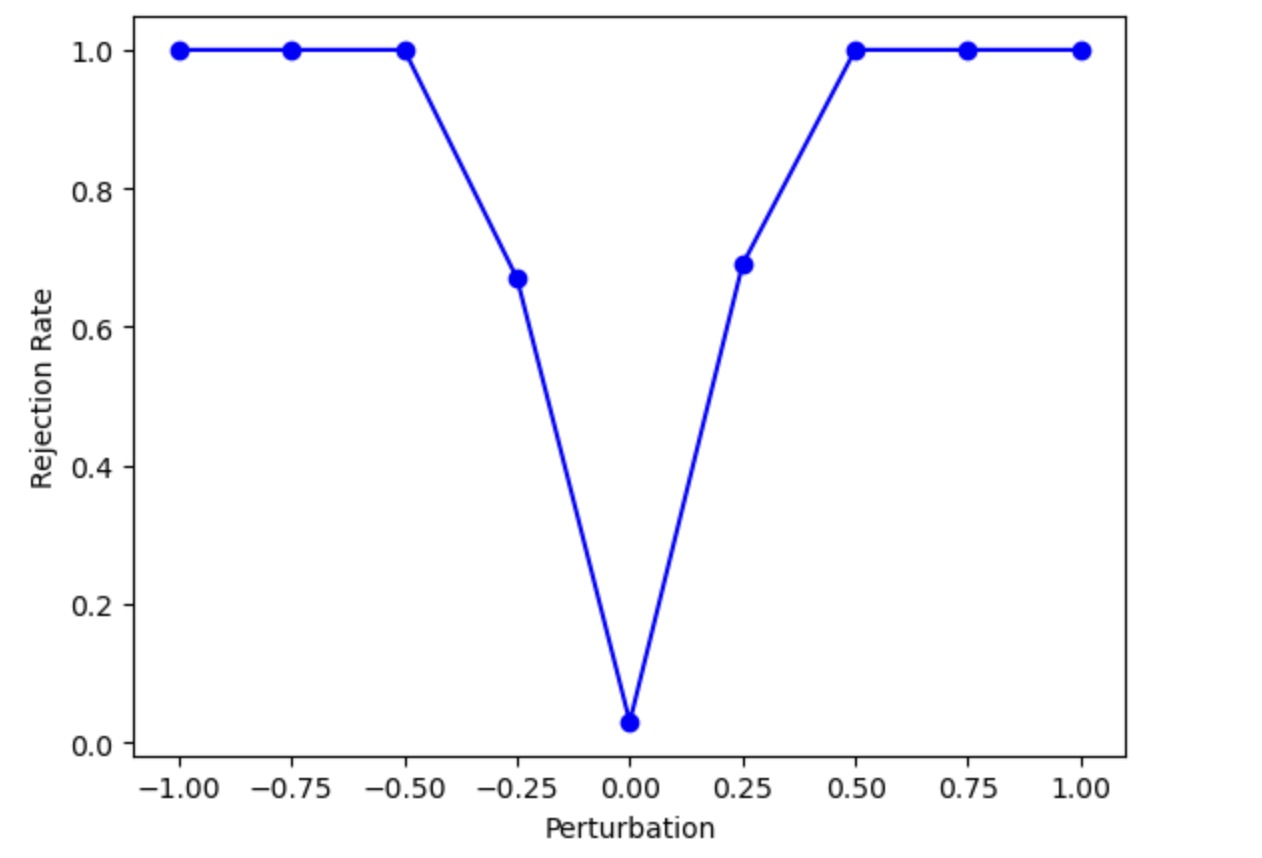}
  \caption{Rejection rate vs $\epsilon$}
\end{subfigure}%
\hfill
\begin{subfigure}[b]{.5\textwidth}
  \centering
  \includegraphics[width=\linewidth]{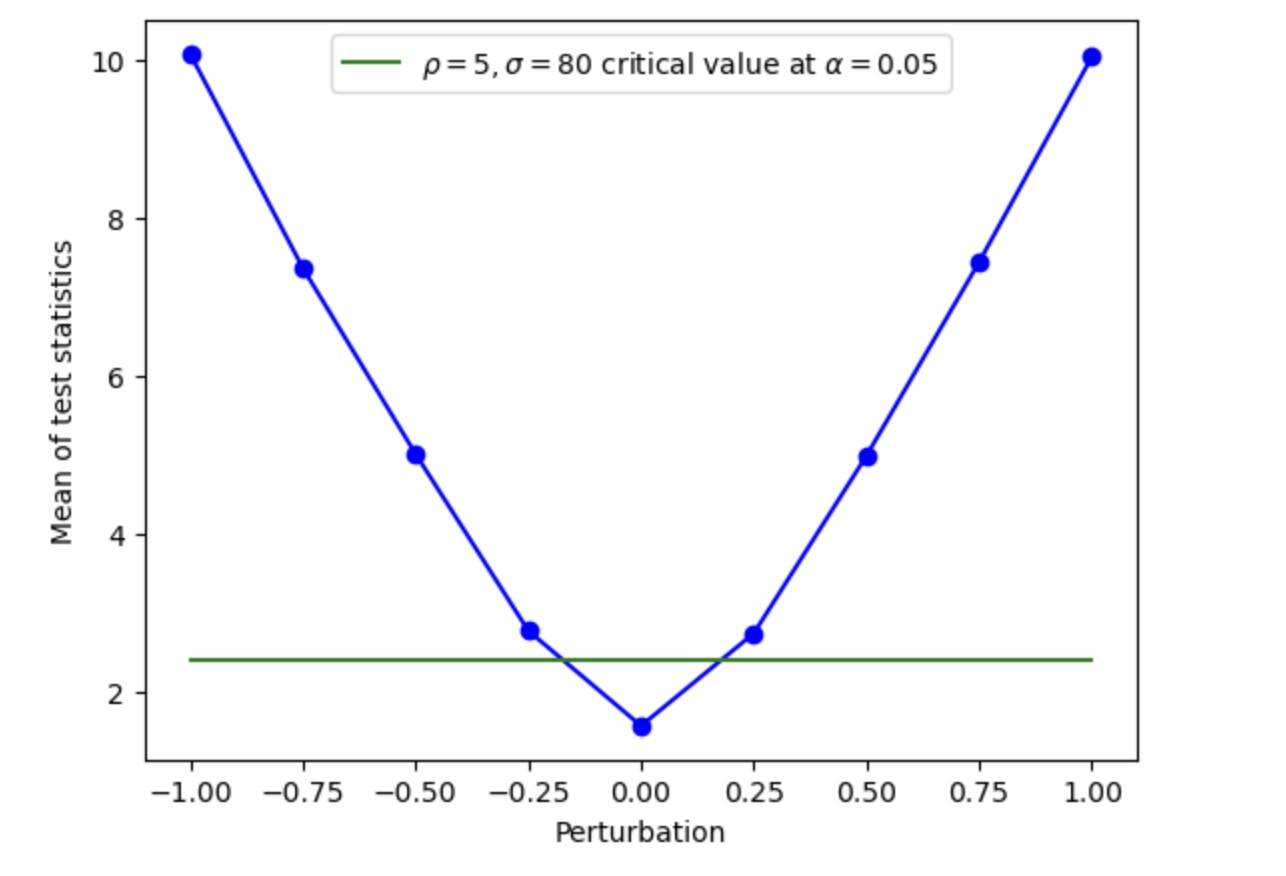}
  \caption{Mean test statistics vs $\epsilon$}
\end{subfigure}
\caption{Effects of the perturbation level $\epsilon$ on the empirical rejection rate and the mean of the test statistics over $1000$ trials in Example (i), with test conducted using $\rho = 5, \sigma = 80$ at a significance level of $\alpha = 0.05$. The green line in plot (b) represents the critical value of the test at $\alpha = 0.05$. Therefore, the test accepts any test statistics below the green line, while rejects any test statistics above the green line. The graphs show that a larger magnitude of the perturbation yields a larger mean of the test statistic and a larger rejection rate.}
\label{fig:model1}
\end{figure}

\begin{figure}[H]
\centering
\begin{subfigure}[b]{.5\textwidth}
  \centering
  \includegraphics[width=\linewidth]{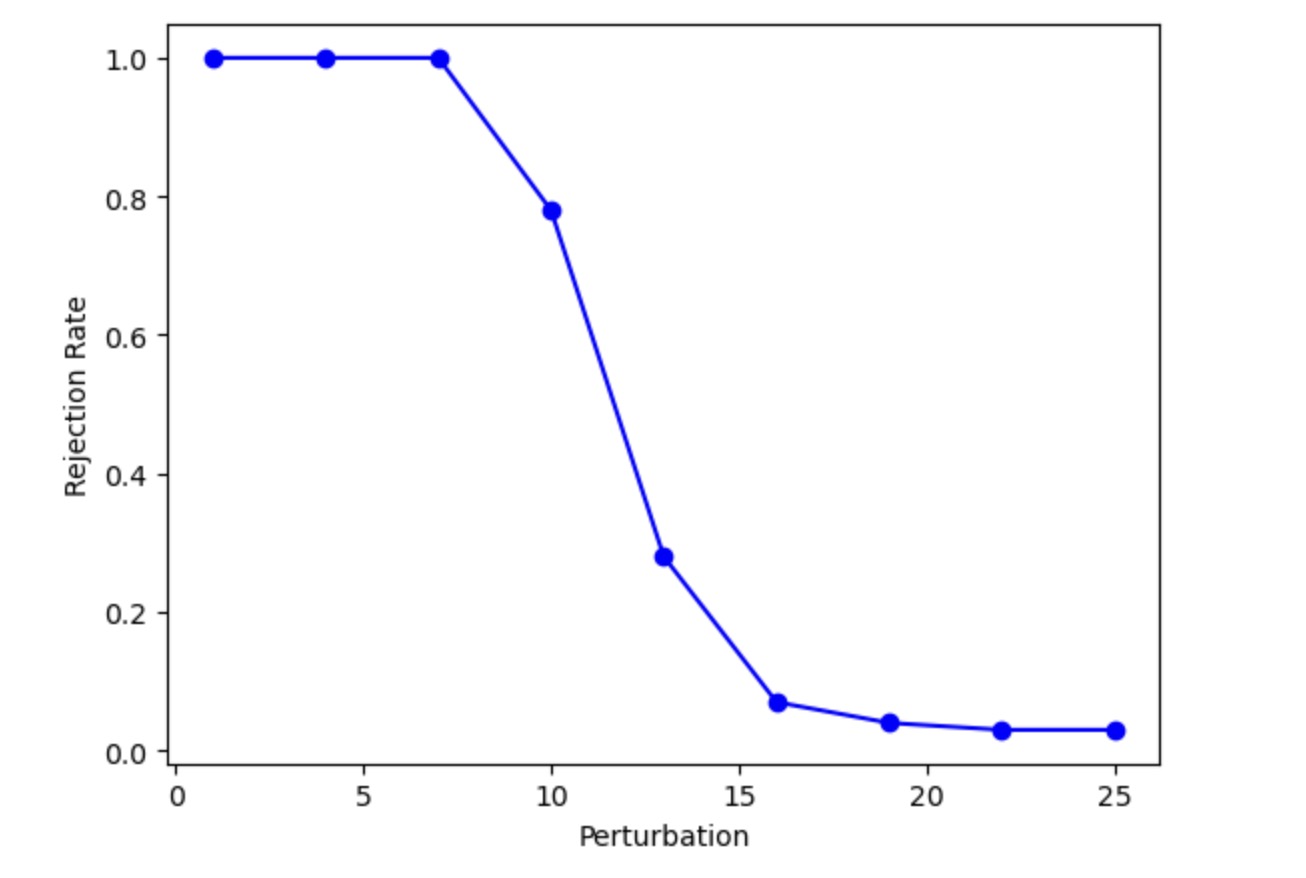}
  \caption{Rejection rate vs $k$}
\end{subfigure}%
\hfill
\begin{subfigure}[b]{.49\textwidth}
  \centering
  \includegraphics[width=\linewidth]{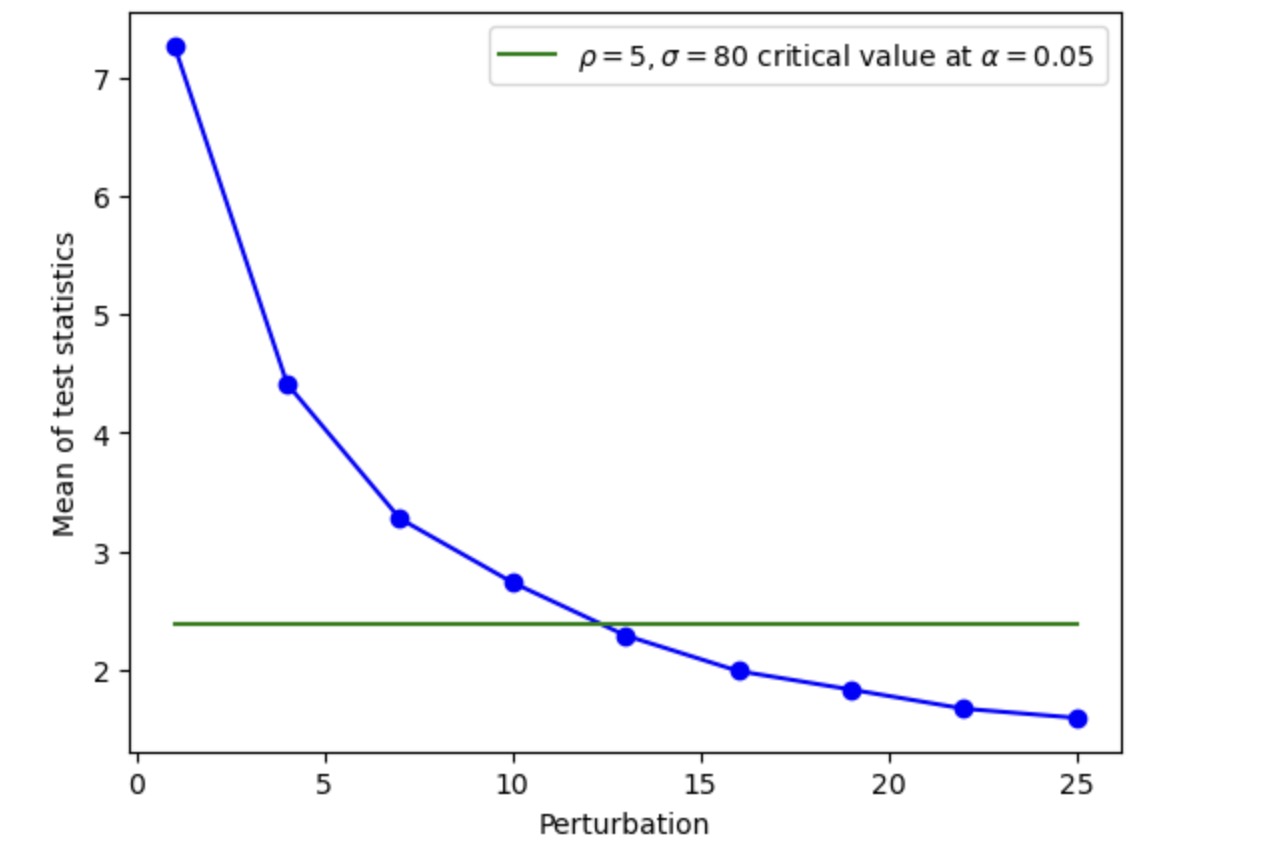}
  \caption{Mean test statistics vs $k$}
\end{subfigure}
\caption{Effects of the perturbation level $k$ on the empirical rejection rate and the mean of the test statistics over $1000$ trials in Example (ii), where $\rho = 5, \sigma = 80$ at a significance level of $\alpha = 0.05$. The green line in plot (b) represents the critical value of the test at $\alpha = 0.05$. A smaller $k$ leads to a larger mean of the test statistic and a larger rejection rate.}
\label{fig:model2}
\end{figure}

\subsection{Convergence Analysis}\label{subsec:convergence_analysis}

In this subsection, we compare the performance of $\mathrm{MPD}^{*\xi}(\mathbb{P}_n, \gamma)$ with the plugin estimator $\mathrm{MPD}(\hat{\mathbb{P}}_n, \gamma)$, where $\hat{\mathbb{P}}_n$ denotes the adapted empirical measure of \citep{backhoff2021estimatingprocessesadaptedwasserstein}.

Let us first recall that the map $\mathbb{P} \mapsto \operatorname{MPD}(\mathbb{P}, \gamma )^{1/\gamma}$ is Lipschitz wrt. $\mathcal{AW}_\gamma$. In consequence, $\mathrm{MPD}(\hat{\mathbb{P}}_n, \gamma)$ is a consistent estimator. However, its convergence rate may depend on the dimension. Indeed, assuming that $\P_0$ has Lipschitz kernels, \citep[Theorem 1.5]{backhoff2021estimatingprocessesadaptedwasserstein} states that $\mathcal{AW}_\gamma(\P_0,\hat{\P}_n)$ is bounded from above by
\begin{align}
\text{rate}(n) = 
\begin{cases}
n^{-1/3} & \text{for } d = 1; \\
n^{-1/4} \log(n + 1) & \text{for } d = 2; \\
n^{-1/(2d)} & \text{for } d \geq 3.
\end{cases}
\end{align}
In consequence, the convergence of $\mathrm{MPD}(\hat{\mathbb{P}}_n, \gamma)$ may suffer from the curse of dimensionality. In contrast, our estimator $\mathrm{MPD}^{*\xi}(\mathbb{P}_n, \gamma)$ achieves a dimension-free convergence rate of the order $O(n^{-1/2})$ under mild assumptions.


To validate our claim empirically, we first note that $\mathrm{MPD}(\hat{\mathbb{P}}_n, \gamma)$ can be computed as follows: using the definition of the adapated empirical measure \citep[Definition 1.2]{backhoff2021estimatingprocessesadaptedwasserstein} and given samples $\{(X_i, Y_i)\}_{i=1}^N \subset [0,1]^d \times [0,1]^d$, let $\varphi^N : [0,1]^d \to [0,1]^d$ be a map that sends each $X_i$ to the center of a cube in a uniform partition of $[0,1]^d$ into $N^{rd}$ sub-cubes. We then have
\begin{align}\label{eq:adapted_mpd}\begin{split}
\mathrm{MPD}(\hat{\mathbb{P}}_n, \gamma) = 2^{1 - \gamma} \sum_{g \in \mathcal{G}} &\frac{|\{j : \varphi^N(X_j) = g\}|}{N} \\
&\cdot \left| g - \frac{1}{|\{j : \varphi^N(X_j) = g\}|} \sum_{j : \varphi^N(X_j) = g} \varphi^N(Y_j) \right|_2^\gamma,
\end{split}
\end{align}
where:
\begin{itemize}
    \item $\mathcal{G} = \{\varphi^N(X_i) : i = 1,\dots,N\}$ is the set of grid centers,
    \item  $|\{i : \varphi^N(X_i) = g\}|$ samples are mapped to each $g \in \mathcal{G}$.
\end{itemize}

We call $\mathrm{MPD}(\hat{\mathbb{P}}_n, \gamma)$ the \emph{adapted empirical MPD.}

Next, we evaluate both estimators over a range of sample sizes $n \in [10^2, 10^4]$ in dimensions $d \in \{ 1,3 \}$. The data are generated by sampling $ X $ and $ Z $ independently from the uniform distribution on $ [-0.5, 0.5]^d $, and setting $ Y = X + Z $, so that $(X,Y)$ forms a martingale pair.

Figure~\ref{fig:loglog_comparison} presents log-log plots comparing the convergence of both estimators for $ d = 1 $ and $ d = 3 $. Reference curves corresponding to the theoretical convergence rates---$ O(n^{-1/2}) $ for the smoothed MPD and the corresponding dimension-dependent rates for the adapted MPD---are included for visual comparison.

As illustrated, the smoothed empirical MPD achieves the dimension-free $ O(n^{-1/2})$-rate across both settings. In contrast, the convergence of the adapted empirical MPD is considerably slower in dimension three. 

\begin{figure}[htbp]
    \centering
    \begin{subfigure}[b]{0.48\linewidth}
        \centering
        \includegraphics[width=\linewidth]{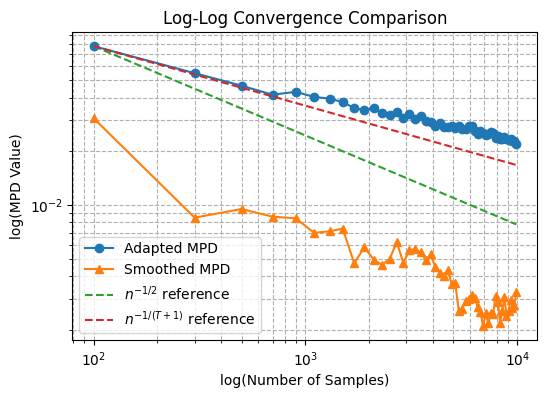}
        \caption{The case of $d=1$.}
        \label{fig:adapted_comp_loglog_d1}
    \end{subfigure}
    \hfill
    \begin{subfigure}[b]{0.48\linewidth}
        \centering
        \includegraphics[width=\linewidth]{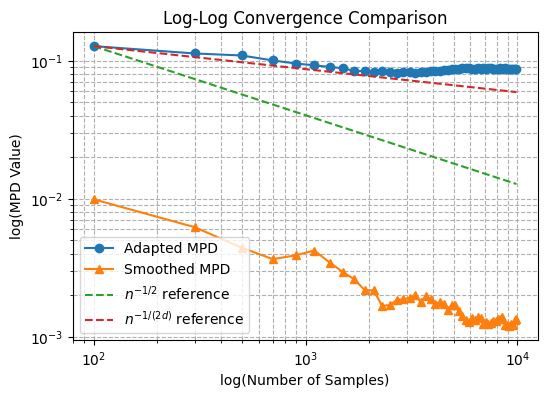}
        \caption{The case of $d=3$.}
        \label{fig:adapted_comp_loglog_d3}
    \end{subfigure}
    \caption{Log-log convergence comparison between the adapted MPD and the smoothed MPD for $d = 1$ and $d = 3$ averaged across 10 sampling trials. Reference curves for the theoretical convergence rates $n^{-1/(2d)}$ and $n^{-1/2}$ (dashed) are included for comparison.}
    \label{fig:loglog_comparison}
\end{figure}

In summary, $\hat\P_n$ yields valid plugin estimators for \emph{any} $\mathcal{AW}_\gamma$-continuous functional. However, its performance can be significantly improved for the specific task of estimating the martingale projection distance.

\subsection{Applications}\label{sec:apps}

Our hypothesis test for martingality provides valuable information in a wide range of areas of interest. For instance, it is well-known that martingales form an important pillar in financial economics and econometrics. No-arbitrage conditions are equivalent  existence of a probability measure under which discounted price processes follow martingale dynamics. Our results can therefore be used  to test the no-arbitrage hypothesis in generative AI models, as we shall illustrate in Section \ref{sec:sde} below.

In addition, another classical problem in econometrics and statistics consists of testing if a real-valued data set follows a given continuous distribution. The Kolmogorov-Smirnov statistic is a well-known non-parametric approach to testing this hypothesis. A natural generalization of this problem consists of testing if a given positive recurrent and irreducible general state-space Markov chain $\left\{  U_{n}\right\}_{n\geq0}$, with a known stationary distribution $\varphi$, follows the transition kernel $\left\{  K\left(  x,\cdot\right)  :x\in S\right\}$. This
is true if and only if for every continuous and bounded function $v\left(
\cdot\right)$ we have $\left(X_{n},Y_{n}\right) =\left(   u_{n},  
u_{n}  +v\left(  u_{n+1}\right)  -\left(  Kv\right)  \left(
u_{n}\right)  \right)$ forms a martingale pair for $\varphi$-a.e.~$u_{n}$. Therefore, this hypothesis can be tested by selecting a family of
functions $v_{1},\dots,v_{d}$ and testing the martingale property for the pair
of $d$-dimensional vectors $\left(  X_{n},Y_{n}\right)  \in \R^{d}\times \R^{d}%
$, where $X_{n}\left(  i\right)  =U_n\lawis\varphi $ and $Y_{n}\left(
i\right)  =U_n  +v_i\left(  U_{n+1}\right)  -\left(
Kv_{i}\right)  \left(  U_{n}\right)$.\footnote{The choice of $v_i$'s may depend on $K$. This is beyond the scope of our focus here.} 

The above two applications of our test will be respectively detailed in Section \ref{sec:sde}  and Appendix \ref{sec:markov}. 
There are numerous other applications of martingale pair tests in the sphere of finance, econometrics, reinforcement learning, and non-parametric regression. 

\subsubsection{Testing no-arbitrage in neural SDE-based European option calibration}\label{sec:sde}
As mentioned briefly above, one application of our results is a test for arbitrage opportunities in existing pricing models for financial derivatives. In the following, we first describe the set-up of the financial market considered and then outline our application.

\textbf{Market model.} The work of \citep{gierjatowicz2020robust}  develops a neural SDE-based European option calibration method. In their set-up, the true dynamics of the wealth process $X=\{X_t\}_{t\geq 0}$ under the risk-neutral measure $\mathbb{Q}=\mathbb{Q}(\theta)$ are given by 
\begin{align}\label{eq:wealth_sde}
\d X^\theta_t = b(t,X^\theta_t,\theta)\d t +\sigma(t,X^\theta_t,\theta)\d W_t
\end{align}
for an $n$-dimensional Brownian motion $W$, a drift function $b: \mathbb{R}_+\times \mathbb{R}^d \times \Theta \rightarrow \mathbb{R}^d$ and a diffusion function $\sigma: \mathbb{R}_+\times \mathbb{R}^d \times \Theta \rightarrow \mathbb{R}^{d\times n}$, where $\Theta \subseteq \mathbb{R}^p$ is a parameter space for some $p\in\N$. To be more specific, \citep{gierjatowicz2020robust} considers a decomposition of $X$ into (\ref{eq:wealth_sde}) into the \textit{tradable} (e.g., stocks) and \textit{non-tradable} components $S$ and $V$: 
\begin{align}\label{eq:sde}
\begin{split}
\d S^\theta_t &= rS^\theta_t\d t+\sigma^S(t,X^\theta_t,\theta)\d W_t ;\\
\d V^\theta_t &= b^V(t,X^\theta_t,\theta)\d t+\sigma^V(t,X^\theta_t,\theta)\d W_t ;\\
X^\theta_t &= (S^\theta_t,V^\theta_t).
\end{split}
\end{align}
From (\ref{eq:sde}), we define $b(t,(s,v),\theta) := (rs,b^V(t,(s,v),\theta))$ and $\sigma(t,(s,v),\theta) := (\sigma^S(t,(s,v),\theta),\sigma^V(t,(s,v),\theta))$. Then the corresponding neural SDE model is obtained simply by parameterizing $(b,\sigma)$ as feed-forward neural networks.

\textbf{Calibration.} Having set up the market model of \citep{gierjatowicz2020robust}, we now describe the authors' proposed calibration method for European options, which are derivatives of the underlying tradable asset (stocks). To simplify notation, we use the notation $\theta, \theta^*, \hat{\theta}$ to denote a generic parameter, the true parameter, and the approximated value of the true parameter through calibration. \citep{gierjatowicz2020robust} assumes that the call option prices at time zero
\begin{align}\label{eq:authorformula}
\mathfrak{p}(K,T):=e^{-r T} \mathbb{E}^{\mathbb{Q}(\theta^*)}\left[\left(S_T-K\right)_{+} \mid S_0=1\right]
\end{align}
are given.
They then calibrate the stock price SDE in \eqref{eq:sde} via \eqref{eq:authorformula}, i.e., through finding $\hat\theta$ such that
\begin{align}\label{eq:author_calibration}
e^{-r T} \mathbb{E}^{\mathbb{Q}(\hat{\theta})}\left[\left(S_T-K\right)_{+} \mid S_0=1\right]\approx \mathfrak{p}(K,T).
\end{align}

Importantly, \citep{gierjatowicz2020robust} calibrates solely to call prices at time zero, disregarding additional market data. However, in practice, it is more realistic to assume that call prices can be calibrated dynamically, taking into account the stock price process $(S_t)_{t \in \{0, \dots, T\}}$. The corresponding option prices at each time $t$ are
\begin{align}\label{eq:pd}
\mathfrak{p}_t(K,T,S_t) := e^{-r(T-t)} \mathbb{E}^{\mathbb{Q}(\theta^*)}\left[\left(S_T - K\right)_{+} \mid S_t\right].  
\end{align}
Armed with our martingale pair test, we will check if the calibration procedure of \citep{gierjatowicz2020robust} in (\ref{eq:author_calibration}) is \textit{consistent} (in the sense of conforming to the true price model in (\ref{eq:pd})) with the additional prices $(S_t)_{t \in \{0, \dots, T\}}$ observed. In other words, given the calibrated parameter $\hat\theta$ found through calibration model in (\ref{eq:author_calibration}) we investigate the following question: $$\text{Does} \ \ \mathfrak{p}_t(K,T,S_t)\approx e^{-r (T-t)} \mathbb{E}^{\mathbb{Q}(\hat{\theta})}\left[\left(S_T-K\right)_{+} \mid S_t\right]?$$ 

\textbf{Martingale pair testing.} The above-stated problem is equivalent to testing if $$(\mathfrak{p}_t(K,T,S_t),e^{-r (T-t)} \left(S_T-K\right)_{+})$$ 
is a martingale coupling under the calibrated measure $\mathbb{Q}(\theta^*)$ obtained from procedure in (\ref{eq:author_calibration}). To tackle this problem, we carry out the following three steps:
\begin{enumerate}
\item Calibrate asset prices for each time-step following the algorithms of \citep{gierjatowicz2020robust}. Obtain $N$ stock trajectories $\{S_t\}_{t \in \{0,\dots,T\}}$ under the calibrated measure $\mathbb{Q}(\hat\theta)$.
\item Given a calibrated stock trajectory $\{S_t\}_{t \in \{0,\dots,T\}}$ from step one, use Monte Carlo simulation to obtain prices of vanilla options $\mathfrak{p}_t(K,T,S_t)$ at each $t \geq 0$ using \eqref{eq:pd}.
\item Apply the martingale pair test to check if $(\mathfrak{p}_t(K,T,S_t),e^{-r (T-t)} \left(S_T-K\right)_{+})$ is a martingale coupling under $\mathbb{Q}(\hat\theta).$
\end{enumerate}

\textbf{Implementation.} We now briefly discuss our implementation for the abovementioned martingale test; see \href{https://github.com/Ericavanee/Bicausal_Wasserstein_MtglProj}{GitHub} for the full implementation details. The work of 
\citep{gierjatowicz2020robust} uses two market models for calibration: the local volatility model (LV) and the local stochastic volatility model (LSV). For step one, we alter the training algorithm of \citep{gierjatowicz2020robust} for both LSV and LV models to return $N = 4000$ calibrated stock trajectories $\{(S_t)_{t \in \{0,\dots,T\}}^i\}_{i = 1}^N$ directly. For step two, for each stock trajectory $\{S_t\}_{t\in \{0,...,T\}}$, we use Monte Carlo simulation to generate $n = 1000$ asset price paths using a Heston model (which is assumed to describe the true market dynamics) approximated via a tamed Euler scheme at each time point $t\in \{0,...,T\}$:
\begin{align}\label{eq:heston}
\begin{split}
\d S_t &= rS_t\d t+S_t\sqrt{V_t}\d W_t, \ \ S_0 = s_0 \\
\d V_t &= \kappa(\mu-V_t)\d t+\eta\sqrt{V_t}\d B_t, \ \ V_0 = v_0 \\
\d\langle B,W \rangle_t &= \rho \d t.
\end{split}
\end{align}
We use the same set of parameters as \citep{gierjatowicz2020robust} for the assumed true market dynamics model: $\theta^* = \{x_0 = 1, r = 0.025, V_0= 0.04, \kappa = 0.78, \mu = 0.11, \eta = 0.68, \rho = 0.044\}$. We then calculate the associated discounted European option prices with each maturity $T$ and strike $K$ via (\ref{eq:pd}).
The pseudocode for this algorithm is deferred to Appendix \ref{appn}.

For the final step, we conduct a martingale pair test of the coupling $(\mathfrak{p}_t(K,T,S_t),e^{-r (T-t)} \left(S_T-K\right)_{+})$ fixing $\sigma = 1, \rho = 5$ and a significance level of $\alpha = 0.05$. We adapt the testing procedures outlined in Algorithm \ref{algo:mtgl_pair_test} in Appendix \ref{appn}.

\textbf{Results.} We find that, for either LV-model-based calibration or LSV-model-based calibration, $(\mathfrak{p}_t(K,T,S_t),e^{-r (T-t)} \left(S_T-K\right)_{+})$ does not form a martingale pair. For the LV model and the LSV model, the test statistics are \texttt{17.242} and \texttt{11.714} respectively, against an $\alpha = 0.05$ critical value $c_{\alpha} =$ \texttt{4.705}.
In conclusion, \citep{gierjatowicz2020robust}'s calibration method is shown to be inconsistent with the market data available. 

We also observe that one of \citep{gierjatowicz2020robust}'s key contributions, i.e., using hedging strategy as a control variate for the calibration model, may fail to work when we examine option prices as a function of the stock price observed at each time point. 
Instead, to avoid creating arbitrage opportunities in the true option price model in \eqref{eq:pd}, we propose the following neural SDE-based option calibration method: as before, the market data is represented by discounted payoffs $\{{e^{-r(T^{j}-t)}(S_{T^j}-K^{j})_{+}}\}_{j=1}^M$ of call options for pairs $\{(K^j,T^j)\}_{j=1}^M$. We obtain their corresponding market prices $\{\mathfrak{p}_t(K^j,T^j)\}_{j=1}^M$ from \eqref{eq:pd}. We then replace the loss function in \citep[Algorithm 1]{gierjatowicz2020robust} by the  martingale projection loss criterion
\begin{align*}
    \theta^*\in \argmin_{\theta\in \Theta} \sum_{j=1}^M  \sum_{i=1}^N \int \Big( (\mathfrak{p}_t({K^j,T^j},S^\theta_{t_i})- &{e^{-r(T^j-t)}(S^\theta_{T^j}-K^j)_{+}}) \\
    &\cdot f_{\xi,\rho}(x-{{e^{-r(T^j-t)}(S^\theta_{T^j}-K^j)_{+}}})\Big)^2\, \d x,
\end{align*}
see Lemma \ref{lemma:computation}.
Pseudo-code  of the tentative new algorithm can be found in {Appendix} \ref{appn}.

\begin{appendix}

\section{Testing concurrence of a Markov chain with given transition kernel}\label{sec:markov}

We consider the problem of testing if an ergodic sequence follows a particular
Markov chain dynamics. This problem is the analog to the problem of testing if an i.i.d.~sequence follows a particular distribution. In the one-dimensional i.i.d.~setting, the Kolmogorov-Smirnov test provides a well-known approach. 

Precisely, we are interested in testing if a $\varphi$-irreducible and positive recurrent
Markov chain sequence $\left\{  U_{n}:n\geq0\right\}  $ taking values in the state-space $S$ (e.g.,~the support of $\varphi$, which may be assumed to be a maximal irreducibility measure) follows a particular
transition kernel, $\left\{  K\left(  z,\cdot\right)  :z\in S\right\}  $. Assuming for simplicity that $S\subseteq\R^m$, this
is true if and only if for all continuous and bounded functions $v:\R^m \to \R^m $ we have that 
$$
\left(  X_{n},Y_{n}\right)  =\left( u_n  ,u_n +v\left(  u_{n+1}\right)  -\left(  Kv\right)  \left(
u_{n}\right)  \right)
$$
forms a martingale pair for $\varphi$-a.e.~$u_{n}$. 
Indeed, if the ergodic chain satisfies this condition we have that for all
continuous and bounded functions $v:\R^m\to \R^m$,
$$
\left(  Kv\right)  \left(  u_{n}\right)  =E\left[  v\left(  U_{n+1}\right)
|U_{n}=u_n\right]
$$
almost everywhere with respect to the stationary measure which is a maximal
irreducibility measure (see Theorems 10.0.1 and 10.1.2 in \citep{meyn2009}). 

As an application of our results in this paper, we can select a family of
continuous and bounded functions $v_{1},\dots,v_{d}:\R^m\to \R^m$ so we can test the martingale property for the pair
of $d$-dimensional vectors $\left(  X_{n},Y_{n}\right)  \in \R^{md}\times \R^{md}%
$, where
\begin{align}
    X_{n}\left(  i\right)  =U_n  ,\text{ }Y_{n}\left(
i\right)  =U_n  +v_i\left(  U_{n+1}\right)  -\left(
Kv_{i}\right)  \left(  U_{n}\right)  .\label{eq:UUU}
\end{align}
To put the discussion into context (with $m=1$ and $d=2$ for simplicity), we consider a simple Gaussian Markov process and an example inspired by the present value process of perpetual cash flow as described in Example 2.2 of \citep{gjessing1997present}.

\begin{example}[Gaussian Markov Process]\label{ex:gaussian_markov}
Consider the simple case of an infinite state space Gaussian Markov Process as the following:
$$U_{n+1} = \kappa U_{n}+\xi_{n+1},$$ where $\kappa \in [0,1]$ and $\xi_{n+1} \sim \mathcal{N}(0,1)$ is iid. Gaussian noise. 

Choosing $v_1(x) = x$ and $v_2(x) = x\mathds{1}_{\{x>0\}}$, we generate $\{(X_n(i),Y_{n}(i))\}_{n=1}^{1000}$ with $\kappa = 0.5$ for $i=1,2$ from \eqref{eq:UUU}. For a martingale pair test with parameters $\{\rho = 5, \sigma = 1, \alpha = 0.05\}$, the loss of the series of couplings is \texttt{3.091e-23} against an asymptotic critical value $c_\alpha=\texttt{4.840}$. Hence, the test correctly accepts the series $\{(X_n,Y_{n}\}_{n=1}^{1000}$ as a martingale with 95\% confidence. 
\end{example}

\begin{example}[Adapted present value process of perpetual cash flow]\label{ex:adapted_perpetual_cash}
Consider the Markov chain
\begin{align}\label{formula:markov}
 {U}_n\ = e^{-rn} \Big(U_0+ \int_{0}^{n} e^{rs}\d {P_s}\Big)
\end{align}
where $n\in\N_0$, $r>0$, $U_0 = 0$, and $ {P}_t = \sum_{i = 1}^{N_{P,t}}S_{P,i}$ is a non-negative compound Poisson process with $N_{P,t} \sim \text{Pois}(\lambda_P)$ and $S_{P,i} \sim \Gamma(\alpha,\beta)$. Choose $r = 1$, $\lambda_P = 2$ and $(\alpha,\beta) = (2,3)$, where $\alpha$ denotes the location parameter and $\beta$ denotes the scale parameter.

To generate $\{U_n\}_{n\in \N}$, we use the observation that 
\begin{align*}
U_{n+1} = U_n + \int_{0}^{1} e^{r(s-n)}\d P_{n+s}.
\end{align*}

Choosing $v_1(x) = x$ and $v_2(x) = x\mathds{1}_{\{x>0\}}$, we generate $\{(X_n(i),Y_n(i))\}_{n=1}^{1000}$ for $i=1,2$ from \eqref{eq:UUU}. For a martingale pair test with parameters $\{\rho = 5, \sigma = 1, \alpha = 0.05\}$, the loss of the series of couplings is \texttt{1.036} against an asymptotic critical  value $c_\alpha=\texttt{4.840}$. Hence, the test correctly accepts the series $\{(X_n,Y_n\}_{n=1}^{1000}$ as a martingale with 95\% confidence. 
\end{example}

It is interesting to note that the modified version of the stochastic process \eqref{formula:markov}
\begin{align*}
U_n = \int_{0}^{n} e^{-rs}\d P_s
\end{align*}
has financial implications. In fact, $ {U}_{\infty}$ has the following density (see \citep{gjessing1997present}):
\begin{align*}
f_{ {U}_{\infty}} = \alpha^{\frac{1}{2}(1+\gamma)}\gamma^{\frac{1}{2}(1-\gamma)}e^{-\gamma}z^{\frac{1}{2}(\gamma-1)}I_{\gamma-1}(2\sqrt{\alpha \gamma z})e^{\alpha z},
\end{align*}
where $\gamma = {\lambda_P}/{r}$ and $\lambda_P$ denotes the intensity of $N_{P,t}$, and $I$ is the modified Bessel function of the first kind of order $\gamma$:
\begin{align*}
I_{\gamma}(x) = (\frac{x}{2})^\gamma \sum_{k = 0}^\infty \frac{1}{k!\Gamma(k+\gamma+1)}(\frac{x}{2})^{2k}.
\end{align*}
In an actuarial context, $ {P}_t$ is interpreted as the surplus generating process and $ {U}_\infty$ represents the present value of perpetual cash flow with respect to $ {P}_t$.

\section{Methodological development and technical proofs}\label{sec:mthd_dev}
In this section, we offer a rigorous step-by-step presentation of the methodological developments leading up to the main results in the previous two sections. We first introduce some necessary notations.

We let $C>0$ denote a large constant depending only on certain parameters, e.g., $(X,Y)$ and $f_{\xi,\rho}$.  The number $L>0$ will denote a large \emph{absolute} constant that does not depend on anything else. The numbers $L,C$ \emph{may not be the same in each occurrence}. For a class $\F$ of real-valued functions and a norm $\n{\cdot}$ on a space containing $\F$, we define the bracketing number $N(\ee,\n{\cdot},\F)$ as the smallest number of $\ee$-brackets needed to cover $\F$, where for $f,g$ the bracket $[f,g]$ is the set $\{h\in\F:f\leq h\leq g\}$, and $[f,g]$ is called an $\ee$-bracket if $\n{f-g}<\ee$. For an event $E$, $\mathds{1}_E$ denotes the indicator random variable of $E$. We will assume throughout that $(p,q)$ forms a conjugate pair, i.e., $p^{-1}+q^{-1}=1$. For a finite set $S$, we denote its cardinality by $ |S|$.
 For $x\in\R$, we denote the floor of $x$ by $\floor{x}$.

 We will also frequently use \eqref{eq:fdensity} and \eqref{eq:f general} in this section, which we recall below: $f_{\xi,\rho}(x)=C_\rho(|x|_2+1)^{-\rho}$ and 
 $$f_{\xi,\rho,\sigma}(x)=\sigma^{-d}C_\rho\left(\frac{\nn{x}_2}{\sigma}+1\right)^{-\rho}.$$
\subsection{The SE-MPD revisited}\label{subsec:deferred_mpd}
In this subsection, we walk through the theoretical developments from Section \ref{sec:mpd}. We start with an example where the classical Wasserstein distance fails to distinguish martingale laws from non-martingale laws. This is a variation of the example discussed in the Introduction. We will elaborate on this modification to illustrate why causality is important.

\begin{example}\label{ex:2}
Consider the random variables $(X,Y, X^\epsilon,Y^\epsilon)$ defined on $(\Omega, \mathcal{F}, \mu)$:
\begin{align*}
\mu( X=1, Y=2)&=\frac{1}{2}=\mu(X=1, Y=0),\\
\mu(X^\ee=1+\epsilon, Y^\ee=2)&=\frac{1}{2}=\mu(X^\ee=1-\epsilon, Y^\ee=0).
\end{align*}
Clearly $(X,Y)$ is martingale in its natural filtration, while $(X^\epsilon, Y^\epsilon)$ is not; however their laws are close in $\mathcal{W}_\gamma$: in fact it is easy to check that  $\mathcal{W}_\gamma(\P, \P^\epsilon)\le \epsilon$ for $(X,Y)\lawis \P$, $(X^\epsilon,Y^\epsilon)\lawis \P^\epsilon$.
\end{example}

What goes wrong in the above example? While the processes are adapted to their natural filtration, the couplings $\pi$ in the definition of $\mathcal{W}_\gamma$ need not be. This motivates the concept of \textit{bi-causal} coupling (see Definition \ref{def:bi-causal}) and what we define as the \textit{adapted, nested} or \textit{bi-causal Wasserstein Distance $\mathcal{A}\mathcal{W}_{\gamma}$} (see Definition \ref{def:bi-causal-wass}).

\begin{example}[Example \ref{ex:2} with causality constraint]\label{ex:3}
Consider again the random variables $(X,Y, X^\epsilon,Y^\epsilon)$ defined in Example \ref{ex:2}. From Definition \ref{def:bi-causal} one can check that any causal coupling $\pi$ from $\P$ to $\P^\epsilon$ needs to  satisfy $\text{Law}(Y|X^\epsilon)=\text{Law}(Y).$ Thus 
\begin{align*}
\mathcal{AW}_\gamma(\P, \P^\epsilon)^\gamma &\geq  \inf \Big\{ \E\big[\E[\nn{Y-Y^\epsilon}_2^\gamma|X^\epsilon]+\epsilon^\gamma\big]: \text{Law}(Y|X^\epsilon)=\text{Law}(Y)\Big\}\\
&=\frac{1}{2} \big(1+1 \big) +\epsilon^\gamma=1+\epsilon^\gamma>1
\end{align*}
for all $\epsilon>0$.
\end{example}

Using Definition \ref{def:bi-causal-wass}, we introduce the central object of our paper: the martingale projection distance (see Definition \ref{def:MPD}). In particular, the following relationship holds for MPD:
\begin{lemma}\label{lem:trivial}
For $\gamma\geq 1$, we have
\begin{align}\label{eq:trivial}
\mathrm{MPD}(\P,\gamma)=0 \quad \Longleftrightarrow \quad \P \text{ is a martingale measure}.
\end{align}
\end{lemma}

\begin{proof}[Proof of Lemma \ref{lem:trivial}]
Clearly, the identity coupling is adapted, so $\mathrm{MPD}(\P_0,\gamma)=0$ if $\P_0$ is a martingale measure. On the other hand, assume that 
$\mathrm{MPD}(\P_0,\gamma)=0$, i.e.,~there exists a sequence $(\Q_n)_{n\in \N}$ of martingale couplings such that $
\lim_{n\to \infty} \mathcal{AW}_\gamma(\P_0, \Q_n)=0$. As the set of martingale measures is closed in $ \mathcal{AW}_\gamma$ (see Proposition 5.7 of \citep{bartl2021wasserstein}), we conclude that $\P_0$ is a martingale measure too.
\end{proof}

Let us define the asymmetric \emph{causal Wasserstein distance} $\mathcal{CW}_\gamma(\Q,\P)$ as
\begin{align*}
\mathcal{CW}_\gamma(\Q,\P)^\gamma :=\inf\big\{\E\big[\nn{Y-Y'}_2^\gamma+\nn{X-X'}_2^\gamma\big]:~&\pi\in \mathcal{P}((\R^d)^4),\, (X,Y,X',Y')\lawis\pi,\\
&\pi \text{ is a causal coupling from $\Q$ to $\P$}\big\}.
\end{align*}

The MPD does not change, if one replaces $\mathcal{AW}_\gamma(\P, \Q)$ by the asymmetric  {causal Wasserstein distance} $\mathcal{CW}_\gamma(\Q,\P)$. 
This is stated in the next proposition:

\begin{proposition}\label{prop:causal}It holds that
\begin{align*}
\mathrm{MPD}(\P,\gamma) =\inf\{ \mathcal{CW}_\gamma(\Q, \P)^\gamma: \E[Y|X]=X \text{ for } (X,Y)\lawis \Q\}.
\end{align*}
\end{proposition}

On the other hand, the causality constraint is essential, as the following example shows:

\begin{example}[Causality is essential]\label{ex:causality}
Consider again $(X, Y, X^\epsilon,Y^\epsilon)$ and $\P^\epsilon$ as defined in Example \ref{ex:2}.
By Theorem \ref{prop:martingale} we have
$$\mathrm{MPD}(\P_\epsilon, \gamma) = \frac{1}{2}\E[|X^\epsilon-\E[Y^\epsilon|X^\epsilon]|^2_2] = \frac{1}{2}|1-\epsilon|^2$$
and so in particular
$$\lim_{\epsilon \rightarrow 0} \mathrm{MPD}(\P_\epsilon,\gamma) = \lim_{\epsilon \rightarrow 0}\frac{1}{2}|1-\epsilon|^2 = \frac{1}{2}.$$
This is no longer true if we remove the causality constraint $\text{Law}(Y'|X,X')=\text{Law}(Y'|X')$ in the definition \eqref{eq:causalcond2} of the MPD: we have
\begin{align*}
\lim_{\epsilon\to 0} \inf\big\{&\E\big[|Y-Y'|_2^2+|X-X'|_2^2\big]:~\pi\in \mathcal{P}((\R^d)^4),(X,Y,X',Y')\lawis\pi,\,\E[Y'|X']=X'\nonumber\\
&\pi(A\times B\times \R\times \R)=\P_\ee(A\times B)\text{ for all Borel sets }A,B\subseteq \R^d \big\}\le \limsup_{\epsilon\to 0} \epsilon^2= 0.
\end{align*}
Indeed this follows by the choice $X'=1$ and $Y'=Y^\epsilon$ in the above infimum.
In particular, we have 
\begin{align*}
   \mathrm{MPD}(\P,\gamma) >\inf\{ \mathcal{CW}_\gamma(\P, \Q)^\gamma: \E[Y|X]=X \text{ for } (X,Y)\lawis \Q\}  
\end{align*}
in general.
\end{example}

 In fact, the causality constraint in the definition of MPD only plays a role through the conditional expectation, as the following corollary states:

\begin{corollary}\label{coro:simple}
We have 
\begin{align*}
\mathrm{MPD}(\P,\gamma)=\inf\big\{&\E\big[|Y-Y'|_2^\gamma+|X-X'|_2^\gamma\big]:~\pi\in \mathcal{P}((\R^d)^4), (X,Y,X',Y')\lawis\pi,\nonumber\\
&\pi(A\times B\times \R^d\times \R^d)=\P(A\times B)\mathrm{\ for\ all\ Borel\ sets\ }A,B\subseteq  \R^d,\nonumber\\
&\E[Y'|X',X]=\E[Y'|X']=X'\big\}.
\end{align*}
\end{corollary}

To delineate the theoretical development of Proposition \ref{prop:causal} and Corollary \ref{coro:simple}, from which Theorem \ref{prop:martingale} follows, we state and prove a series of results.

\begin{proof}[Proof of Proposition \ref{prop:causal} and Corollary \ref{coro:simple}]
Let us define
\begin{align}
\begin{split}
    \widetilde{\mathrm{MPD}}(\P,\gamma) :=\inf\big\{&\E\big[|Y-Y'|_2^\gamma+|X-X'|_2^\gamma\big]:~\pi\in \mathcal{P}((\R^d)^4), (X,Y,X',Y')\lawis\pi,\\
&\pi(A\times B\times \R^d\times \R^d)=\P(A\times B)\text{ for all Borel sets }A,B\subseteq  \R^d,\\
&\E[Y'|X',X]=\E[Y'|X']=X'\big\}.
\end{split}\label{eq:tildeMPD def}
\end{align}
By definition we have
\begin{align}
\widetilde{\mathrm{MPD}}(\P,\gamma) \le \inf\{ \mathcal{CW}_\gamma(\P, \Q)^\gamma: \E[Y|X]=X \text{ for } (X,Y)\lawis \Q\}\le \mathrm{MPD}(\P,\gamma).\label{eq:tildeMPD}
\end{align}
In consequence, to prove Proposition \ref{prop:causal} and Corollary \ref{coro:simple} we only need to show that $\widetilde{\mathrm{MPD}}(\P,\gamma)=\mathrm{MPD}(\P,\gamma)$. For this, we first show the following:

\begin{lemma}\label{lemma:tildeMPD}
We have
\begin{align}\label{eq:n5}
\widetilde{\mathrm{MPD}}(\P,\gamma)=2^{1-\gamma}\E\left[\nn{X-\E[Y|X]}_2^\gamma\right].    
\end{align}
\end{lemma}

\begin{proof}
We first show the $\le$-inequality in \eqref{eq:n5}.
For this we define
\begin{align}
    X'=X+\frac{1}{2}(\E[Y|X]-X),\quad Y'=Y+\frac{1}{2}(X-\E[Y|X]).\label{eq:X'Y'}
\end{align}
We note that $X'$ is  $\sigma(X)$-measurable, and $Y'$ is $\sigma(X,Y)$-measurable. Thus $\sigma(X,X')=\sigma(X)$ and we compute
\begin{align}\label{eq:n1}
\E[Y'|X',X]=\E[Y'|X]=\frac{1}{2} (\E[Y|X]+ X)=X'.
\end{align}
Furthermore by construction
\begin{align*}
    X'-X=\frac{1}{2} (\E[Y|X]-X) =Y-Y'
\end{align*}
and thus
\begin{align}\label{eq:n3}
 \E\big[|Y-Y'|_2^\gamma+|X-X'|_2^\gamma\big]= 2^{1-\gamma}\E\left[\nn{X-\E[Y|X]}_2^\gamma\right]. 
\end{align}
Using the tower property we also compute
\begin{align*}
\E[Y'|X']= \E[\E[Y'|X,X']|X']\stackrel{\eqref{eq:n1}}{=}\E[X'|X']=X',
\end{align*}
so the martingale property holds.

Therefore, it suffices to prove
$$\widetilde{\mathrm{MPD}}(\P,\gamma)\geq 2^{1-\gamma}\E\left[\nn{X-\E[Y|X]}_2^\gamma\right].$$
Consider a coupling $\pi$ where $(X,Y)\lawis\P$ and $\E[Y'|X',X]=X'$. First, using the elementary inequality $(a^\gamma+b^\gamma)\ge 2^{1-\gamma} (a+b)^\gamma$ for $a,b\ge 0$ and then applying Jensen's inequality we obtain
\begin{align*}
\E\big[|Y-Y'|_2^\gamma+|X-X'|_2^\gamma\big]&=\E\big[\E\big[|Y-Y'|_2^\gamma+|X-X'|_2^\gamma|X\big]\big]\\
&\ge  2^{1-\gamma} \E\big[\E\big[(|Y-Y'|_2+|X-X'|_2)^\gamma|X\big]\big]\\
&\geq 2^{1-\gamma}\E\big[\E\big[|Y-Y'|_2+|X-X'|_2|X\big]^\gamma\big].
\end{align*}
Note that by assumption we have $\E[Y'|X',X]=\E[Y'|X']=X'$. By repeated use of Jensen's inequality, the inner conditional expectation can be bounded by
\begin{align*}
\E[\nn{Y-Y'}_2+\nn{X-X'}_2|X] &\geq
\E[| X-X'+Y'-Y|_{2}|X] \\ &=
\E[\E[| X-X'+Y'-Y|_{2}|X',X]|X] \\ &\geq
\E\left[\nn{\E[X-X'+Y'-Y|X,X']}_2|X\right]\\
&=\E\left[\nn{X-\E[Y|X,X']}_2|X\right]\\
&\geq \nn{\E\left[X-\E[Y|X,X']|X\right]}_2\\
&=\nn{X-\E[Y|X]}_2.
\end{align*}
Combining the two estimates above yields
$$\widetilde{\mathrm{MPD}}(\P,\gamma)\geq 2^{1-\gamma}\E\big[\E\big[|Y-Y'|_2+|X-X'|_2|X\big]^\gamma\big]\geq 2^{1-\gamma}\E\big[|X-\E[Y|X]|_2^\gamma\big],$$
as required.
\end{proof}

\begin{remark}
    The proof of Lemma \ref{lemma:tildeMPD} indicates that the infimum in the definition \eqref{eq:tildeMPD def} of $\widetilde{\mathrm{MPD}}(\P,\gamma)$ is always attained. On the other hand, this is not true for the infimum in \eqref{eq:causalcond2} for $\mathrm{MPD}(\P,\gamma)$. A simple counterexample is given by the following: take $\gamma=1$, $X$ uniform on $\{0,1\}$, $Y|X=0$ uniform on $\{0,1,-1\}$, and $Y|X=1$ uniform on $\{0,-2\}$. The infimum in $\eqref{eq:tildeMPD def}$ is uniquely attained by the coupling $(X,Y,X',Y')$ where $(X',Y')=(X,Y)$ on the event $\{X=0\}$, and $(X',Y')=(X-1,Y+1)$ on the event $\{X=1\}$. It is straightforward to verify that $\P(Y'=0| X'=0)>0$ but $\P(Y'=0| X'=0,X=1)=0$. Therefore, the coupling $(X,Y,X',Y')$ is not bi-causal, and hence the infimum in \eqref{eq:causalcond2} for $\mathrm{MPD}(\P,\gamma)$ is not attained as the infimum value taken over a subset is the same.
\end{remark}

\begin{lemma}\label{lemma:tildeMPD2}
We have
\begin{align*}
\widetilde{\mathrm{MPD}}(\P,\gamma)=\mathrm{MPD}(\P,\gamma).
\end{align*}
\end{lemma}

\begin{proof}
We recall from \eqref{eq:tildeMPD} that
\begin{align*}
\widetilde{\mathrm{MPD}}(\P,\gamma)\le \mathrm{MPD}(\P,\gamma).
\end{align*}
It thus suffices to show the reverse inequality. We will do this by constructing a sequence of bi-causal couplings $(\pi^\delta)$, which achieve $\eqref{eq:n5}$ for $\delta\downarrow 0$: we construct $(X^\delta, Y^\delta)$ according to Lemma \ref{lem:bartl} below and set $\pi^\delta\lawis (X,Y, X^\delta, Y^\delta)$. It is now easy to show that $\pi^\delta$ achieves \eqref{eq:n4} for $\delta\downarrow 0$: we simply write 
\begin{align*}
    \E\big[|Y-Y^\delta|_2^\gamma+| X-X^\delta|_2^\gamma\big]
    &= \E\big[|Y-Y'+(Y'-Y^\delta)|_2^\gamma +|X-X'+(X'-X^\delta)|_2^\gamma\big].
\end{align*}
Taking $\delta\downarrow 0$, the $\le$-inequality in \eqref{eq:n4} then follows from \eqref{eq:n3} and $|Y^\delta-Y'|_2=|X^\delta-X'|_2\le d\delta$. 
It remains to note that $\pi^\delta$ is bi-causal, as $\sigma(X, X^\delta)=\sigma(X)=\sigma(X^\delta)$. This concludes the proof.
\end{proof}

We have used the following lemma, which is a slight extension of Lemma 3.1 in \citep{bartl2023sensitivity}:
\begin{lemma}\label{lem:bartl}
Let $(X,Y, X',Y')$ be as in \eqref{eq:X'Y'}. For each $\delta>0$ there exist random variables $X^\delta, Y^\delta$ such that the following hold:
\begin{itemize}
    \item $X^\delta$ is $\sigma(X)$-measurable and $Y^\delta$ is $\sigma(X,Y)$-measurable,
    \item $X$ is $\sigma\left(X^\delta\right)$-measurable,
    \item $|Y^\delta-Y'|_2 =|X^\delta-X'|_2 \leq d\delta$,
    \item $\E[Y^\delta| X, X^\delta]=\E[Y^\delta | X^\delta]=X^\delta.$
\end{itemize}
\end{lemma}

\begin{proof}
For $\delta>0$ we consider the Borel mappings
$$
\begin{aligned}
& \psi_\delta: \mathbb{R}^d \rightarrow(0, \delta)^d \text { and } \\
& \phi_\delta: \mathbb{R}^d \rightarrow \delta \mathbb{Z}^d:=\{\delta k: k \in \mathbb{Z}^d\},
\end{aligned}
$$
where $\psi_\delta$ is a (Borel-)isomorphism and\footnote{The constraint $z\leq x$ below is imposed to guarantee uniqueness of the argmin.} $$\phi_\delta(x):=\text{argmin} \{|x-z|_2: z\in \delta\mathbb{Z}^d, z\le x\}.$$ We set
\begin{align*}
X^\delta&:=\phi_\delta\left(X'\right)+\psi_\delta\left(X\right),\\
Y^\delta&:= Y' +(X^\delta-X') .
\end{align*}
By definition,
$X^\delta$ is $\sigma(X)=\sigma(X,X')$-measurable, 
$X$ is $\sigma\left(X^\delta\right)$-measurable, $Y^\delta$ is $\sigma(X,Y)=\sigma(X, X^\delta, Y')$-measurable, 
and $$|Y^\delta-Y'|_2=|X^\delta-X'|_2  \leq 2d\delta.$$
The martingale property follows from 
\begin{align*}
\E[Y^\delta | X,X^\delta]=\E[Y^\delta | X] 
&= \E[Y'+(X^\delta-X') | X]\stackrel{\eqref{eq:n1}}{=} X' +(X^\delta-X')=X^\delta,
\end{align*}
recalling that $\sigma(X)=\sigma(X^\delta)\supseteq \sigma(X').$ This concludes the proof.
\end{proof}

Combining Lemmas \ref{lemma:tildeMPD} and \ref{lemma:tildeMPD2} yields Proposition \ref{prop:causal} and Corollary \ref{coro:simple}. Theorem \ref{prop:martingale} also follows immediately.
\end{proof}

While we have derived the closed-form formula for MPD in Theorem \ref{prop:martingale}, note that if $X$ has a continuous distribution under $\P_0$ and $\E_n$ denotes the empirical distribution, $\E_n[Y|X] = Y$ generally does not converge to $X$ as the sample size increases, as we argued at the beginning of Section \ref{subsection:smoothed_MPD}. To overcome this, we have applied a smoothing technique and introduced the SE-MPD (see Definition \ref{def:smooth}) in Section \ref{subsection:smoothed_MPD}. We have motivated this by Proposition \ref{prop:martingale}, which states that under mild conditions, $\P_0$ is martingality-preserving (see Definition \ref{def:martingality-preserving}). However, this is not always true, as shown by the following example:

\begin{example}\label{ex:smoothing}
    We provide a counterexample where $(X+\xi,Y+\xi)$ is a martingale but $(X,Y)$ is not, and $(X,Y,\xi)\lawis \P\otimes\P_\xi$ with $X,Y,\xi$ real-valued and absolutely continuous. For $f\in L^1(\R)$ we denote by $\sF f$ its Fourier transform. Recall two facts from Fourier analysis:
    \begin{itemize}
        \item If $f\in L^1(\R)$ is nonnegative and $\sF f\geq 0$, then $\sF f\in L^1(\R)$. This is Corollary 8.7 of \citep{chandrasekharan2012classical}.
        \item Fourier inversion: if $f,\sF f\in L^1(\R)$, then $\sF \sF f(x)=f(-x)$.
    \end{itemize}
Together with \citep{tuck2006positivity}, it follows from the above facts that there exists a function $f_\xi$ such that $f_\xi\geq 0$, $\int f_\xi=1$, and that $\sF f_\xi(x)=0$ for $|x|\geq 1$. For example, let $f_\xi=\sF \Phi$ where $\Phi(x)=\max(1-|x|,0)$. 
Similarly, using Fourier inversion, we may construct a function $\psi$ that is $O(|x|^{-3})$ at $|x|\to\infty$  and such that $\sF\psi(x)=0$ for $|x|\leq 1$. For example, this can be done by taking $\psi=\sF\Psi$ where $\Psi$ vanishes on $[-1,1]$ and satisfies $\Psi',\Psi'',\Psi'''\in L^1(\R)$. In particular, there exists a probability density $f_X$ of an integrable random variable on $\R$ such that $h:=\psi/f_X\in L^\infty(\R)$. Now suppose that $X$ and $\xi$ have marginal densities given by $f_X$ and $f_\xi$ constructed above, and consider any  coupling $(X,Y)$ satisfying that $\E[Y|X]=X+h(X)$. Since $h$ is bounded and not identically zero, $(X,Y)$ is integrable and is not a martingale. We check that $(X+\xi,Y+\xi)$ is a martingale. It follows from our construction that for each $x\in \R$ (for convenience we work with regular conditional probabilities),
\begin{align*}
    \E[Y+\xi\mid X+\xi=x]&=\int \frac{f_\xi(u)f_X(x-u)}{\int f_\xi(v)f_X(x-v)\d v}(u+\E[Y|X=x-u])\d u\\
    &=\frac{\int f_\xi(u)f_X(x-u)(x+h(x-u))\d u}{\int f_\xi(v)f_X(x-v)\d v}\\
    &=x+\frac{\int f_\xi(x-u)\psi(u)\d u}{\int f_\xi(v)f_X(x-v)\d v}.
\end{align*}
Note that the Fourier transform of the numerator $\int f_\xi(x-u)\psi(u)\d u$ is equal to $\sF f_\xi\sF \psi$, which is identically zero by our construction. Hence, $\E[Y+\xi\mid X+\xi=x]=x$ for all $x\in\R$, proving that $(X+\xi,Y+\xi)$ is a martingale.    
\end{example}

To characterize laws that are martingality-preserving, we presented Proposition \ref{prop:martingale2}. We now give a detailed proof of this result. 

\begin{proof}[Proof of Proposition \ref{prop:martingale2}]
We first note that the martingale condition 
$$\E[(Y+\xi)-(X+\xi)| X+\xi] =0$$
can be rewritten as
$$\E[(Y-X)| X+\xi] =0.$$
Now we define the functions $m,n:\R^d\to \R^d$ via $m(a):=\E[(Y-X)\mathds{1}_{\{X\leq a-\xi\}}]$ and $n(a):=\E[(Y-X)\mathds{1}_{\{X\leq a\}}]$ for $a\in\R^d$, where $\leq$ denotes the lexicographical order in $\R^d$. Using a monotone class argument, it then suffices to prove that $m(a)=0$ for all $a\in\R^d$ if and only if $n(a)=0$ for all $a\in\R^d$.
It follows from the triangle inequality that $m,n$ are uniformly bounded. Furthermore, using a change of variable,
$$m(a)=\int_{\R^d}n(a-x)f_\xi (x)\d x=\int_{\R^d}n(x)f_\xi (a-x)\d x.$$ In particular, it follows immediately that $n(a)=0$ for all $a\in\R^d$ implies $m(a)=0$ for all $a\in\R^d$. To see the converse we argue as follows:  since the Fourier transform of $f_\xi$ has no real zeros, we can apply Wiener's Tauberian theorem (see \citep[Theorem 8]{wiener1988fourier}) to conclude that the linear span of the set of translates $\{f_\xi(a-\cdot): a\in \R^d\}$ is dense in $L^1(\R^d)$.
In particular, $m(a)=0$ for all $a\in\R^d$  only if $n(x)=0$ for a.e.~$x\in\R^d$. Since $n$ is right-continuous we thus have $n(a)=0$ for all $a\in\R^d$. This concludes the proof.
\end{proof}

\begin{remark}
    With the same proof as above, the following also holds: suppose that $(\xi,\xi')$ is a martingale independent of $(X,Y)$, where $(X,Y)$ and $\xi$ satisfy the conditions in Proposition \ref{prop:martingale2}. Then $(X,Y)$ is a martingale if and only if $(X+\xi,Y+\xi')$ is a martingale. Our theory of smoothed MPD and martingality testing thus extends to smoothing by a martingale pair whose first marginal satisfies the conditions in Proposition \ref{prop:martingale2}. 
\end{remark}

An example of a martingality-preserving smoothing law that we use for this work is given by (\ref{eq:fdensity}). To confirm that it is indeed martingality-preserving, we note that by \citep[Theorem 2.2]{joarder1996characterization} we have
\begin{align*}
\int_{\R^d}f_\xi (x)e^{i\langle t, x \rangle}\d x = \int_0^\infty F_{0,1}\Big(\frac{d}{2}, \frac{(r|t|_2)^2}{4}\Big)\,H(\d r),      
\end{align*}
where $H$ is the cdf of $|\xi|_2$ and $F_{0,1}$ is the generalized hypergeometric function. As $(r|\cdot|_2)^2\ge 0$ we have $$F_{0,1}\Big(\frac{d}{2}, \frac{(r|t|_2)^2}{4}\Big)\ge 1>0.$$
Therefore, laws with a density of the form \eqref{eq:fdensity} are martingality-preserving. A standard scaling argument also yields that the densities of the form \eqref{eq:f general} also qualify.  

In addition to (\ref{eq:fdensity}), there are other smoothing measures that are martingality-preserving. We now give a few examples of measures $\P_\xi$, which satisfy the assumptions of Proposition \ref{prop:martingale2}.

\begin{example}\label{ex:1}
Assume that $f_\xi$ is of the form \begin{align}f_\xi(x_1,\dots,x_d)=\prod_{j=1}^d f_j(x_j)\label{eq:phi}\end{align}for some probability density functions $f_j,~1\leq j\leq d$ that satisfy $\int |x|f_j(x)\d x<\infty$ and either one of the following conditions holds:
\begin{enumerate}[(i)]
    \item $f_j$ is symmetric and strictly convex on $(0,\infty)$;\label{casei}
    \item $f_j$ is the density of an infinitely divisible distribution.\label{caseii}
\end{enumerate}
Then the assumptions of Proposition \ref{prop:martingale2} are satisfied. To see this, note that by \eqref{eq:phi} we have
$$\int_{\R^d}f_\xi (x)e^{i\langle t, x \rangle}\d x=\prod_{j=1}^d\int_\R f_j(x_j)e^{it_jx_j}\d x_j,\quad t=(t_1,\dots,t_d)\in\R^d.$$ Therefore, it suffices to show that the Fourier transform of each $\psi_j$ has no real zeros. Now (\ref{casei}) follows from \citep{tuck2006positivity}, while (\ref{caseii}) is a consequence of \citep[Lemma 7.5]{sato1999levy}. We also remark that a further sufficient condition for (\ref{caseii}) above is the complete monotonicity of the densities; see \citep[Theorem 51.6]{sato1999levy}.
\end{example}

\begin{example}\label{ex:smoothex}
    
We recall that a (centered) multivariate Student's $t$-distribution with a degree of freedom $\nu>0$ and scaling matrix $\Sigma$ (that is symmetric positive definite) has pdf given by 
$$C(\Sigma)\Big(1+\frac{\langle x\Sigma^{-1},x\rangle}{\nu}\Big)^{-(\nu+d)/2},$$
where $C(\Sigma)$ is an appropriate normalizing constant. The multivariate Student's $t$-distribution is known to be infinitely divisible (\citep{grigelionis2013student}), and hence the assumptions of Proposition \ref{prop:martingale2} are satisfied. That is, multivariate Student's $t$-distributions are martingality-preserving.
\end{example}

\subsection{Finite-sample rates for \texorpdfstring{$\gamma =1$}{} revisited}
\label{sec:tech_dev_finite_and_general}
In this section, we prove Proposition \ref{prop:finitesample}. The arguments here are also fundamental towards proving Theorem \ref{thm:limitd simple}. Before delving into the proofs, we offer a convenient expression for the smoothed MPD.
 Recall that $\E_n^{*\xi}$ denotes the expectation of  $(X,Y)\lawis \P_n^{*\xi}$.
Theorem \ref{prop:martingale} implies that
$$\mathrm{MPD}^{*\xi}(\P_n, \gamma)=2^{1-\gamma} \E_n^{*\xi}[|{ X -\E_n^{*\xi}[Y|X]}|_2^\gamma].$$
The inner expectation can be computed more explicitly, as the next lemma shows:

\begin{lemma}\label{lemma:computation}
Suppose that $(X,Y)\lawis {\P}_n^{*\xi}$. Then we have
\begin{align*}
\E_n^{*\xi}[Y-X|X]&=\frac{\sum_{i=1}^n(Y_i-X_i)f_\xi(X-X_i)}{\sum_{i=1}^nf_\xi(X-X_i)},
\end{align*}
where we recall that $f_\xi$ is the density of $\P_\xi$.
\end{lemma}

\begin{proof}
We will prove the claim by checking that for each Borel set $S$,
\begin{align}\label{eq:cond_exp_def}
\E_n^{*\xi}[(Y-X)\mathds{1}_{\{X\in S\}}] = \E_n^{*\xi}\bigg[\frac{\sum_{i=1}^n(Y_i-X_i)f_\xi(X-X_i)}{\sum_{i=1}^nf_\xi(X-X_i)}\mathds{1}_{\{X\in S\}} \bigg].
\end{align}
To check this, we first recall that ${\P}_n =1/n \sum_{i=1}^n \delta_{(X_i,Y_i)}$ is the empirical measure of the observations, and that $\xi$ has density $f_\xi$. Furthermore,  assume for notational simplicity that the observations $X_1, \dots, X_n$ are pairwise distinct. Recall that we fixed a probability space $(\Omega, \F, \mu)$, on which all random variables are defined. By the law of total probability we have for any Borel sets $S,T\subseteq \R^d$ and $(X,Y, \xi)\lawis \P_n\otimes \P_\xi$,
\begin{align*}
&\hspace{0.5cm} \P^{*\xi}_n\left(S\times T\right)\\
&= \mu \left(X+\xi \in S, Y+\xi\in T\right)\\
&= \sum_{i=1}^n \mu \left(X+\xi \in S, Y+\xi\in T \mid (X,Y)=(X_i,Y_i)\right) \mu\left((X,Y)=(X_i,Y_i)\right)  \\
&= \frac{1}{n}  \sum_{i=1}^n \mu \left(\xi \in S-X_i, \xi\in T-Y_i\right),
\end{align*}
where the last equality follows from independence of $(X,Y)$ and $\xi$. Following the same arguments,
\begin{align*}
\E_n^{*\xi}[(Y-X)\mathds{1}_{\{X\in S\}}]  &= \frac{1}{n} \sum_{i=1}^n \E[(Y+\xi-(X+\xi))\mathds{1}_{\{X+\xi\in S\}} | (X,Y)=(X_i,Y_i)]\\
&=\frac{1}{n} \sum_{i=1}^n \int_{S} (Y_i-X_i)  f_\xi(x-X_i)\,\d x.
\end{align*}
On the other hand,
\begin{align*}
\P_n^{*\xi}\left(S\right)= \frac{1}{n}  \sum_{j=1}^n \mu \left(\xi \in S-X_j\right)=\frac{1}{n}  \sum_{j=1}^n \int_S f_\xi(x-X_j)\,\d x,
\end{align*}
and thus
\begin{align*}
\E_n^{*\xi}[(Y-X)\mathds{1}_{\{X\in S\}}]
&= \frac{1}{n} \sum_{i=1}^n \int_S (Y_i-X_i)  f_\xi(x-X_i)\,\d x \\
&=\frac{1}{n} \sum_{j=1}^n  \int_S \frac{\sum_{i=1}^n(Y_i-X_i)f_\xi(x-X_i)}{\sum_{i=1}^nf_\xi(x-X_i)} f_\xi(x-X_j)\,\d x\\
&=\E_n^{*\xi}\bigg[\frac{\sum_{i=1}^n(Y_i-X_i)f_\xi(X-X_i)}{\sum_{i=1}^nf_\xi(X-X_i)}\mathds{1}_{\{X\in S\}} \bigg].
\end{align*}
This proves \eqref{eq:cond_exp_def}.
\end{proof}

We will also use the following discrepancy bound to control the fluctuation of empirical processes:
\begin{lemma}[Corollary 14.1.2 of \citep{talagrand2022upper}]\label{lemma:empiricalbound}There exists a universal constant $L>0$ such that the following holds: consider a measure space $(\Omega,\nu)$ and an i.i.d.~sequence $\{Z_n\}_{n\in \N}$ sampled from $\nu$. Let $\F\subseteq L^2(\nu)$ be a countable set with $0\in\F$. For $\S\subseteq\F$ define the function $h_\S$ via
$$h_\S(\omega):=\sup_{f,f'\in \S}|f(\omega)-f'(\omega)|.$$
Then
$$\E\Big[\sup_{f\in\F}\Big|\sum_{i=1}^n(f(Z_i)-\E[f(Z)])\Big|\Big]\leq L\sqrt{n}\inf\sup_{f\in\F}\sum_{l\geq 0}2^{l/2}\n{h_{A_l(f)}}_2,$$
where $Z\lawis \nu$ and the infimum is taken among all sequences of refining partitions $\{\A_l\}$ of $\F$ such that $|\A_0|=1,~|\A_l|\leq N_l:=2^{2^l}$, and $\A_l(f)$ is the set in the partition $\A_l$ that contains $f$.
\end{lemma}

We make a few conventions to shorten notation. For $k\in\Z\setminus\{0\}$, define the interval $I_k$ to be $[k-1,k]$ if $k<0$ and $[k,k+1]$ if $k> 0$.  For $k=(k^1, \dots, k^d) \in \Z^d$ we define the multi-interval \begin{align}
    I_k=I_{k^1}\times I_{k^2}\times \cdots I_{k^d}.\label{eq:Ik}
\end{align}
Next we define the class of functions $\F_k=\{f_a\}_{a\in I_k}\cup\{0\}$, where
\begin{align}
    f_a(x,y):=(y-x)f_\xi(a-x),\qquad x,y\in\R^d.\label{eq:fk}
\end{align} Denote by $f^j_a$ the $j$-th coordinate of $f_a$. We also define  
\begin{align}
     \xi_n(x):=\frac{1}{n}\sum_{i=1}^n(Y_i-X_i)f_\xi\left(x-X_i\right),\qquad x\in\R^d.\label{eq:xindef}
 \end{align} 
Furthermore we write $$p_n(x):= \frac{1}{n} \sum_{i=1}^n f_\xi(x-X_i).$$We will write $\E_n^{*\xi}$ for the expectation of  $(X,Y)\lawis \P_n^{*\xi}$.  Combining Theorem \ref{prop:martingale} and Lemma \ref{lemma:computation} leads to the following representation of the MPD:
\begin{align}
\begin{split}
    \mpd^{*\xi}(\P_n,\gamma)&=2^{1-\gamma}\E_n^{*\xi}\left[\nn{X-\E_n^{*\xi}[Y|X]}_2^\gamma\right]\\
&=2^{1-\gamma}\E_n^{*\xi}\left[\left|\frac{\sum_{i=1}^n(Y_i-X_i)f_\xi(X-X_i)}{\sum_{i=1}^nf_\xi(X-X_i)}\right|_2^\gamma\right]\\
&=2^{1-\gamma}\int p_n(x)\left|\frac{\xi_n(x)}{p_n(x)}\right|_2^\gamma\d x\\
&=2^{1-\gamma}\int \frac{|\xi_n(x)|_2^\gamma}{p_n(x)^{\gamma-1}}\,\d x.
\end{split}\label{eq:En expression}
\end{align}
Note that since $(X,Y)\lawis \P$ is a martingale pair we have 
\begin{align}
    \xi_n(x)\to \E\left[(Y-X)f_\xi\left(x-X\right)\right]=0\in\R^d \qquad\mu\text{-a.s.}\label{eq:slln}
\end{align}
for each $x\in\R$ by the strong law of large numbers. From \eqref{eq:En expression}, we may deduce the proof of Proposition \ref{prop:2k}.

\begin{proof}[Proof of Proposition \ref{prop:2k}] 
    Recall that $f_{\xi,\rho,\sigma}$ denotes the density of $\sigma\xi$ for $\xi$ having density $f_{\xi,\rho}$. It follows that $f_{\xi,\rho,\sigma}(x)=\sigma^{-1}f_{\xi,\rho}(x/\sigma)$ for any $\sigma>0$. By \eqref{eq:En expression} applied with $\gamma=1$ and a change of variable, we have
    \begin{align*}
        \sqrt{n}\,\mpd^{*\xi}_{\sigma_n}(\P_n,1)&=\int \Big|\frac{1}{\sqrt{n}}\sum_{i=1}^n(Y_i-X_i)f_{\xi,\rho,\sigma_n}(x-X_i)\Big| \d x\\
        &=\int \Big|\frac{1}{\sqrt{n}}\sum_{i=1}^n(Y_i-X_i)f_{\xi,\rho}(u-\frac{X_i}{\sigma_n})\Big| \d u.
    \end{align*}
Suppose first that the alternative holds. By Taylor's Theorem, for each $k\in\N$, we may write
\begin{align*}
    f_{\xi,\rho}(u-\frac{X_i}{\sigma_n})&=\sum_{j=0}^k\frac{f_{\xi,\rho}^{(j)}(u)}{j!}(-\frac{X_i}{\sigma_n})^j+O\Big(|\frac{X_i}{\sigma_n}|^{k+1}\sup_{u-1\leq t\leq u+1}|f_{\xi,\rho}^{(k+1)}(t)|\Big)\\
    &=\sum_{j=0}^k\frac{f_{\xi,\rho}^{(j)}(u)}{j!}(-\frac{X_i}{\sigma_n})^j+O((\sigma_n)^{-(k+1)}).
\end{align*}
Therefore, we have
\begin{align*}
    \sqrt{n}\,\mpd^{*\xi}_{\sigma_n}(\P_n,1)
    &=\int \Big|\frac{1}{\sqrt{n}}\sum_{i=1}^n(Y_i-X_i)\sum_{j=0}^k\frac{f_{\xi,\rho}^{(j)}(u)}{j!}(-\frac{X_i}{\sigma_n})^j\Big| \d u+O(\sqrt{n}(\sigma_n)^{-(k+1)}).
\end{align*}
By the central limit theorem, for each fixed $j\in\{0,\dots,k-1\}$, under both the null and the alternative,
$$\Big|\frac{1}{\sqrt{n}}\sum_{i=1}^n(Y_i-X_i)\frac{f_{\xi,\rho}^{(j)}(u)}{j!}(-\frac{X_i}{\sigma_n})^j\Big|=O((\sigma_n)^{-j}).$$
It then follows from the triangle inequality that
$$\sqrt{n}\,\mpd^{*\xi}_{\sigma_n}(\P_n,1)=\frac{\sqrt{n}(\sigma_n)^{-k}}{k!} \Big|\frac{1}{{n}}\sum_{i=1}^n(Y_i-X_i)X_i^k\Big|\int f_{\xi,\rho}^{(k)}(u)\d u+O(1+\sqrt{n}(\sigma_n)^{-(k+1)}).$$
We then conclude from the strong law of large numbers that under the alternative hypothesis and the assumptions on the sequence $(\sigma_n)_{n\in\N}$, \eqref{eq:alt} holds, and that if $n^{1/(2k)}=o(\sigma_n)$, the sequence $\{\sqrt{n}\,\mpd^{*\xi}_{\sigma_n}(\P_n,1)\}_{n\in\N}$ is tight under the alternative hypothesis.

On the other hand, under the null hypothesis, a direct computation of $\E[\sqrt{n}\,\mpd^{*\xi}_{\sigma_n}(\P_n,1)]$ using the central limit theorem and the boundedness of $(X,Y)$ yields that 
$$\E[\sqrt{n}\,\mpd^{*\xi}_{\sigma_n}(\P_n,1)]=O\bigg(\sum_{j=0}^\infty(\sigma_n)^{-j}\bigg)=O(1),$$
where we used the assumption $\sigma_n\to\infty$. This implies that $\{\sqrt{n}\,\mpd^{*\xi}_{\sigma_n}(\P_n,1)\}_{n\in\N}$ is tight.
\end{proof}

In view of \eqref{eq:En expression}, it is natural to investigate upper bounds for $\E[\nn{\xi_n(x)}_2]$ in order to determine finite-sample rates of  $\E[\mpd^{*\xi}(\P_n,1)]$. This is the goal of the following lemma, which also plays a crucial role when proving Theorem \ref{thm:limitd}.

\begin{lemma}
    \label{lemma:Esup when gamma=1} 
    Assume $(X,Y)\in L^{2p}$ for some $p>1$ and that the density of $\xi$ is given by \eqref{eq:f general}, i.e.,
    $$f_{\xi,\rho,\sigma}(x)=\sigma^{-d}C_\rho\left(\frac{\nn{x}_2}{\sigma}+1\right)^{-\rho}.$$
    Then we have for all $k\in\Z^d$ satisfying $\nn{k}_2\geq \sigma$,\footnote{The supremum over $x\in I_k$ is understood as an essential supremum, which is equivalent to the supremum over all rationals $x$ in $I_k$ given that $\xi_n(x)$ is a.s.~continuous in $x\in\R^d$, a consequence of the continuity of $f_{\xi,\rho}$ we impose in \eqref{eq:f general}. The same applies for the suprema over $f_a\in\F_k$ below, etc. Therefore, we may apply Lemma \ref{lemma:empiricalbound} directly by constructing the partitions for $\F_k$. }
    \begin{align*}
        \E\big[\sup_{x\in I_k} |\xi_n(x)|_2 \big]&\leq Ln^{-1/2}\frac{dC_\rho}{\sigma^{d+1}}\Bigg((\sqrt{d}\rho+\sigma)\E[\nn{X}_2^{2p}]^{1/(2q)}(\frac{\nn{k}_2}{2})^{-(p-1)}\\
    &\hspace{3cm}+(\sqrt{d}\rho+\nn{k}_2)(\frac{\nn{k}_2}{2\sigma})^{-\rho-1}\Bigg)\sum_{j=1}^d\n{X^j-Y^j}_{2p}.
    \end{align*}
   In particular, for all $(X,Y)\in L^{2p}$  and $\nn{k}_2\geq\sigma$ there exists $C>0$ independent of $k$ and $n$ such that
    $$\E\big[\sup_{x\in I_k} |\xi_n(x)|_2 \big]\leq Cn^{-1/2}(|k|_2^{1-p}+|k|_2^{-\rho}).$$
\end{lemma}

\begin{proof}
Let us fix $k\in \Z^d$, $|k|_2\geq\sigma$. We aim to bound $\E\left[\sup_{x\in I_k}|\xi_n(x)|_2\right]$ using Lemma \ref{lemma:empiricalbound}. 
 Writing $\xi_n(x)= (\xi^1_n(x), \dots, \xi^d_n(x))$ we have 
\begin{align}
|\xi_n(x)|_2\le \sum_{j=1}^d |\xi_n^j(x)|\label{eq:j}
\end{align}
by the triangle inequality. It is thus sufficient to fix $j\in \{1, \dots, d\}$ and bound $$\E\big[\sup_{x\in I_k} |\xi_n^j(x)|\big]=\frac{1}{n}\E\big[\sup_{a\in I_k}|\sum_{i=1}^nf_a^j(X_i,Y_i)|\big]=\frac{1}{n}\E\big[\sup_{f_a\in \F_k}|\sum_{i=1}^nf_a^j(X_i,Y_i)|\big],$$
where we recall the class of functions $\F_k$ defined in \eqref{eq:fk} above.

To this end we take $\Omega=\R^d\times \R^d$ and let $\nu$ be the conditional distribution of $(X,Y)$ given $(X_1, Y_1, \dots, X_n,Y_n)$ in Lemma \ref{lemma:empiricalbound}. 
In order to apply Lemma \ref{lemma:empiricalbound} we first have to construct the partitions $\A_l,\,l\geq 0$ of $\F_k$. We proceed as follows: set $\A_0:=\{\F_k\}$.  For $l\geq 1$, we divide the box $I_k$ uniformly in the $d$ directions into $\max(1,(\floor{N_l^{1/d}}-1)^d)$ many smaller boxes $I_{k,s}$ indexed by $\{s\in \N^d: 1\leq s^j<\max(2,\floor{N_l^{1/d}})$ for $1\leq j\leq d\}$, where we also recall $N_l=2^{2^l}$.  Define 
$$\A_l:=\{\{0\}\}\cup\Bigg(\bigcup_{s\in \N^d: 1\leq s<\max(2,\floor{N_l^{1/d}})}\left\{\{f_a\}_{a\in I_{k,s}}\right\}\Bigg).$$
It follows that $|\A_l|\leq N_l$ for $l\geq 0$. 
For $d$-dimensional vectors $A,B\in \R^d$ with $A<B$ we identify $\{f_a\}_{a\in[A,B]}$ with $[A,B]:=[A_1,B_1]\times [A_2,B_2]\times \cdots\times [A_d, B_d]$ in the following. Recall \eqref{eq:f general}, from which we compute
\begin{align}
    \nn{\nabla f_{\xi,\rho,\sigma}(x)}_2=\rho\sigma^{-(d+1)}C_\rho \left(\nn{\frac{x}{\sigma}}_2+1\right)^{-(\rho+1)}.\label{eq:gradient of f}
\end{align}Consider $A,B\in I_k$. For $a\in [A,B]$ and  $\nn{x}_2\leq\nn{k}_2/2$, we have $\nn{a-x}_2\geq\nn{a}_2-\nn{x}_2\geq \nn{k}_2/2$. 
By the mean-value theorem, using that $f_{\xi,\rho,\sigma}$ is radially symmetric, we have 
\begin{align}\label{eq:h}
\begin{split}&\hspace{0.5cm} \sup_{A\le a\le b\le B} |f_a^j(x,y)-f_b^j(x,y)|\\
&= |x^j-y^j| \sup_{A\leq a<b\leq B}|f_{\xi,\rho,\sigma}(a-x)-f_{\xi,\rho,\sigma}(b-x)|\\
&\leq\begin{cases}
 |A-B|_2 |x^j-y^j|  |\nabla f_{\xi,\rho,\sigma}(k/2)|_2&\text{  if } |x|_2 \le \nn{k}_2/2,\\
  |A-B|_2 |x^j-y^j|\sup_{{x\in \R^d}} \nn{\nabla f_{\xi,\rho,\sigma}(x)}_2&\text{ otherwise,}
\end{cases}\\
&\leq\begin{cases}
 |A-B|_2 |x^j-y^j|  \rho\sigma^{-(d+1)}C_\rho \left(\nn{\frac{k}{2\sigma}}_2+1\right)^{-(\rho+1)} &\text{  if } |x|_2 \le \nn{k}_2/2,\\
 |A-B|_2 \rho\sigma^{-(d+1)}C_\rho|x^j-y^j| &\text{ otherwise,}
\end{cases}\\
&=:|A-B|_2\tilde{h}^j_k(x,y).
\end{split}
\end{align}
In addition,
\begin{align}\label{eq:h2}
\begin{split}
\sup_{a\in I_k}|f^j_a(x,y)|&=|x^j-y^j| \sup_{a\in I_k}|f_{\xi,\rho,\sigma}(a-x)|\\
&\leq\begin{cases}|x^j-y^j| |f_{\xi,\rho,\sigma}(|k|_2/2)|&\text{ if } |x|_2 \le |k|_2/2,\\
C_\rho\sigma^{-d} |x^j-y^j| &\text{ otherwise,}
\end{cases}\\
&=:\tilde{g}_k^j(x,y).
\end{split}
\end{align}
Following Lemma \ref{lemma:empiricalbound} we write $h_{\S}^j:=\sup_{f,g\in \S}|f^j-g^j|$ for a class $\S$ of functions. Therefore, for $A=k$ and $B=k+1$, and recalling that $\F_k=\{f_a\}_{a\in I_k}\cup\{0\}$, we obtain
\begin{align*}
\|h_{\F_k}^j\|_2&=\big\|\sup_{a,b\in I_k}|f_a^j(X,Y)-f_b^j(X,Y)|\vee (\sup_{a\in I_k}|f^j_a(X,Y)|)\big\|_2\leq 2\|\tilde{g}_k^j(X,Y)\|_2.
\end{align*} 
As $\{0\}\in \mathcal{A}_l$ for all $l\ge 1$, we thus have for $f=0$,
$$\sum_{l\geq 0}2^{l/2}\|h^j_{A_l(f)}\|_2=\sum_{l\geq 0}2^{l/2}\|h^j_{A_l(0)}\|_2=\|h^j_{\F_k}\|_2\leq 2\|\tilde{g}_k^j(X,Y)\|_2.$$
For $f\neq 0$ we again use \eqref{eq:h}, this time with adjusted bounds $A\le B$ such that $B-A=(\max(1,\floor{N_l^{1/d}}-1))^{-1}\cdot\bf{1}$, to obtain  
$$\sum_{l\geq 0 }2^{l/2}\|h^j_{A_l(f)}\|_2\leq \sum_{l\geq 0}2^{l/2}(\max(1,\floor{N_l^{1/d}}-1))^{-1}\|\tilde{h}_k^j(X,Y)\|_2.$$
Note that 
\begin{align*}
    \sum_{l\geq 0}2^{l/2} (\max(1,\floor{N_l^{1/d}}-1))^{-1}&\leq \sum_{l\leq \log_2 d}2^{l/2}+L\sum_{l>\log_2 d}2^{l/2-2^l/d}\\
    &\leq L2^{(\log_2 d)/2}+L\sum_{l\geq 0}2^{(l+\log_2 d)/2-2^l}\\
    &\leq L\sqrt{d}+L\sqrt{d}\sum_{l\geq 0}2^{l/2-2^l}\leq L\sqrt{d}.
\end{align*}
This yields
$$\sum_{l\geq 0 }2^{l/2}\|h_{A_l(f)}\|_2\leq L\sqrt{d}\|\tilde{h}_k^j(X,Y)\|_2.$$ 
It now follows from \eqref{eq:slln} and Lemma \ref{lemma:empiricalbound} that  for $|k|_2>\sigma$, 
\begin{align}
\E\big[\sup_{x\in I_k} |\xi_n^j(x)| \big]&= \frac{1}{n}\E\left[\sup_{f\in\F_k}\Big|\sum_{i=1}^n(f^j(X_i,Y_i)-\E[f^j(X,Y)])\Big|\right]\nonumber \\
&\leq Ln^{-1/2}\left(\sqrt{d}\|\tilde{h}_k^j(X,Y)\|_2+\|\tilde{g}_k^j(X,Y)\|_2\right).\label{eq:xinub}\end{align}
Combining with \eqref{eq:j} yields that  for $|k|_2>\sigma$, 
\begin{align}
    \E\big[\sup_{x\in I_k} |\xi_n(x)|_2 \big]\leq Ln^{-1/2}\sum_{j=1}^d\left(\sqrt{d}\|\tilde{h}_k^j(X,Y)\|_2+\|\tilde{g}_k^j(X,Y)\|_2\right).\label{eq:inse}
\end{align}
It then remains to bound $\|\tilde{h}_k^j(X,Y)\|_2$ and $\|\tilde{g}_k^j(X,Y)\|_2$.

Denote by $q$ the conjugate of $p>1$, so that $p/q=p-1$. Recalling the definition of $\tilde g_k^j$ given in \eqref{eq:h2}, an application of Markov's inequality leads to 
\begin{align*}
\E\left[\frac{\tilde{g}^j_k(X,Y)^{2q}}{|X^j-Y^j|^{2q}}\right]^{1/(2q)}&\leq \left(\P_0(\nn{X}_2\geq\nn{k}_2/2)(\frac{C_\rho}{\sigma^d})^{2q}+|f_{\xi,\rho,\sigma}(k/2)|^{2q}\right)^{1/(2q)}\\
&\leq (\frac{C_\rho}{\sigma^d})\E[\nn{X}_2^{2p}]^{1/(2q)}(\frac{\nn{k}_2}{2})^{-p/q}+|f_{\xi,\rho,\sigma}(k/2)|\\
&\leq \frac{C_\rho}{\sigma^d}\left(\E[\nn{X}_2^{2p}]^{1/(2q)}(\frac{\nn{k}_2}{2})^{-(p-1)}+(\frac{\nn{k}_2}{2\sigma})^{-\rho}\right).
\end{align*}
By H\"{o}lder's inequality for $1/2=1/(2p)+1/(2q)$ and since $\nn{X}_2\in L^{2p}$ we thus conclude from the above that
\begin{align}
\E[\tilde{g}_k^j(X,Y)^2]^{1/2}&\leq \n{X^j-Y^j}_{2p}\E\left[\frac{\tilde{g}^j_k(X,Y)^{2q}}{|X^j-Y^j|^{2q}}\right]^{1/(2q)}\nonumber\\
&\leq \n{X^j-Y^j}_{2p}\frac{C_\rho}{\sigma^d}\left(\E[\nn{X}_2^{2p}]^{1/(2q)}(\frac{\nn{k}_2}{2})^{-(p-1)}+(\frac{\nn{k}_2}{2\sigma})^{-\rho}\right).\label{eq:h3}
\end{align}
Similarly,
$$\E[\tilde{h}_k^j(X,Y)^2]^{1/2}\leq \n{X^j-Y^j}_{2p}\frac{\rho C_\rho}{\sigma^{d+1}}\left(\E[\nn{X}_2^{2p}]^{1/(2q)}(\frac{\nn{k}_2}{2})^{-(p-1)}+(\frac{\nn{k}_2}{2\sigma})^{-\rho-1}\right).$$
Inserting the above estimates into \eqref{eq:inse} leads to 
 \begin{align*}
        \E\big[\sup_{x\in I_k} |\xi_n(x)|_2 \big]&\leq Ln^{-1/2}\sum_{j=1}^d\n{X^j-Y^j}_{2p}\\
        &\hspace{0.5cm}\cdot\Bigg(\frac{\sqrt{d}\rho C_\rho}{\sigma^{d+1}}\left(\E[\nn{X}_2^{2p}]^{1/(2q)}(\frac{\nn{k}_2}{2})^{-(p-1)}+(\frac{\nn{k}_2}{2\sigma})^{-\rho-1}\right)\\
        &\hspace{2cm}+\frac{C_\rho}{\sigma^d}\left(\E[\nn{X}_2^{2p}]^{1/(2q)}(\frac{\nn{k}_2}{2})^{-(p-1)}+(\frac{\nn{k}_2}{2\sigma})^{-\rho}\right)\Bigg).
    \end{align*}
Rearranging the terms completes the proof. 
\end{proof}

\begin{proof}[Proof of Proposition \ref{prop:finitesample}]
Summing the estimate in Lemma \ref{lemma:Esup when gamma=1} over $k\in\Z^d$ satisfying $|k|_2\geq\sigma$, we obtain the contribution
\begin{align}\label{eq:tt1}
    \begin{split}
    &\hspace{0.5cm}\sum_{|k|_2\geq\sigma}\E\big[\sup_{x\in I_k} |\xi_n(x)|_2 \big]\\
&\leq Ln^{-1/2}\sum_{j=1}^d\n{X^j-Y^j}_{2p}C_\rho\Bigg[ {d^{3/2}\rho  }\left(\E[\nn{X}_{2}^{2p}]^{1/(2q)}2^{p-1}\frac{\sigma^{-p}}{p-1-d}+2^{\rho+1}\frac{\sigma^{-1}}{\rho+1-d}\right)\\
&\hspace{1cm}+ d\left(\E[\nn{X}_{2}^{2p}]^{1/(2q)}2^{p-1}\frac{\sigma^{-(p-1)}}{p-1-d}+2^{\rho}\frac{1}{\rho-d}\right)\Bigg].
\end{split}
\end{align}

It remains to estimate the contribution over $k\in\Z^d$ satisfying $|k|_2<\sigma$, where we use a similar argument as in the proof of Lemma \ref{lemma:Esup when gamma=1}. 
  Observe that one can replace $\tilde{h}_k^j(x,y)$ in \eqref{eq:h} by $\sup|\nabla f_{\xi,\rho,\sigma}||x-y|$ and $\tilde{g}_k^j(x,y)$ by $ \sup|f_{\xi,\rho,\sigma}||x-y|$.
  Inserting these bounds into \eqref{eq:xinub}, one finds that for $k\in\Z^d$ satisfying $|k|_2<\sigma$ and $j\in\{1,\dots,d\}$,  
   \begin{align*}
       \E\left[\sup_{x\in I_k}|\xi^j_n(x)|\right]&\leq Ln^{-1/2}({ \sup|\nabla f_{\xi,\rho,\sigma}|+\sup| f_{\xi,\rho,\sigma}|})\n{X-Y}_2\\
       &\leq Ln^{-1/2}(C_\rho(\sqrt{d}\rho\sigma^{-(d+1)}+\sigma^{-d}))\n{X-Y}_{2p}.
   \end{align*}
   Since $\#\{k\in\Z^d:|k|_2<\sigma\}\leq L(\sigma^d+1)$, we have
\begin{align}
    \sum_{|k|_2<\sigma}\E\left[\sup_{x\in I_k}|\xi_n(x)|\right]\leq Ln^{-1/2}\sum_{j=1}^d\n{X^j-Y^j}_{2p}C_\rho(\sqrt{d}\rho\sigma^{-(d+1)}+\sigma^{-d})(\sigma^d+1).\label{eq:tt2}
\end{align}
Combining \eqref{eq:tt1} and \eqref{eq:tt2} yields the proof.
\end{proof}

\subsection{Asymptotic distribution of the SE-MPD revisited}\label{sec:thm5}
In Section \ref{sec:asymp}, we state the asymptotic distribution of the SE-MPD for both the i.i.d.~case and the stationary $\alpha$-mixing case. For the i.i.d.~case, our main result was stated in Theorem \ref{thm:limitd simple}, which we rephrase with further details below. 

\begin{theorem}\label{thm:limitd}
Let $\gamma\geq 1$ and for $\rho>\gamma+d$, consider the density   $f_{\xi,\rho}$ given by \eqref{eq:fdensity}, i.e., $f_{\xi,\rho}(x)=C_\rho(|x|_2+1)^{-\rho}$. Suppose that there exists $\tilde{\delta}>0$  such that the $\R^d\times\R^d$-valued martingale coupling $(X,Y)\in L^{m+\tilde{\delta}}$ and one of the following holds:
\begin{enumerate}[(i)]
    \item $m=2(\gamma+d+\rho(\gamma-1))$,
    \item $m=2(d+1+\rho-\frac{\rho-1}{\gamma})$, and $\rho>\gamma d+1$.
\end{enumerate}
 Then $f_{\xi,\rho}$ is a martingality-preserving law and we have the convergence in distribution 
\begin{align}
n^{\gamma/2}\mpd^{*\xi}({\P}_n,\gamma)\ddd 2^{1-\gamma}\int_{\R^d}\frac{\nn{G_x}_2^\gamma}{\E[f_{\xi,\rho}(x-X)]^{\gamma-1}}\,\d x, \qquad n\to \infty,\label{eq:limitd}
\end{align}
where  $\{G_x\}$ is a centered $\R^d$-valued Gaussian random field with covariance 
\begin{align}\label{eq:Gx}
    \E[G_xG_y^\top]=\E[(Y-X) f_{\xi,\rho}(x-X)f_{\xi,\rho}(y-X)(Y-X)^\top],~x,y\in\R^d.
\end{align}
In particular, the sequence $\{n^{\gamma/2}\mpd^{*\xi}({\P}_n,\gamma)\}_{n\in \N}$ is tight.
\end{theorem}

\begin{remark}\label{remark:momentchoice}
    
As a consequence of Theorem \ref{thm:limitd}, the following holds:
\begin{enumerate}[(i)]
\item If $(X,Y)\in L^{2\min(\gamma d+2,\gamma+d+(\gamma-1)\max(d,2))+\tilde{\delta}}$ (in particular, $\in L^{2(\gamma d+\min(\gamma,2))}$ if $d\geq 2$) for some $\tilde{\delta}>0$, then there exists $\rho$ such that  \eqref{eq:limitd} holds.

\item If $(X,Y)\in L^{2(d+1)+\tilde{\delta}}$ for some $\tilde{\delta}>0$, then \eqref{eq:limitd} holds for $\gamma=1$ and any $\rho>\gamma+d$.

\item If all moments of $(X,Y)$ exist, then \eqref{eq:limitd} holds for any  $\gamma\geq 1$ and $\rho>\gamma+d$.
\end{enumerate}

\end{remark}

Let us also point out the positive dependence on $\gamma$ of the number of moments of $(X,Y)$. In particular, if $\gamma=1$, then the moment condition does not depend on $\rho$ (as it suffices to consider case (i)). In other words, the class $(X,Y)$ of ``permissible'' martingale couplings shrinks in size as $\gamma$ increases, and $\gamma=1$ is the optimal choice. For this reason, the case $\gamma=1$ is the most widely applicable (and turns out the most technically tractable as well), hence deserves a thorough study.

Using Lemma \ref{lemma:computation} we can provide the proof of Theorem \ref{thm:limitd}.

\begin{proof}[Proof of Theorem \ref{thm:limitd}] 
As $\rho> d+1$, Example \ref{ex:main} immediately implies that $f_{\xi,\rho}$ is martingality-preserving. 
Recall \eqref{eq:En expression}.
We fix $\ee>0$ and $y\in\R$. We first observe that
$$\lim_{K\to \infty} \int_{[-K,K]^d}\frac{\nn{G_x}_2^\gamma}{\E[f_{\xi,\rho}(x-X)]^{\gamma-1}}\,\d x = \int\frac{\nn{G_x}_2^\gamma}{\E[f_{\xi,\rho}(x-X)]^{\gamma-1}}\,\d x$$ in $L^1$, and thus also in distribution. Next, Lemma \ref{lemma:limsup} below implies
\begin{align*}
\int_{\R^d\setminus[-K,K]^d}\frac{\nn{\frac{1}{\sqrt{n}}\sum_{i=1}^n(Y_i-X_i)f_{\xi,\rho}\left(x-X_i\right)}_2^\gamma}{(\frac{1}{n}\sum_{i=1}^nf_{\xi,\rho}\left(x-X_i\right))^{\gamma-1}}\,\d x \ddd 0
\end{align*}
for $n\to \infty.$ The claim now follows from the above together with Lemma \ref{lemma:convergeK} below by taking limits $n\to \infty$ and then $K\to \infty.$
\end{proof}

We have used the following lemmas:

\begin{lemma}\label{lemma:convergeK}
    Suppose that $(X,Y)\in L^{1}$.  For each $K>0$,
    $$\int_{[{-K},K]^d} \frac{\nn{\frac{1}{\sqrt{n}}\sum_{i=1}^n(Y_i-X_i)f_{\xi,\rho}\left(x-X_i\right)}_2^\gamma}{(\frac{1}{n}\sum_{i=1}^nf_{\xi,\rho}\left(x-X_i\right))^{\gamma-1}}\,\d x\ddd \int_{[{-K},K]^d} \frac{\nn{G_x}_2^\gamma}{\E[f_{\xi,\rho}(x-X)]^{\gamma-1}}\,\d x.$$
\end{lemma}

\begin{proof}
Recall $f_a(x,y)=(y-x)f_{\xi,\rho}(a-x).$ As $[-K,K]^d$ is bounded, it follows from  \citep[Example 19.7]{van2000asymptotic}, that $\{f_a\}_{a\in [-K,K]^d}$ is Donsker. By the continuous mapping theorem,
$$\Big\{\big|\frac{1}{\sqrt{n}}\sum_{i=1}^n(Y_i-X_i)f_{\xi,\rho}\left(x-X_i\right)\big|_2^\gamma\Big\}_{x\in[-K,K]^d}\ddd\left\{ \nn{G_x}_2^\gamma\right\}_{x\in[-K,K]^d}$$
weakly in $L^\infty([-K,K]^d)$.
On the other hand, define $g_a(x):=f_{\xi,\rho}(a-x),~x\in\R^d$. Recall from \eqref{eq:gradient of f} that $f_{\xi,\rho}$ is Lipschitz, so again \citep[Example 19.7]{van2000asymptotic} implies that the bracketing number
$N_{[\,]}(\ee,\{g_a\}_{a\in[-K,K]^d}, L^1(\mu))$ is finite for every $\ee>0.$
Hence the class $\{g_a\}_{a\in[-K,K]^d}$ is Glivenko-Cantelli by \citep[Theorem 19.4]{van2000asymptotic}. Furthermore, 
the set $\{\E[f_{\xi,\rho}(x-X)]^{\gamma-1}:{x\in [-K,K]^d}\}$ is bounded away from zero. Combining these results and using Slutsky's theorem, this leads to
$$\frac{|\frac{1}{\sqrt{n}}\sum_{i=1}^n(Y_i-X_i)f_{\xi,\rho}\left(x-X_i\right)|_2^\gamma}{(\frac{1}{n}\sum_{i=1}^nf_{\xi,\rho}\left(x-X_i\right))^{\gamma-1}}\ddd \frac{\nn{G_x}_2^\gamma}{\E[f_{\xi,\rho}(x-X)]^{\gamma-1}}$$
weakly in $L^\infty([-K,K]^d)$. Applying the continuous mapping theorem yields the desired convergence in distribution.
\end{proof}

\begin{lemma}\label{lemma:limsup}
In the setting of Theorem \ref{thm:limitd}, for any $\ee,\eta>0$ there exists $K_0>0$ such that for any $K>K_0$,

$$\limsup_{n\to\infty} \mu\Big(\int_{\R^d\setminus[-K,K]^d}\frac{|\frac{1}{\sqrt{n}}\sum_{i=1}^n(Y_i-X_i)f_{\xi,\rho}\left(x-X_i\right)|_2^\gamma}{(\frac{1}{n}\sum_{i=1}^nf_{\xi,\rho}\left(x-X_i\right))^{\gamma-1}}\,\d x>\eta \Big)<\ee.$$
\end{lemma}

We now detail the proof of Lemma \ref{lemma:limsup}, which requires a few preliminary results. 
We introduce the quantities
$$S_\xi=S_\xi(k):=\sup_{x\in I_k}\n{(Y-X)f_{\xi,\rho}(x-X)}_2$$
and $$M_{\xi,p}=M_{\xi,p}(k):=\big\|\sup_{x\in I_k}(Y-X)f_{\xi,\rho}(x-X)\big\|_p,~p\geq 1.$$
Note that $S_\xi\leq M_{\xi,2}$.

The next lemma is \citep[Theorem 3.1]{lederer2014new} applied with $\ee=1$ and $\ell=p/2$ therein.
\begin{lemma}\label{lemma:momentbound}
    Suppose that $M_{\xi,p}<\infty$ for some $p\geq 2$. It holds for $\nn{k}_2>5$ that
    $$\big\|\big(\sup_{x\in I_k}|\xi_n(x)|-2\E\big[\sup_{x\in I_k}|\xi_n(x)|\big]\big)_+\big\|_{p/2}\leq \frac{55\sqrt{p}}{\sqrt{n}}M_{\xi,p}(k)+\frac{3\sqrt{p}}{\sqrt{n}}S_\xi(k).$$
\end{lemma}

\begin{lemma}\label{lemma:binomial}

    Let $\{Z_n\}_{n\in\N}$ be a  sequence of i.i.d.~random variables with $Z_n\lawis \mathrm{Ber}(p)$. Then
    $$\lim_{p\to 1}\mu\big(\sum_{i=1}^n Z_i>\frac{n}{2}\text{ for all }n\in \N \big)=1.$$
\end{lemma}

\begin{proof}Let us introduce the events  
\begin{align*}
A &:=\{|\{i\in \{1, \dots, n\}: Z_i=1\}|>n/2\text{ for all }n\},\\
A_n &:= \{|\{i\in \{1, \dots, n\}: Z_i=1\}|>n/2\}.
\end{align*}
A union bound yields
\begin{align*}
    \mu(A^c)&\le \mu\bigg(\bigcup_{n=1}^{(1-p)^{-1/2}} A_n^c\bigg)+\mu\big(\bigcup_{n=\floor{(1-p)^{-1/2}}}^{\infty} A_n^c\big)\\
    &\leq \mu(Z_n=0 \text{ for some } 1\le n\le (1-p)^{-1/2}) +\sum_{n=\floor{(1-p)^{-1/2}}}^\infty \mu\big(\sum_{i=1}^n Z_i\leq \frac{n}{2}\big)\\
    &\leq \sqrt{1-p}+\sum_{n=\floor{(1-p)^{-1/2}}}^\infty \mu\big(\eta_{n,p}\leq \frac{n}{2}\big),
\end{align*}
where $\eta_{n,p}\lawis \mathrm{Bin}(n,p)$. 
As we are interested in the limit $p\to 1$ we can assume without loss of generality that $p>3/4$. Let us also recall 
Hoeffding's inequality, which states that for all $t> 0$ we have 
$$\mu (np-\eta_{n,p}\ge t)\le \exp\big(-\frac{2t^2}{n}\big).$$
We thus obtain for $t=n(p-1/2)$, 
$$\mu\big(\eta_{n,p}\leq \frac{n}{2}\big) \le
\mu\big(np-\eta_{n,p}\ge n(p-\frac{1}{2})\big)
\leq \exp\big(-2n(p-\frac{1}{2})^2\big).$$
Combining the above and using the geometric sum formula leads to
\begin{align*}
    \mu(A^c)&\leq\sqrt{1-p}+\sum_{n=\floor{(1-p)^{-1/2}}}^\infty \exp\big(-2n(p-\frac{1}{2})^2\big)\\
    &=\sqrt{1-p}+\exp\big(- 2\floor{(1-p)^{-1/2}}(p-\frac{1}{2})^2\big)\sum_{n=0}^\infty \exp\big(-2n(p-\frac{1}{2})^2\big)\\
    &\leq \sqrt{1-p}+\exp\big(-\frac{\floor{(1-p)^{-1/2}}}{8}\big)\sum_{n=0}^\infty e^{-n/8}.
\end{align*}
As $p\to 1$, the right-hand side tends to $0$. This shows $\mu(A)\to 1$ as $p\to 1$.
\end{proof}

\begin{lemma}\label{lemma:M}
    We have 
    $$M_{\xi,2\gamma}^\gamma\leq C(|k|_2^{\gamma-(m+\tilde{\delta})/2}+|k|_2^{-\gamma\rho}). $$
\end{lemma}
\begin{proof}
    Recall from \eqref{eq:h2} in  the proof of Lemma \ref{lemma:Esup when gamma=1} that 
    $$M_{\xi,2\gamma}^\gamma=\E\big[\sup_{a\in I_k}|f_a(X,Y)|^{2\gamma}\big]^{1/2}\leq \n{\tilde{g}_k(X,Y)}_{2\gamma}^{\gamma}.$$
    We use a similar argument as in Lemma \ref{lemma:Esup when gamma=1} to bound the right-hand side using H\"{o}lder's and Markov's inequalities. Assume that $(X,Y)\in L^{2\gamma p}$, where $p=(m+\tilde{\delta})/(2\gamma)>1$. Using our definition \eqref{eq:h2},
    \begin{align*}
    \E[\tilde{g}_k^j(X,Y)^{2\gamma}]&\leq C(\E[|Y^j-X^j|^{2\gamma}\mathds{1}_{\{\nn{X}_2\geq\nn{k}_2/2\}}]+\nn{k}_2^{-2\rho\gamma}\E[|Y^j-X^j|^{2\gamma}])\\
        &\leq C\big(\mu(\nn{X}_2\geq \frac{\nn{k}_2}{2})^{1/q}+\nn{k}_2^{-2\rho\gamma}\big)\\
        &\leq C(\nn{k}_2^{-2\gamma(p-1)}+\nn{k}_2^{-2\rho\gamma}).
    \end{align*}
    The claim thus follows.
\end{proof}

\begin{proof}[Proof of Lemma \ref{lemma:limsup}]Recall our notation \eqref{eq:Ik}. 
 We first  bound $\E\left[\sup_{x\in I_k}|\xi_n(x)|_2^\gamma\right]$, where $\gamma> 1$.   By Jensen's inequality,  $S_\xi\leq M_{\xi,2}\leq M_{\xi,2\gamma}$. Using the triangle inequality and Lemmas \ref{lemma:Esup when gamma=1} (applied with $\sigma=1$)  and \ref{lemma:momentbound}, we obtain
\begin{align*}
    \big\|\sup_{x\in I_k} |\xi_n(x)|_2 \big\|_{\gamma}&\leq 2\E\big[\sup_{x\in I_k} |\xi_n(x)|_2 \big]+\big\|\big(\sup_{x\in I_k} |\xi_n(x)|_2-2\E[\sup_{x\in I_k}|\xi_n(x)|_2]\big)_+\big\|_{\gamma}\\
    &\leq \frac{C}{\sqrt{n}}(|k|_2^{1-(m+\tilde{\delta})/2}+|k|_2^{-\rho})+\frac{C}{\sqrt{n}}M_{\xi,2\gamma}.
\end{align*}
Lemma \ref{lemma:M} states that
\begin{align}
    \E\big[\sup_{x\in I_k}  |\xi_n(x)|_2^\gamma\big]&\leq Cn^{-\gamma/2}(|k|_2^{\gamma-(m+\tilde{\delta})/2}+|k|_2^{-\gamma\rho}).\label{eq:xinub2}
\end{align}

 We now bound from below the quantity $$p_n(x)^{\gamma-1}= \Big(\frac{1}{n} \sum_{i=1}^n f_{\xi,\rho}(x-X_i)\Big)^{\gamma-1}.$$ We consider the event 
\begin{align*}
A_k&:=\bigcap_{n=1}^\infty A_{n,k}:=\bigcap_{n=1}^\infty\left\{|\{i\in \{1, \dots,n\}: |X_i|_2  \le |k|_2 /2\}|>\frac{n}{2}\right\},\qquad |k|_2 >5.
\end{align*} 
For $|x|_2 >5/2$ we have $(1+|x|_2)^{-\rho}\geq |x|_2^{-\rho}/C$. As $$|x-X_i|_2\ge |x|_2-|X_i|_2,$$ the inequality
 \begin{align}\label{eq:denominator lb}
    \inf_{x\in I_k}\frac{1}{n}\sum_{i=1}^nf_{\xi,\rho}(x-X_i)\geq \frac{1}{C}(1+|k|_2/2)^{-\rho}\geq \frac{|k|_2 ^{-\rho}}{C} 
\end{align}holds on the event $A_k$ 
for any $n\in\N$.
In addition, since $p_k:=\mu(|X|_2  \leq k)\to 1$ as $k\to\infty$ and $\{X_n\}$ are independent, Lemma \ref{lemma:binomial} yields  $\mu(A_k)\to 1$ as $\nn{k}_2 \to\infty$.

 For $\ee>0$, pick $K_1$ large such that for all $k\geq K_1$, $\mu[A_{k}^c]<\ee/2$. We now distinguish the two cases stated in the theorem:\\

\emph{Case I}: take $m=2(\gamma+d+\rho(\gamma-1))$, where $\rho>d+\gamma$.\footnote{Let us recall that the lower bound for $\rho$ is needed from Theorem \ref{prop:martingale}.}  Next, for $K>5$ we have by using \eqref{eq:denominator lb} and \eqref{eq:xinub2},
\begin{align*}
    \E\left[\mathds{1}_{A_{K_1}}\int_{\R^d\setminus[-K,K]^d}\frac{\nn{\xi_n(x)}_2^\gamma}{p_n(x)^{\gamma-1}}\d x\right]&\leq C\sum_{\nn{k}_2 \geq K}\frac{\E[\sup_{x\in I_k}\nn{\xi_n(x)}_2^\gamma]}{\nn{k}_2  ^{-\rho(\gamma-1)}}\\
    &\leq Cn^{-\gamma/2}\sum_{\nn{k}_2 \geq K}\frac{\nn{k}_2  ^{\gamma-(m+\tilde{\delta})/2}+\nn{k}_2  ^{-\gamma\rho}}{\nn{k}_2  ^{-\rho(\gamma-1)}}\\
    &\leq Cn^{-\gamma/2}(K^{\gamma-(m+\tilde{\delta})/2+d+\rho(\gamma-1)}+K^{d-\rho}),
\end{align*}which converges to $0$ as $K\to\infty$ for fixed $n$. By Markov's inequality, there exists $K_2>K_1$ such that for all $K\geq K_2$,
$$\mu\Bigg(\Bigg\{\int_{\R^d \setminus[-K,K]^d}\frac{|\frac{1}{\sqrt{n}}\sum_{i=1}^n(Y_i-X_i)f_{\xi,\rho}\left(x-X_i\right)|_2^\gamma}{(\frac{1}{n}\sum_{i=1}^nf_{\xi,\rho}\left(x-X_i\right))^{\gamma-1}}\,\d x>\eta\Bigg\}\cap A_{K_1}\Bigg)<\frac{\ee}{2}.$$
The desired statement is then immediate.\\

\emph{Case II}: take $m=2(d+1+\rho-(\rho-1)/\gamma)$. Recall that $(X,Y)\in L^{m+\tilde{\delta}}$   and  $\rho>\gamma d+1$.  
Define the random variables
$$\zeta_{n,k}:= \sup_{x\in I_k}\nn{\sqrt{n}\xi_n(x)}_2=\sup_{x\in I_k}\big|\frac{1}{\sqrt{n}}\sum_{i=1}^n(Y_i-X_i)f_{\xi,\rho}\left(x-X_i\right)\big|_2,~k\in\Z^d,~n\in\N.$$
It follows from Markov's inequality and Lemma \ref{lemma:Esup when gamma=1} that for any $\alpha>0$,
\begin{align}
    \mu(\zeta_{n,k}\geq \alpha)\leq C \alpha^{-1}\left(\nn{k}_2  ^{-2\rho}+\nn{k}_2  ^{2-(m+\tilde{\delta})}\right)^{1/2}.\label{eq:znk}
\end{align}
Therefore, by the union bound and   \eqref{eq:denominator lb}, there exists $\delta>0$ small enough such that for any $\ee>0$, for $n$ large we have
\begin{align*}\mu\left(\int_{\R^d\setminus[-K,K]^d}\frac{\nn{\xi_n(x)}_2^\gamma}{p_n(x)^{\gamma-1}}\,\d x>\eta n^{-\gamma/2} \right)    &\leq \mu\left(\sum_{\nn{k}_2  \geq K}\frac{\sup_{x\in I_k}\nn{\xi_n(x)}_2^\gamma}{\inf_{x\in I_k}p_n(x)^{\gamma-1}}>\eta n^{-\gamma/2} \right)\\
    &\leq \sum_{\nn{k}_2  \geq K}\mu\left(\frac{\zeta_{n,k}^\gamma}{\nn{k}_2  ^{-\rho(\gamma-1)}}>\frac{\eta \nn{k}_2  ^{-1-\delta}}{C} \right)+\ee\\
    &\leq C\sum_{\nn{k}_2  \geq K}\frac{(\nn{k}_2  ^{-2\rho}+\nn{k}_2  ^{2-(m+\tilde{\delta})})^{1/2}}{(\eta \nn{k}_2  ^{-(\rho(\gamma-1)+1+\delta)})^{1/\gamma}}+\ee.
\end{align*}
Since $m=2(d+1+\rho-\frac{\rho-1}{\gamma})$ and $\rho>\gamma d+1$, we have for $\delta>0$ small enough that \begin{align*}
    \sum_{\nn{k}_2  \geq K}&\frac{(\nn{k}_2  ^{-2\rho}+\nn{k}_2  ^{2-(m+\tilde{\delta})})^{1/2}}{(\eta \nn{k}_2  ^{-(\rho(\gamma-1)+1+\delta)})^{1/\gamma}}  \\
    &\leq {C{{\eta}^{-1/\gamma}} (K^{d-({\rho-1-\delta})/{\gamma}}+K^{d+1-(m+\tilde{\delta})/2+(\rho(\gamma-1)+1+\delta)/\gamma})}\to 0
\end{align*}as $K\to\infty$. Thus the claim follows.
\end{proof}

In fact, the expectation of the limit distribution in Theorem \ref{thm:limitd} exerts asymptotic behavior as $\sigma$ approaches infinity. This observation is stated in Theorem \ref{thm:expectation}. 

\begin{proof}[Proof of Theorem \ref{thm:expectation}]
Observe first that by Fubini's theorem and standard properties for the normal distribution,
\begin{align}
    \E\left[\int\nn{G_x}_2\,\d x\right]&=\int \E[\nn{G_x}_2]\d x   =\int \E\left[\sqrt{\sum_{j=1}^d\lambda_j Z_j^2}\right]\d x,\label{eq:limitexpectation}
\end{align}
where $\lambda_1(x)\geq\lambda_2(x)\geq\dots\geq\lambda_d(x)$ are the (non-negative) eigenvalues of the covariance matrix $\E[G_xG_x^\top]$, and $Z_1,\dots,Z_d$ are i.i.d.~standard Gaussian. 
Let us   recall two facts from linear algebra that give upper and lower bounds on the largest eigenvalue: 
\begin{enumerate}[(a)]
    \item Since $\mathrm{Tr}(\E[G_xG_x^\top])=\sum_{j=0}^d\lambda_j(x)$,
    \begin{align}
        \lambda_1(x)\geq \frac{1}{d}\mathrm{Tr}(\E[G_xG_x^\top])=\frac{1}{d}\sum_{1\leq j\leq d}\E[(G_x)_j^2]=\frac{1}{d}\E\left[\sum_{i=1}^d \big(Y_i-X_i\big)^2f_{\xi,\rho,\sigma}(x-X)^2\right].\label{eq:linalg1}
    \end{align}
    \item By the  Gershgorin circle theorem (Theorem 6.1.1 of \citep{horn2012matrix}), 
    \begin{align}
        \lambda_1(x)\leq \max_{1\leq i\leq d}\sum_{1\leq j\leq d}|\E[(G_x)_i(G_x)_j]|\leq \E\left[\Big(\sum_{i=1}^d |Y_i-X_i|\Big)^2f_{\xi,\rho,\sigma}(x-X)^2\right].\label{eq:linalg2}
    \end{align}
\end{enumerate}

We first prove the lower bound, where it follows from \eqref{eq:limitexpectation} and \eqref{eq:linalg1} that 
$$ \E\left[\int\nn{G_x}_2\,\d x\right]\geq \int\E[\sqrt{\lambda_1(x)}|Z_1|]\d x\geq \frac{1}{C}\int \sqrt{\E\left[\sum_{i=1}^d \big(Y_i-X_i\big)^2f_{\xi,\rho,\sigma}(x-X)^2\right]}\d x.$$Note that there exists $C>0$ such that $f_{\xi,\rho}(x/\sigma)\ge C (|x|_2/\sigma)^{-\rho}$  and $(x-X)^{-2\rho}\mathds{1}_{\{-\sigma<X<\sigma\}}\ge \sigma^{-2\rho}/C$ for all $x\in(2\sigma,3\sigma)$. 
By \eqref{eq:sigma} we then have for $\sigma\ge 1$ and $\nn{x}_2\in(2\sigma,3\sigma)$,
\begin{align*}
    \E\left[\sum_{i=1}^d \big(Y_i-X_i\big)^2f_{\xi,\rho,\sigma}(x-X)^2\right]&\geq \E\left[\sum_{i=1}^d \big(Y_i-X_i\big)^2f_{\xi,\rho,\sigma}(x-X)^2\mathds{1}_{\{-\sigma<\nn{X}_2<\sigma\}}\right]\\
    &\geq \frac{\sigma^{2(\rho-d)}}{C}\E\left[\sum_{i=1}^d \big(Y_i-X_i\big)^2\nn{x-X}_2^{-2\rho}\mathds{1}_{\{-\sigma<\nn{X}_2<\sigma\}}\right]\\
    &\geq \frac{\sigma^{-2d}}{C}\E\left[\sum_{i=1}^d \big(Y_i-X_i\big)^2 \mathds{1}_{\{-\sigma<\nn{X}_2<\sigma\}}\right]\geq \frac{\sigma^{-2d}}{C},
\end{align*}where the last step follows from our assumption that   $\sum_{i=1}^d(Y_i-X_i)^2$ is not a constant zero. 
Therefore,
\begin{align*}
    \E\left[\int\nn{G_x}_2\,\d x\right]&\geq \frac{1}{C}\int_{2\sigma\leq \nn{x}_2\leq 3\sigma}\sqrt{\E\left[\sum_{i=1}^d \big(Y_i-X_i\big)^2f_{\xi,\rho,\sigma}(x-X)^2\right]}\d x\\
    &\ge \sigma^d \sqrt{\frac{\sigma^{-2d}}{C}}\geq \frac{1}{C}.
\end{align*}
For the upper bound, note that by \eqref{eq:limitexpectation} and \eqref{eq:linalg2},
$$ \E\left[\int\nn{G_x}_2\,\d x\right]\leq \int\E[\sqrt{\lambda_1(x)}\nn{Z}_2]\d x\leq C\int\sqrt{\E\left[\Big(\sum_{i=1}^d |Y_i-X_i|\Big)^2f_{\xi,\rho,\sigma}(x-X)^2\right]}\d x.$$
By \eqref{eq:sigma},
\begin{align}
   \E\bigg[\Big(\sum_{i=1}^d |Y_i-X_i|\Big)^2f_{\xi,\rho,\sigma}(x-&X)^2\bigg]\leq C \sigma^{-2d}\E\left[\Big(\sum_{i=1}^d |Y_i-X_i|\Big)^2\mathds{1}_{\{\nn{x-X}_2<\sigma\}}\right]\nonumber\\
    &+\sigma^{2(\rho-d)}\E\left[\Big(\sum_{i=1}^d |Y_i-X_i|\Big)^2f_{\xi,\rho}(x-X)^2\mathds{1}_{\{\nn{x-X}_2>\sigma\}}\right].\label{eq:phisigma}
\end{align}
First, if $|x|_2<2\sigma$,
\begin{align*}
    &\hspace{0.5cm}\sigma^{-2d}\E\bigg[\Big(\sum_{i=1}^d |Y_i-X_i|\Big)^2\mathds{1}_{\{\nn{x-X}_2<\sigma\}}\bigg]\\
    &\hspace{3cm}+\sigma^{2(\rho-d)}\E\bigg[\Big(\sum_{i=1}^d |Y_i-X_i|\Big)^2f_{\xi,\rho}(x-X)^2\mathds{1}_{\{\nn{x-X}_2>\sigma\}}\bigg]\\
    &\leq C \bigg(\sigma^{-2d}\E\bigg[\Big(\sum_{i=1}^d |Y_i-X_i|\Big)^2\bigg]+\sigma^{2(\rho-d)}\E\bigg[\Big(\sum_{i=1}^d |Y_i-X_i|\Big)^2\nn{x-X}_2^{-2\rho}\mathds{1}_{\{\nn{x-X}_2>\sigma\}}\bigg]\bigg)\\
    &\leq C\sigma^{-2d}.
    \end{align*}
This gives \begin{align}
    \int_{\{\nn{x}_2<2\sigma\}} \sqrt{\E\left[\Big(\sum_{i=1}^d |Y_i-X_i|\Big)^2f_{\xi,\rho,\sigma}(x-X)^2\right]}\d x\leq C \int_{\{\nn{x}_2<2\sigma\}}\,\d x\sigma^{-d}\leq C.\label{eq:int1}
\end{align}
Second, consider $\nn{x}_2>2\sigma$.
Suppose that $(X,Y)\in L^{2p}$ and $p^{-1}+q^{-1}=1$, where we may assume $p>2$. By H\"{o}lder's inequality,
\begin{align*}
    \E\left[\Big(\sum_{i=1}^d |Y_i-X_i|\Big)^2\mathds{1}_{\{\nn{x-X}_2<\sigma\}}\right]&\leq  \E\left[\Big(\sum_{i=1}^d |Y_i-X_i|\Big)^{2p}\right]^{1/p}\P_0(\nn{X}_2>\nn{x}_2-\sigma)^{1/q}\\
    &\leq C(\nn{x}_2-\sigma)^{-2p/q}\leq C\nn{x}_2^{-2p/q},
\end{align*}where we used that $\nn{X}_2\in L^{2p}$ and $\sum|Y_i-X_i|\in L^{2p} $.  Again by H\"{o}lder's inequality,
\begin{align*}
\E\left[\Big(\sum_{i=1}^d |Y_i-X_i|\Big)^2f_{\xi,\rho}(x-X)^2\mathds{1}_{\{\nn{x-X}_2>\sigma\}}\right]^{q}&\leq C\E[f_{\xi,\rho}(x-X)^{2q}\mathds{1}_{\{\nn{x-X}_2>\sigma\}}]\\
    &\leq C \E[\nn{x-X}_2^{-2q\rho}\mathds{1}_{\{\nn{x-X}_2>\sigma\}}].
    \end{align*}
    In addition,
    \begin{align*}
    \E[\nn{x-X}_2^{-2q\rho}\mathds{1}_{\{\nn{x-X}_2>\sigma\}}]&= \int_0^{\sigma^{-2q\rho}} \P_0(\nn{x-X}_2^{-2q\rho} \ge y)\,\d y\\
    &=\int_0^{\sigma^{-2q\rho}}\P_0(\nn{x-X}_2\leq y^{-1/(2q\rho)})\,\d y\\
    &\leq  \int_{2\nn{x}_2^{-2q\rho}}^{\sigma^{-2q\rho}}\P_0(\nn{X}_2\geq \nn{x}_2-y^{-1/(2q\rho)})\,\d y+{2\nn{x}_2^{-2q\rho}}\\
    &\leq C(\nn{x}_2^{-2p}\sigma^{-2q\rho}+\nn{x}_2^{-2q\rho}).
\end{align*}
We conclude using \eqref{eq:phisigma} that 
$$\E\left[\Big(\sum_{i=1}^d |Y_i-X_i|\Big)^2f_{\xi,\rho,\sigma}(x-X)^2\right]\leq C (\sigma^{-2d}\nn{x}_2^{-2p/q}+\sigma^{2(\rho-d)}\nn{x}_2^{-2\rho}).$$
Recall we assumed $q<2<p$. This then yields
\begin{align}\begin{split}
   &\hspace{0.5cm}\int_{\{\nn{x}_2>2\sigma\}} \sqrt{\E\left[\Big(\sum_{i=1}^d |Y_i-X_i|\Big)^2f_{\xi,\rho,\sigma}(x-X)^2\right]}\d x\\
   &\leq C \int_{\{\nn{x}_2>2\sigma\}}(\sigma^{-d}\nn{x}_2^{-p/q}+\sigma^{\rho-d}\nn{x}_2^{-\rho})\,\d x\leq C (\sigma^{-p/q}+1)\leq C.
\end{split}\label{eq:int2}
\end{align}
    Combining \eqref{eq:int1} and \eqref{eq:int2} gives the upper bound
    $$\int \E\left[\sqrt{\sum_{j=1}^d\lambda_j Z_j^2}\right]\d x\leq {C},$$and hence finishing the proof by \eqref{eq:limitexpectation}.
\end{proof}

A comparable limit distribution result also exists for stationary $\alpha$-mixing sequences. In our proof of Theorem \ref{thm:mixingconvergence}, we follow a similar path to the proof of Theorem \ref{thm:limitd}, starting from an empirical bound. And similar to our proof for finite-sample rates in Appendix \ref{sec:tech_dev_finite_and_general}, we start by considering \eqref{eq:Ik} and \eqref{eq:xindef}. For $k\in\R^d$, we let $\F_k=\{f_a\}_{a\in I_k}\cup\{0\}$ and $\F_k^j=\{f_a^j\}_{a\in I_k}\cup\{0\}$, where $f_a(x,y)=(y-x)f_{\xi,\rho}(a-x),~x,y\in\R^d$ and $f_a^j$ denotes the $j$-th component of $f_a$ for $1\leq j\leq d$. The following lemma parallels Lemma \ref{lemma:Esup when gamma=1} in Appendix \ref{sec:tech_dev_finite_and_general}. 

\begin{lemma}\label{lemma:numerator}
    Suppose that  for parameters $\lambda >2$, $p>1$, $r\in(1,p\lambda-1)$, and $s\in[2,\infty)\cap((dr-2)/(r-1),\infty)$, it holds that    
    $\{(X_i,Y_i)\}_{i\in\N}$ forms a stationary $\alpha$-mixing sequence in $L^{p\lambda }$ with mixing coefficients $\{\alpha_n\}_{n\in\N}$ satisfying $A_{\alpha,\lambda }<\infty$.
    Then for some constant $C(s,r)$ depending on $s,r$ only  and $C(s)$ depending on $s$ only, we have for each $k\in\Z^d$ such that $\nn{k}_2\geq\sigma$, 
\begin{align}
&\E\left[\sup_{x\in I_k}\nn{\xi_n(x)}_2\right]\nonumber\leq d\Big(LC_1(k)+C(s)C_2(k)^{1/s}+(1-2^{d/s-1})^{-1}C_3(k)^{1/d}C_2(k)^{(1/d-1/s)/2}\nonumber\\
    &\hspace{10cm}+C(s,r)C_4(k)^{1/(1+\beta)}\Big),\label{eq:mixingfinitesample}
\end{align}  where the quantities  $C_j(k),\,1\leq j\leq 4$ are explicitly defined in \eqref{eq:C2}, \eqref{eq:deltajk}, \eqref{eq:Nk}, and \eqref{eq:ed} below respectively. 
    In particular, there exists $C>0$ independent of $k$ and $n$ such that for each $\nn{k}_2\geq\sigma$, 
    \begin{align}
        \E\left[\sup_{x\in I_k}\nn{\xi_n(x)}_2\right]\leq C\nn{k}_2^{-\min(c_1,c_2,c_3)},\label{eq:mixingC}
    \end{align}
    where:\begin{itemize}
    \item $c_1:=\min(\rho(r+1),(p\lambda-r-1))$ is the negative exponent of $\nn{k}_2$ for the term $C_1(k)$; see \eqref{eq:C2};
        \item $c_2:=\min(\rho/s,(p-1)/s)$ is the negative exponent for $C_2(k)$; see \eqref{eq:deltajk};
        
        \item $c_3:=(\min(\rho+1,p-1)+(\frac{1}{2d}-\frac{v}{2s})\min(\rho,p-1))/(1+\beta)$ is the negative exponent for $C_4(k)^{1/(1+\beta)}$; see \eqref{eq:ed};\footnote{The remaining term $C_3(k)^{1/d}C_2(k)^{(1/d-1/s)/2}$ carries a  larger negative exponent in $\nn{k}_2$ than the term $C_4(k)^{1/(1+\beta)}$.}
        \item $\beta:=-(1/r-1)(1-2/s)\in[0,1)$;
        \item $v:=(1-\frac{1}{r}(1-\frac{2}{s}))^{-1}\in[1,\infty)$.
    \end{itemize}
\end{lemma}

\begin{remark}\label{remark:solution}
    For each fixed $d$ and $\lambda$, it is possible to solve (at least numerically) for the feasible pairs of 
    \begin{align}
        (p,\rho,r,s)~~~~~\mathrm{s.t.}~~~~~\min(c_1,c_2,c_3)>d.\label{eq:system of ineqs}
    \end{align} It is also elementary to check that for each $\rho>2d^2$, 
     with the choices $p=2d^2+2,~r=2,~s=2d$, $(p,\rho,r,s)$ forms a feasible pair.  In other words, in the setting of Theorem \ref{thm:mixingconvergence}, we may replace \eqref{eq:mixingC} by 
     \begin{align}
         \E\left[\sup_{x\in I_k}\nn{\xi_n(x)}_2\right]\leq C\nn{k}_2^{-(d+\delta)}\label{eq:mixingC2}
     \end{align}for some $\delta>0$.
\end{remark}

To establish Lemma \ref{lemma:numerator}, the following empirical bound serves as the analogue of Lemma \ref{lemma:empiricalbound} for $\alpha$-mixing sequences. This arises from a more general result proven along the way in \citep[Theorem 2 and Corollary 1]{hariz2005uniform}, which applies to the special case of $\alpha$-mixing sequences; see also \citep{rio2017asymptotic} for relevant literature.

\begin{lemma}[\citep{hariz2005uniform,rio2017asymptotic}]\label{lemma:alpharate}
Let $\{X_i\}_{i\in\N}$ be a stationary  $\alpha$-mixing sequence, and $\F$ be a class of  real-valued  functions. For $f\in\F$, define the norm
$$\n{f}_{2,X}:=\sqrt{\int_0^1\alpha^{-1}(u)Q_f(u)^2\d u},$$
    where $Q_f$ is the quantile of $|f(X_1)|$. Suppose that $F:=\sup_{f\in\F}|f|\in L^{r+1}$ for some $r>1$ and $$\int_0^1 N(x,\n{\cdot}_{2,X},\F)^{v/s}\d x<\infty,$$
    where $s\geq 2$ and  $v=(1-\frac{1}{r}(1-\frac{2}{s}))^{-1}$. 
Then {for all $\delta>0$},
\begin{align}
    &\E\left[\sup_{f,g\in\F,\,\n{f-g}_{2,X}<\delta}\left|\frac{1}{\sqrt{n}}\sum_{i=1}^n (f(X_i)-g(X_i))\right|\right]\nonumber\\
    \begin{split}
        &\leq C(s,r)\ee(\delta)^{\frac{1}{1+\beta}}+C(s)\delta^{1/s}+C(s)\sum_{\ell=q_0+1}^{q+1}N(2^{-\ell},\n{\cdot}_{2,X},\F)^{1/s}2^{-\ell}\\
    &\hspace{3cm}+L\sqrt{n}\,\E[F\mathds{1}_{F>n^{1/(2r)}}],
    \end{split}\label{eq:est}
\end{align}
where {we have used the following definitions:}
\begin{itemize}
    \item $q_0=q_0(\delta)$ is the largest integer with $N(2^{-q_0},\n{\cdot}_{2,X},\F)\leq 1/\sqrt{\delta}$;
    \item $\ee(\delta)=\int_0^{2^{-q_0}}{N(x,\n{\cdot}_{2,X},\F)}^{v/s}\d x$;
    \item $q=\lfloor \log(n\ee(\delta)^{1/(1+\beta)})/(2\log 2)\rfloor+1$;
    \item $\beta=-(1/r-1)(1-2/s)$.
\end{itemize} 
\end{lemma}

We also record a preliminary estimate for the last term in \eqref{eq:est}, proving that it has power decay in $k$ and is $O(1)$ in $n$. Eventually, we will show the same for the right-hand side of \eqref{eq:est}.
\begin{lemma}\label{lemma:r}
    In the above setting, suppose that $\F=\F_k^j$ and $(X,Y)\in L^{p\lambda}$, then for $1<r<p\lambda-1$, there exists $C>0$ independent of $n,k$ such that $\sup_{f\in\F_k^j}|f|\in L^{r+1}$ and 
    $$\sqrt{n}\,\E[F\mathds{1}_{F>n^{1/(2r)}}]\leq C_1(k)\leq C\left( \nn{k}_2^{-\rho(r+1)}+\nn{k}_2^{-(p\lambda-r-1)}\right),$$where $F=F_k^j=\sup_{f\in\F_k^j}|f|$ and $C_1(k)$ is defined in \eqref{eq:C2} below.
\end{lemma}

\begin{proof}
We first estimate $\E[|F|^{r+1}]$. Recall from \eqref{eq:h2} that $|F(x,y)|\leq \tilde{g}_k^j(x,y)$ and hence using H\"{o}lder's and Markov's  inequalities with $q'=p\lambda/(r+1)$ and $p'$ being conjugates (similarly as the derivation of \eqref{eq:h3}), 
\begin{align}
   &\hspace{0.5cm} \E[|F|^{r+1}]\nonumber\\
   &\leq \E\left[\Big(\frac{\tilde{g}_k^j(X,Y)}{|X^j-Y^j|}\Big)^{(r+1)p'}\right]^{1/p'}\E[|X^j-Y^j|^{(r+1)q'}]^{1/q'}\nonumber\\
    &\leq \left(|f_{\xi,\rho}(k/2)|^{(r+1)p'}+(C_\rho\sigma^{-d})^{(r+1)p'}\E[\nn{X}_2^{\lambda p}](\nn{k}_2/2)^{-\lambda p}\right)^{1/p'}\E[|X^j-Y^j|^{p\lambda}]^{1/q'}\nonumber\\
    &\leq \left((\frac{\nn{k}_2}{2\sigma}+1)^{-\rho(r+1)}+\E[\nn{X}_2^{\lambda p}]^{1/p'}(\frac{\nn{k}_2}{2})^{-(p\lambda-r-1)}\right)(C_\rho\sigma^{-d})^{r+1}\E[|X^j-Y^j|^{p\lambda}]^{1/q'}\nonumber\\
    &=:C_1(k).\label{eq:C2}
\end{align}
In particular, this shows $F\in L^{r+1}$. Now, using again H\"{o}lder's and Markov's  inequalities,
\begin{align*}
    \sqrt{n}\,\E[F\mathds{1}_{F>n^{1/(2r)}}]&\leq \sqrt{n}\n{F}_{r+1}\P(F>n^{1/(2r)})^{r/(r+1)}\\
    &\leq \sqrt{n}\n{F}_{r+1}n^{-1/2}\E[|F|^{r+1}]^{r/(r+1)}\\
    &\leq \E[|F|^{r+1}]\leq C_1(k).
\end{align*}
   It follows from definition that $$C_1(k)\leq C\left( \nn{k}_2^{-\rho(r+1)}+\nn{k}_2^{-(p\lambda-r-1)}\right).$$ The proof is complete.
\end{proof}

\begin{proof}[Proof of Lemma \ref{lemma:numerator}]
 First, by the mean-value theorem, for any $a,b\in \R^d$ and $1\leq j\leq d$ we have
$$|f_a^j(x,y)-f_b^j(x,y)|\leq C_5|x^j-y^j|\nn{a-b}_2\sup_{t\in[a,b]}\left(\frac{\nn{t-x}_2}{\sigma}+1\right)^{-(\rho+1)},$$where $C_5=\rho C_\rho \sigma^{-(d+1)}$. 
By Markov's and H\"{o}lder's inequalities, with $(p,q)$ denoting a conjugate pair, for  $\nn{k}_2>10$, $a,b\in I_k$, and $\lambda >0$,
\begin{align*}
   &\hspace{0.5cm}\P[|f_a^j(X,Y)-f_b^j(X,Y)|>s]\\
   &\leq \P\left[C_5\nn{a-b}_2|X^j-Y^j|\sup_{t\in I_k}(\nn{t-X}_2+1)^{-(\rho+1)}>s\right]\\
    &\leq C_5^\lambda \nn{a-b}_2^\lambda  s^{-\lambda }\E\left[|X^j-Y^j|^\lambda \sup_{t\in I_k}\left(\frac{\nn{t-X}_2}{\sigma}+1\right)^{-(\rho+1)\lambda}\right]\\
    &\leq C_5^\lambda\nn{a-b}_2^\lambda  s^{-\lambda }\E[|X^j-Y^j|^{p\lambda }]^{1/p}\E\left[\sup_{t\in I_k}\left(\frac{\nn{t-X}_2}{\sigma}+1\right)^{-(\rho+1)q\lambda}\right]^{1/q}\\
    &\leq C_5^\lambda\nn{a-b}_2^\lambda  s^{-\lambda }\n{X^j-Y^j}_{p\lambda}^\lambda\left((\frac{\nn{k}_2}{2\sigma})^{-(\rho+1)\lambda }+\E[\nn{X}_2^{p\lambda}]^{1/q}(\frac{\nn{k}_2}{2})^{-p\lambda /q}\right),
\end{align*}where the last step follows in a similar way as the derivation of \eqref{eq:h3}, given that $(X,Y)\in L^{p\lambda }$. Recall that $Q_f$ is the quantile of $|f(X_1)|$. 
 As a consequence, for $a,b\in I_k$,
\begin{align*}
    &\hspace{0.5cm}Q_{f_a^j-f_b^j}(u)\\
    &\leq  C_5\nn{a-b}_2\n{X^j-Y^j}_{p\lambda}u^{-1/\lambda }\left((\frac{\nn{k}_2}{2\sigma})^{-(\rho+1)\lambda }+\E[\nn{X}_2^{p\lambda}]^{1/q}(\frac{\nn{k}_2}{2})^{-p\lambda /q}\right)^{1/\lambda}.
\end{align*}
In particular, for $A,B\in I_k$ with $A<B$,
\begin{align}\label{eq:ab}
    &\sup_{a,b\in [A,B]}Q_{f_a^j-f_b^j}(u)\nonumber\\
    &\hspace{0.5cm}\leq C_5\nn{a-b}_2\n{X^j-Y^j}_{p\lambda}u^{-1/\lambda }\left((\frac{\nn{k}_2}{2\sigma})^{-(\rho+1)\lambda }+\E[\nn{X}_2^{p\lambda}]^{1/q}(\frac{\nn{k}_2}{2})^{-p\lambda /q}\right)^{1/\lambda}.
\end{align}
Moreover, for $a\in I_k$,
\begin{align*}
    &\hspace{0.5cm}\P[|f_a^j(X,Y)|>s]\\
    &\leq\P\left[C_\rho\rho^{-d}|X^j-Y^j|\sup_{t\in I_k}\left(\frac{\nn{t-X}_2}{\sigma}+1\right)^{-\rho}>s\right]\\
    &\leq  s^{-\lambda }\E\left[(C_\rho\rho^{-d})^\lambda|X^j-Y^j|^\lambda \sup_{t\in I_k}\left(\frac{\nn{t-X}_2}{\sigma}+1\right)^{-\rho\lambda }\right]\\
    &\leq (C_\rho\rho^{-d})^\lambda \E[|X^j-Y^j|^{p\lambda}]^{1/p}s^{-\lambda }\E\left[\sup_{t\in I_k}\left(\frac{\nn{t-X}_2}{\sigma}+1\right)^{-\rho q\lambda }\right]^{1/q}\\
    &\leq (C_\rho\rho^{-d})^\lambda \E[|X^j-Y^j|^{p\lambda}]^{1/p}s^{-\lambda }\left((\frac{\nn{k}_2}{2\sigma})^{-\rho\lambda }+\E[\nn{X}_2^{p\lambda}]^{1/q}(\frac{\nn{k}_2}{2})^{-p\lambda /q}\right).
\end{align*}This implies (recalling that $\F_k=\{f_a\}_{a\in I_k}\cup\{0\}$)
\begin{align*}
    &\hspace{0.5cm}\sup_{f,g\in \F_k}Q_{f^j-g^j}(u)\\
    &\leq LC_\rho\rho^{-d}\n{X^j-Y^j}_{p\lambda}\left((\frac{\nn{k}_2}{2\sigma})^{-\rho\lambda }+\E[\nn{X}_2^{p\lambda}]^{1/q}(\frac{\nn{k}_2}{2})^{-p\lambda /q}\right)^{1/\lambda}u^{-1/\lambda }.
\end{align*}
 Since $A_{\alpha,\lambda }<\infty$, we have by definition,
\begin{align}
    \delta_k^j&:=\sup_{f,g\in\F_k}\n{f^j-g^j}_{2,X}\nonumber\\
    &\leq LC_\rho\rho^{-d}\n{X^j-Y^j}_{p\lambda}\left((\frac{\nn{k}_2}{2\sigma})^{-\rho\lambda }+\E[\nn{X}_2^{p\lambda}]^{1/q}(\frac{\nn{k}_2}{2})^{-p\lambda /q}\right)^{1/\lambda}A_{\alpha,\lambda }=C_2(k).\label{eq:deltajk}
\end{align}
In addition, \eqref{eq:ab} implies that for $A,B\in I_k$ with $A<B$,
\begin{align*}
    &\sup_{a,b\in[A,B]}\n{f_a^j-f_b^j}_{2,X}\\
    &\leq C_5\nn{a-b}_2\n{X^j-Y^j}_{p\lambda}A_{\alpha,\lambda }\left((\frac{\nn{k}_2}{2\sigma})^{-(\rho+1)\lambda }+\E[\nn{X}_2^{p\lambda}]^{1/q}(\frac{\nn{k}_2}{2})^{-p\lambda /q}\right)^{1/\lambda}.
\end{align*}
Therefore,
\begin{align}
    \label{eq:Nk}&N(x,\n{\cdot}_{2,X},\F_k)\nonumber\\
    &\le L\left(x^{-1}\sqrt{d}C_5 \n{X^j-Y^j}_{p\lambda}A_{\alpha,\lambda }\left((\frac{\nn{k}_2}{2\sigma})^{-(\rho+1)\lambda }+\E[\nn{X}_2^{p\lambda}]^{1/q}(\frac{\nn{k}_2}{2})^{-p\lambda /q}\right)^{1/\lambda}\right)^d\nonumber\\
    &=:x^{-d}C_3(k).
\end{align}
 In particular, $\int_0^1 N(x,\n{\cdot}_{2,X},\F_k)^{v/s}\,\d x<\infty$ for $s>dv$.

 Next we compute $q_0${$(\delta_k^j)$} and $\ee(\delta_k^j)$. First, by definition of $q_0${$(\delta_k^j)$}, $N(2^{-q_0},\n{\cdot}_{2,X},\F_k)\leq (\delta_k^j)^{-1/2}\leq N(2^{-q_0-1},\n{\cdot}_{2,X},\F_k)$.  Using \eqref{eq:Nk}, we have
$$(\delta_k^j)^{-1/2}\leq {N(2^{-q_0-1},\n{\cdot}_{2,X},\F_k)}\leq 2^{d(q_0+1)}C_3(k).$$
This in turn yields
\begin{align}
    2^{q_0+1}\geq (C_3(k)(\delta_k^j)^{1/2})^{-1/d}\label{eq:q_0}
\end{align}
and thus
$$2^{-q_0} \le C_3(k)^{1/d} (\delta_k^j)^{1/(2d)}.$$
Using \eqref{eq:deltajk}, \eqref{eq:Nk}, and \eqref{eq:q_0}, we obtain
\begin{align}
    \ee(\delta_k^j)&=\int_0^{2^{-q_0}} N(x,\n{\cdot}_{2,X},\F_k)^{v/s}\,\d x\\
    &\hspace{-0.1cm}\stackrel{\eqref{eq:Nk}}{\leq }\int_0^{2^{-q_0}}(x^{-d}C_3(k)))^{v/s}\d x\nonumber\\
    &= C_3(k)^{v/s}(1-\frac{dv}{s})^{-1}2^{-q_0(1-(dv)/s)}\nonumber\\
    &\hspace{-0.1cm}\stackrel{\eqref{eq:q_0}}{\leq } (1-\frac{dv}{s})^{-1}C_3(k)^{1/d}(\delta_k^j)^{\frac{1}{2d}(1-\frac{dv}{s})}\nonumber\\
    &\hspace{-0.1cm}\stackrel{\eqref{eq:deltajk}}{\leq }(1-\frac{dv}{s})^{-1}C_3(k)^{1/d}C_2(k)^{\frac{1}{2d}-\frac{v}{2s}}=:C_4(k).\label{eq:ed}
\end{align}
In addition, we compute using \eqref{eq:Nk} and \eqref{eq:q_0} that 
\begin{align}
\sum_{\ell=q_0+1}^{q+1}N(2^{-\ell},\n{\cdot}_{2,X},\F)^{1/s}2^{-\ell}&\hspace{-0.1cm}\stackrel{\eqref{eq:Nk}}{\leq } \sum_{\ell=q_0+1}^{q+1}C_3(k)^{1/s}2^{-\ell+\ell d/s}\nonumber\\
&\hspace{-0.1cm}\stackrel{\eqref{eq:q_0}}{\leq } (1-2^{d/s-1})^{-1}C_3(k)^{1/s}(C_3(k)(\delta_k^j)^{1/2})^{1/d-1/s}\nonumber\\
&\hspace{-0.1cm}\stackrel{\eqref{eq:deltajk}}{\leq }(1-2^{d/s-1})^{-1}C_3(k)^{1/d}C_2(k)^{(1/d-1/s)/2}.\label{eq:combine}
\end{align}
Combining \eqref{eq:deltajk}, \eqref{eq:ed}, \eqref{eq:combine}, Lemma \ref{lemma:r}, and choosing $\delta=\delta_k^j$ in Lemma \ref{lemma:alpharate} yields
\begin{align*}
    &\hspace{0.5cm}\E\left[\sup_{f,g\in\F_k}\left|\frac{1}{\sqrt{n}}\sum_{i=1}^n (f^j(X_i,Y_i)-g^j(X_i,Y_i))\right|\right]\\
    &\leq C(s,r)C_4(k)^{1/(1+\beta)}+C(s)C_2(k)^{1/s}+(1-2^{d/s-1})^{-1}C_3(k)^{1/d}C_2(k)^{(1/d-1/s)/2}+LC_1(k),
\end{align*}
The proof of \eqref{eq:mixingfinitesample} is then complete by noting that  
\begin{align*}
    \E\left[\sup_{x\in I_k}|\xi_n^j(x)|\right]&=\E\left[\sup_{f\in\F_k}\left|\frac{1}{\sqrt{n}}\sum_{i=1}^n (f^j(X_i,Y_i)-\E[f^j(X_i,Y_i)])\right|\right]\\
    &\leq \E\left[\sup_{f,g\in\F_k}\left|\frac{1}{\sqrt{n}}\sum_{i=1}^n (f^j(X_i,Y_i)-g^j(X_i,Y_i))\right|\right],
\end{align*}which follows since   $0\in\F_k$. That \eqref{eq:mixingC} is straightforward to verify by choosing the smallest (negative) powers of $\nn{k}_2$ in \eqref{eq:mixingfinitesample}, and noting that $p/q=p-1$.\end{proof}

\begin{proof}[Proof of Theorem \ref{thm:mixingconvergence}]
The proof mimics that of case (i) of Theorem \ref{thm:limitd}, so we only give a sketch here. Note first that Lemma \ref{lemma:convergeK} generalizes to the $\alpha$-mixing case using the main result of \citep{doukhan1995invariance}, given our assumption $A_{\alpha,\lambda}<\infty$. On the other hand, the analogue of Lemma \ref{lemma:limsup} (with $\gamma=1$) follows by replacing Lemma \ref{lemma:Esup when gamma=1} by \eqref{eq:mixingC} of Lemma \ref{lemma:numerator} along with \eqref{eq:mixingC2} of Remark \ref{remark:solution}.
 Therefore, 
Theorem \ref{thm:mixingconvergence} then follows from the analogues of Lemmas \ref{lemma:convergeK} and \ref{lemma:limsup} just as in Section \ref{sec:asymp}.
\end{proof}

\subsection{Generalized smoothing kernel revisited}

We will revisit the proof of our main results for both the finite-sample case and the asymptotic case for the general $f_\xi$. We begin with a walk through the methodological development of Proposition \ref{prop:general_xi} for the finite-sample case. We start with the following lemma:

\begin{lemma}\label{lem:finite_sample}
Under Assumption \ref{ass:1} there exists $C=C(\beta, \delta, d, r, D)$ such that  $$\log N_{[\,]}(\epsilon,\F^r_j, L_2(\mu))\le C \Big(\frac{1}{\epsilon}\Big)^{2-\delta}.$$
\end{lemma}

\begin{proof}
We want to apply Corollary 2.7.4 of  \citep{van1996}   with $\alpha= 2d/(2-\delta)$ for $\delta\in(0,1)$, $V= 2-\delta\ge (2d)/\beta$.

We first calculate derivatives: note that for $i_1,j=1,\dots, d$
\begin{align*}
\nabla_{y_{i_1}} f_{a}^r(x,y)_j &=  (1\vee |a|_2^r)f_\xi(a-x)\mathds{1}_{\{i_1=j\}};\\
\nabla_{x_{i_1}} f_{a}^r(x,y)_j &=  -(1\vee |a|_2^r) [f_\xi(a-x)\mathds{1}_{\{i_1=j\}} +(y_j-x_j)\nabla_{x_{i_1}} f_\xi(a-x)].
\end{align*}
In particular, since all the other derivatives are identically zero, we only need to consider
\begin{align}\label{eq:deriv}
\begin{split}
\nabla_{x_{i_2}\dots x_{i_{k+1}}}^k \nabla_{y_{i_1}} f_{a}^r(x,y)_j &= (1\vee |a|_2^r) (-1)^k \nabla_{x_{i_2}\dots x_{i_{k+1}}}^k f_\xi(a-x)\mathds{1}_{\{i_1=j\}};  \\
\nabla_{x_{i_1}\dots x_{i_{k+1}}}^{k+1} f_{a}^r(x,y)_j &= (-1)^{k+1} (1\vee |a|_2^r)\Big(\nabla_{x_{i_2}\dots x_{i_{k+1}}}^{k} f_\xi(a-x)\mathds{1}_{\{i_1=j\}} \\
&+(y_j-x_j)\nabla_{x_{i_1}\dots x_{i_{k+1}}}^{k+1} f_\xi(a-x)\\
&+\sum_{i_l=j} \nabla_{x_{i_1}\dots x_{i_{l-1}} x_{i_{l+1}} x_{i_{k+1}}} f_\xi(a-x)\Big)
\end{split}
\end{align}
for $k\ge 0$, $i_1, \dots, i_{k+1},j\in \{1, \dots, d\}$. 
Next, as $|a|_2^r\le 2^{r-1}(|a-x|_2^r+|x|_2^r)$ we conclude that
\begin{align*}
(1\vee |a|_2^r)\le 2^{r-1}(1\vee  |a-x|_2^r +1\vee |x|_2^r).
\end{align*}
Using the above we obtain for $j=1,\dots,d$, 
\begin{align*}
|f_a^r(x,y)_j| &\le (1\vee |a|_2^r)|y_j-x_j| |f_\xi(a-x)| \\
&\le 2^{r-1}(1\vee  |a-x|_2^r +1\vee |x|_2^r) |y_j-x_j| |f_\xi(a-x)|.
\end{align*}
Recall that by Assumption \ref{ass:1} we have 
\begin{align*}
  |a-x|_2^r   |f_\xi(a-x)| \vee |f_\xi(a-x)| \le C
\end{align*}
for all $x,a\in \R^d$. Thus
\begin{align*}
2^{r-1}(1\vee  |a-x|_2^r +1\vee |x|_2^r) |y_j-x_j| |f_\xi(a-x)|
&\le C (1+|x|_2^r)  |y_j-x_j|.
\end{align*}
In conclusion,
\begin{align}
|f_a^r(x,y)_j|\le  C (1+|(X,Y)|_2^{r+1}).\label{eq:fra}
\end{align}
Using \eqref{eq:deriv} and similar arguments we obtain
\begin{align}\label{eq:details}
\begin{split}
\|\nabla^k_{x,y} f_a^r(x,y)_j\|_\infty &\le C(1\vee  |a-x|_2^r +1\vee |x|_2^r) \max_{1\le l\le k} \|\nabla^l_{x,y} f_\xi(a-x)\|_\infty (1+|y_j-x_j|)   \\
&\le C (1+|x|_2^r) (1+|y_j-x_j|)\\
&\le C(1+|(X,Y)|_2^{r+1})
\end{split}
\end{align}
for all $k\ge 1$ and $j\in \{1,\dots, d\}.$  

We now compute the bracketing number. Consider a partition $\mathbb{R}^d\times \R^d=\cup_{j=1}^{\infty} I_j$ into cubes of side length one. Let $K_m$ be the collection of $j\ge 1$ such that $I_j$ is in the annulus $\{x\in \mathbb{R}^d\times \R^d: \ m-1\le |x|_\infty \le m\}$. Then $$|K_m|=(2m)^{2d}-(2(m-1))^{2d}\le 4d(2m)^{2d-1}.$$ Next, if $j\in K_m$, then by \eqref{eq:details} all derivatives above are bounded by $S_j:=C(1+m^{1+r})$. Furthermore, if $j\in K_m$, then by Markov's inequality for $m>1$
\begin{align*}
\mu(I_j)\le \mu( \, |(X,Y)|_\infty\ge m-1)\le  \frac{\mathbb{E}[|(X,Y)|_\infty^s]}{(m-1)^s}.
\end{align*}
Recall that by Corollary 2.7.4 of \citep{van1996} we have
\begin{align*}
&\hspace{0.5cm}\log N_{[\,]}(\epsilon,\F^r_j, L_2(\mu))\\
&\le   C(\alpha, V) \Big( \frac{1}{\epsilon} \Big)^V \Big( \sum_{j=1}^\infty S_j^{\frac{2V}{V+2}} \mu(I_j)^{\frac{V}{V+2}} \Big)^{\frac{V+2}{2}}\\
&\le   C(\alpha, \delta,d) \Big( \frac{1}{\epsilon} \Big)^{2-\delta}
\Big(1+ \sum_{m=2}^\infty  (2m)^{2d-1} (1+m^{1+r})^{\frac{4-2\delta}{4-\delta}} \Big(\frac{\mathbb{E}[|(X,Y)|_\infty^s]}{(m-1)^s}\Big)^{\frac{2-\delta}{4-\delta}} \Big)^{\frac{4-\delta}{2}}.
\end{align*}
Noting that $(4-2\delta)/(4-\delta)\le 1$ and $(2-\delta)/(4-\delta)\ge 1/2-\delta/4$, we conclude that the above is bounded by 
\begin{align*}
 C(\alpha, \delta,d) \Big( \frac{1}{\epsilon} \Big)^{2-\delta} \Big(1+ \sum_{j=1}^\infty m^{2d-1+1+r-s(1/2-\delta/4)}\Big)^{\frac{4-\delta}{2}}.
\end{align*}
It thus suffices to check when 
\begin{align*}
    \sum_{j=1}^\infty m^{2d-1+1+r-s(1/2-\delta/4)}<\infty;
\end{align*}
this is the case if $$s>\frac{1+2d+r}{1/2-\delta/4}>4(1+2d+r).$$ This concludes the proof in view of Assumption \ref{ass:1}.
\end{proof}

\begin{corollary}\label{cor:finite_sample}
For any $n\in \N$,
\begin{align*}
\sqrt{n} \,\E\Big[\sup_{f\in \mathcal{F}_j^r} |\E_{\P_n}[f]-\E_{\P_0}[f]|\Big] &\le  C\int_0^{C\big(1+\E\big[|(X,Y)|_2^{2(r+1)}\big]\big)} \sqrt{\log N_{[\,]} (\epsilon,\F^r_j, L_2(\mu))}\,\d\epsilon\\
&\le  C\int_0^{C\big(1+\E\big[|(X,Y)|_2^{2(r+1)}\big]\big)} \Big(\frac{1}{\epsilon}\Big)^{1-\delta/2}\,\d\epsilon,
\end{align*}
where $C=C(\beta, \delta, d,r,D).$
\end{corollary}

\begin{proof}
This follows from combining Lemma \ref{lem:finite_sample} and Corollary 19.35 of \citep{van1996}, noting in particular that by \eqref{eq:fra},
\begin{align*}
   \sup_{f\in \mathcal{F}_j^r} |f(X,Y)|^2 \le C (1+|(X,Y)|_2^{r+1})^2
\end{align*}
and thus 
\begin{align*}
    \E\Big[\sup_{f\in \mathcal{F}_j^r} |f(X,Y)|^2\Big] \le C\E\Big[(1+|(X,Y)|_2^{r+1})^2\Big]\le C\Big(1+\E\Big[|(X,Y)|_2^{2(r+1)}\Big]\Big)<\infty
\end{align*}
as $2(r+1)<s.$
\end{proof}

We are now in a position for the proof of Proposition \ref{prop:general_xi}.

\begin{proof}[Proof of Proposition \ref{prop:general_xi}]
Recall from \eqref{eq:En expression} that
\begin{align*}
\sqrt{n}\,\E_n^{*\xi}\left[|X-\E_n^{*\xi}[Y|X]|_2\right] = \int_{\R^d}  \big|\frac{1}{\sqrt{n}} \sum_{i=1}^n(Y_i-X_i)f_\xi\left(x-X_i\right)\big|_2\,\d x.
\end{align*}
Using Tonelli's theorem and Corollary \ref{cor:finite_sample} we conclude
\begin{align*}
&\hspace{0.5cm} \sqrt{n}\,\E\big[\E_n^{*\xi}\left[|X-\E_n^{*\xi}[Y|X]|_2\right]\big]\\
&=  \E\left[ \int_{\R^d}  \big|\frac{1}{\sqrt{n}} \sum_{i=1}^n(Y_i-X_i)f_\xi\left(x-X_i\right) (1\vee |x|_2^r)\big|_2 (1\vee |x|_2^r)^{-1}\,\d x\right]\\
 &= \int_{\R^d} \E\Big[\big|\frac{1}{\sqrt{n}} \sum_{i=1}^n(Y_i-X_i)f_\xi\left(x-X_i\right) (1\vee |x|_2^r)\big|_2\Big] (1\vee |x|_2^r)^{-1}\,\d x\\
 &\le L \sup_{1\le j\le d} \E\Big[\sqrt{n}\sup_{f\in \mathcal{F}_j^r} |\E_{\P_n}[f]-\E_{\P_0}[f]|\Big] \int_{\R^d} (1\vee |x|_2^r)^{-1}\,\d x\\
 &\le C \int_0^{C\big(1+\E\big[|(X,Y)|_2^{2(r+1)}\big]\big)} \Big(\frac{1}{\epsilon}\Big)^{1-\delta/2}\,\d\epsilon.
\end{align*}
This completes the proof.
\end{proof}

We now walk through the proof of our main asymptotic result, Proposition \ref{prop:general_xi2}, for general $f_\xi$. We will first prove a lemma that follows from a classical Donsker theorem.

\begin{lemma}\label{lem:asymptotic}
Under Assumption \ref{ass:1}, the class $\mathcal{F}^r_j$ is Donsker for $j\in \{1, \dots, d\}$.
\end{lemma}

\begin{proof}
We apply Example 2.10.25 of \citep{van2000asymptotic}. We use the same notation as in the proof of Lemma \ref{lem:finite_sample}. We need to check that for $\beta> d/2$ we have $$\sum_{j=1}^\infty S_j \mu(I_j)^{1/2}<\infty,$$ where we recall that $S_j=C(1+m^{1+r})$ and 
\begin{align*}
\mu(I_j)\le \mu( |(X,Y)|_\infty\ge m-1)\le  \frac{\mathbb{E}[|(X,Y)|_\infty^s]}{(m-1)^s}.
\end{align*}
Arguing as in the proof of Lemma \ref{lem:finite_sample} we need to have
\begin{align*}
\sum_{j=1}^\infty m^{2d-1+1+r-s/2}<\infty;
\end{align*}
this is satisfied for $s>2(2d+1+r)$. The claim follows.
\end{proof}

We are now ready to prove Proposition \ref{prop:general_xi2}.

\begin{proof}[Proof of Proposition \ref{prop:general_xi2}]
Recall that
\begin{align*}
\sqrt{n}\,\E_n^{*\xi}\left[|X-\E_n^{*\xi}[Y|X]|_2\right] = \int_{\R^d}  |\frac{1}{\sqrt{n}} \sum_{i=1}^n(Y_i-X_i)f_\xi\left(x-X_i\right)|_2\,\d x.
\end{align*}
Define the functional $I:C_b(\R^d)\to \R$ via
\begin{align*}
I(g)=\int_{\R^d} g(x) (1\vee |x|_2^r)^{-1}\,\d x.
\end{align*}
Taking two bounded continuous functions $f,g$ we have
\begin{align*}
|I(f)-I(g)|\le \|f-g\|_\infty \int_{\R^d} (1\vee |x|_2^r)^{-1}\d x\le L\|f-g\|_{\infty}
\end{align*}
by H\"older's inequality; in conclusion, $I$ is continuous in $L^\infty$-norm.
Writing 
\begin{align*}
 &\sqrt{n}\,\E_n^{*\xi}\left[|X-\E_n^{*\xi}[Y|X]|\right] \\
 &\hspace{1cm}=   \int_{\R^d}  \big|\frac{1}{\sqrt{n}} \sum_{i=1}^n(Y_i-X_i)f_\xi\left(x-X_i\right) (1\vee |x|_2^r)\big|_2 (1\vee |x|_2^r)^{-1}\,\d x,
\end{align*}
the claim follows from the continuous mapping theorem and the fact that $\mathcal{F}^r_j$ is Donsker, see Lemma \ref{lem:asymptotic}. This concludes the proof.
\end{proof}

\section{Deferred algorithms}\label{appn}
\begin{algorithm}[H]
\caption{Simulate Asymptotic Distribution given $\{\rho, \sigma\}$}
\begin{flushleft}
\textbf{Input}: total number of grid points $n$ per dimension $d$, total number of simulations $N$, domain of integration $[x_{\min},x_{\max}]^d$, samples of martingale couplings $(X,Y)$, martingale projection parameters $\{\rho, \sigma\}$, smoothing kernel $f_{\sigma,\xi,\rho}$ as defined in \eqref{eq:fdensity}.
\newline
\textbf{Initialization}: initialize grid $g = \{g_1,\dots,g_{n^d}\}$ based on domain of integration $[x_{\min},x_{\max}]^d$ and number of grid points $n$.
\begin{algorithmic}
\State Given samples $(X,Y)$, grid $g$, and $f_{\sigma,\xi,\rho}$:
\For{index $i$ $1:n^{d}$}
    \For{index $j$ $1:n^{d}$}
        \State Generate each entry $(i,j)$ of the covariance matrix $M_{n \times n}$ of Gaussian random field $\{G_{x}\}$:  $$\E[G_{g_i}G^{T}_{g_j}] = \E[(Y-X)f_{\xi,\rho,\sigma}(g_i-X)f_{\xi,\rho,\sigma}(g_j-X)(Y-X)^T].$$
    \EndFor
\EndFor
\For{index $i$ $1:N$}
    \State Compute a sample $$ S_i = \sum_{x \in g} |G_x|_2\d x.$$
\EndFor

\State \Return $\{S_i\}_{i=1}^{N}$.
\end{algorithmic}
\end{flushleft}
\end{algorithm}

\begin{algorithm}[H]
\caption{Martingale Pair Test}
\begin{flushleft}
\textbf{Input}: Simulated asymptotic distribution $\{S_i\}_{i=1}^{N}$, $\rho$, $n$ testing samples $\{(X_i,Y_i)\}_{i=1}^n$, lower bound $l$ and upper bound $u$ for integration, significance level $\alpha$, smoothing kernel $f_{\sigma,\xi,\rho}$ as defined in \eqref{eq:fdensity}.
\newline
\textbf{Initialization}: Compute critical value $c_{\alpha}$ by taking the $(1-\alpha)$-th percentile of $\{S_i\}_{i=1}^{N}$.
\begin{algorithmic}
\State Compute martingale projection error: $$M_e = \frac{1}{n} \int_l^u|\sum_{i=1}^n(Y_i-X_i)f_{\sigma,\xi,\rho}(x-X_i)|_2\d x.$$
\If{$\sqrt{n}M_e \leq c_\alpha$}
    \State \Return $M_e,$ \texttt{True}
\Else
    \State \Return $M_e,$ \texttt{False}
\EndIf
\end{algorithmic}
\end{flushleft}
\label{algo:mtgl_pair_test}
\end{algorithm}

\begin{algorithm}[H]
\caption{Monte Carlo-based option pricing}\label{alg:cap}
\begin{flushleft}
\textbf{Input}: calibrated stock trajectories $\{(S_t)_{t \in \{0,\dots,T\}}^i\}_{i = 1}^N$, number of inner MC trials $n$, true Heston model parameters $\theta^*=\{x_0, r, V_0, \kappa, \mu, \eta, \rho\}$, a list of call options strike prices $\textbf{K}$, a list of sorted maturities $\textbf{T} = \{T_{\text{min}},\dots, T_{\text{max}}\}$, number of time steps $N_{\text{steps}}$.
\newline 
\textbf{Initialization}: time increment $dt = 1/N_{\text{steps}}$
\begin{algorithmic}
\For{stock path $i$ 1: $N$} 
    \For {time $t$ 1: $T_{\text{max}}$}
        \State Generate $n$ Heston model-based stock trajectories $\{(S^{\theta^*}_{u})_{u\in \{0,\dots,T_{\text{max}}-j\}}^p\}_{p = 1}^n$ originating from $S_t$ at time $t$.
        \For{maturity $T \in \textbf{T}$}
            \If{$t<T$}
                \For{strike $K \in \textbf{K}$}
                    \State Price the corresponding vanilla option with maturity $T$ and strike $K$ at time $t$ by taking the mean over $n$ Monte Carlo trials: $$
\mathfrak{p}_t(K,T,S_j)=e^{-r (T-t)} \frac{1}{n}\sum_{p = 1}^{n} ((S_{T}^{\theta^*})^p-K)_{+}$$
                \EndFor
            \EndIf
        \EndFor
    \EndFor
\EndFor
\For{maturity $T \in \textbf{T}$}
    \For{strike $K \in \textbf{K}$}
        \State Price the corresponding vanilla option using $$
\mathfrak{p'}_t(K,T)=e^{-r (T-t)} \frac{1}{n}\sum_{p = 1}^{n} ((S_{T})^i-K)_{+}$$
    \EndFor
\EndFor
\State \Return $(\mathfrak{p},\mathfrak{p'})$
\end{algorithmic}
\end{flushleft}
\end{algorithm}

\begin{algorithm}[H]
\caption{New Calibration to Market European Vanilla Option Prices}
\begin{flushleft}
\textbf{Input}: Discounted option payoffs $\{{e^{-r(T^{j}-t)}(S_{T^j}-K^{j})_{+}}\}_{j=1}^M$, their corresponding market prices $\{\mathfrak{p}_t(K^j,T^j)\}_{j=1}^M$, total number of training epoch $N_{\text{epoch}}$, total number stock trajectories generated $N$, total number of time steps $N_{\text{steps}}$, smoothing kernel $f_{\sigma,\xi,\rho}$ as defined in \eqref{eq:fdensity}.
\newline
\textbf{Initialization}: $\theta =\{x_0, r, V_0, \kappa, \mu, \eta, \rho\}$, time increment $dt = 1/N_{\text{steps}}$.
\begin{algorithmic}
\For{epoch 1: $N_{\text{epochs}}$}
    \State Generate $N$ sample paths $(x^{\pi,\theta,i}_{t_n})^{N_{\text{steps}}}_{n=0}$ for $i = 1,\dots,N$ using tamed Euler scheme on Heston model.
    \State \textbf{During one epoch:} use Adam to update $\theta$, where
    \begin{align*}
    \theta^* = \argmin_{\theta\in \Theta} \sum_{j=1}^M  \sum_{i=1}^N \int \Big( (\mathfrak{p}({K^j,T^j},S^\theta_{t_i})- &{e^{-r(T^j-t)}(S^\theta_{T^j}-K^j)_{+}}) \\
    &\cdot f_{\sigma,\xi,\rho}(x-{{e^{-r(T^j-t)}(S^\theta_{T^j}-K^j)_{+}}})\Big)^2\, \d x.
    \end{align*}
\EndFor
\State \Return $\theta$ for all payoffs.
\end{algorithmic}
\end{flushleft}
\end{algorithm}

\end{appendix}

\section*{Acknowledgement}
We are grateful to two anonymous referees for their valuable feedback that has significantly improved this work. 
The material in this paper is based upon work supported by the Air Force Office of Scientific Research under award number FA9550-20-1-0397. Additional support is gratefully acknowledged from NSF 1915967, 2118199, 2229012, 2312204, 2345556.

\bibliography{main}
\bibliographystyle{imsart-number}

\end{document}